\documentclass[12pt]{amsart}
\usepackage{t1enc}
\usepackage[latin1]{inputenc}
\usepackage[english]{babel}
\usepackage{amssymb}
\usepackage{amsmath}
\usepackage{graphics}
\usepackage{hyperref}
\usepackage{enumerate}
\usepackage{caption}
\usepackage{lscape}
\usepackage{tikz}

\usepackage{amsthm, amssymb}
\usepackage{amsfonts}
\usepackage{epsfig,multicol}   

\usepackage{mathrsfs}
\numberwithin{equation}{section}
\theoremstyle{theorem}
\newtheorem{theorem}{Theorem}[section]

\newtheorem{lemma}[theorem]{Lemma}
\newtheorem{corollary}[theorem]{Corollary}
\theoremstyle{definition}
\newtheorem{definition}[theorem]{Definition}
\newtheorem{example}[theorem]{Example}
\newtheorem{remark}[theorem]{Remark}

\newcommand{\arxiv}[1]{\href{http://arxiv.org/abs/#1}{\tt arXiv:\nolinkurl{#1}}}
\newcommand{\doi}[1]{\href{http://dx.doi.org/#1}{{\tt DOI:#1}}}

\newcommand{\googlebooks}[1]{(preview at \href{http://books.google.com/books?id=#1}{google books})}

\def\<{\langle}
\def\>{\rangle}

\begin{document}

\def\hpic #1 #2 {\mbox{$\begin{array}[c]{l} \epsfig{file=#1,height=#2}
\end{array}$}}
 
\def\vpic #1 #2 {\mbox{$\begin{array}[c]{l} \epsfig{file=#1,width=#2}
\end{array}$}}

\title{The Asaeda-Haagerup fusion categories}
\author{Pinhas~Grossman
\and
Masaki Izumi
\and
Noah Snyder
}
%%\address{
%%}%
%%\email{p.grossman@unsw.edu.au}
%
%\author{Masaki Izumi}
%%\address{
%%}%
%%\email{izumi@math.kyoto-u.ac.jp}
%%
%\author{Noah~Snyder}
%\address{
%}%
%\email{nsnyder@math.columbia.edu}
%

\maketitle

\begin{abstract}
The classification of subfactors of small index revealed several new subfactors.  The first subfactor above index $4$, the Haagerup subfactor, is increasingly well understood and appears to lie in a (discrete) infinite family of subfactors where the $\mathbb{Z}/3\mathbb{Z}$ symmetry is replaced by other finite abelian groups.  The goal of this paper is to give a similarly good description of the Asaeda--Haagerup subfactor which emerged from our study of its Brauer--Picard groupoid.  More specifically, we construct a new subfactor $\mathcal{S}$ which is a $\mathbb{Z}/4\mathbb{Z} \times \mathbb{Z}/2\mathbb{Z}$ analogue of the Haagerup subfactor and we show that the even parts of the Asaeda--Haagerup subfactor are higher Morita equivalent to an orbifold quotient of $\mathcal{S}$.  This gives a new construction of the Asaeda--Haagerup subfactor which is much more symmetric and easier to work with than the original construction.  As a consequence, we can settle many open questions about the Asaeda--Haagerup subfactor: calculating its Drinfel'd center, classifying all extensions of the Asaeda--Haagerup fusion categories, finding the full higher Morita equivalence class of the Asaeda--Haagerup fusion categories, and finding intermediate subfactor lattices for subfactors coming from the Asaeda--Haagerup categories.  The details of the applications will be given in subsequent papers.
\end{abstract}

\section{Introduction}

A major motivation for classifying mathematical objects is to produce interesting new examples which would not have been found without an extensive search.  For example, it was the classification of simple Lie algebras which uncovered $E_8$ and the classification of finite simple groups which revealed the Monster group.  Similarly, a major motivation in the classification of small index subfactors is to find new interesting examples.

To any finite index finite depth subfactor $N \subset M$ one can assign a rich algebraic invariant called its standard invariant.  One way to describe the standard invariant is that it consists of a pair of a unitary fusion category $\mathcal{C}$ (consisting of certain $N$-$N$ bimodules and called the principal even part ) and an algebra object $A \in \mathcal{C}$ (which is the bimodule $M$).  Furthermore, for finite index finite depth subfactors of the hyperfinite $II_1$ factor, the standard invariant is a complete invariant \cite{ MR1055708}.  Almost all known unitary fusion categories can be constructed from finite groups or from quantum groups at roots of unity.  Indeed, all subfactors of index less than $4$ have standard invariants coming from quantum $SU(2)$ via conformal inclusions, and all finite depth subfactors of index equal to $4$ have standard invariants which come from finite subgroups of $SU(2)$.   It is natural to wonder whether the subfactors of index between $4$ and $5$ come from similar constructions.  

There are 5 pairs of finite index finite depth subfactor standard invariants of index between $4$ and $5$: 
\begin{itemize}
\item the Haagerup subfactor (and its dual),
\item  the Asaeda--Haagerup subfactor (and its dual),
\item the extended Haagerup subfactor (and its dual), 
\item  the Goodman-de la Harpe-Jones 3311 subfactor (and its dual),
\item and the self-dual Izumi 2221 subfactor (and its complex conjugate).
\end{itemize}  
\noindent The last two of these pairs come from quantum groups in an appropriate sense.  %The 3311 subfactor is related to the $E_6$ subfactor via the GHJ construction, and thus comes from a conformal inclusion for quantum $SU(2)$.  As observed by Xu and ???, the 2221 subfactor comes from a conformal inclusion for the loop group of the exceptional Lie group $G_2$ (at level ???).  
The remaining three seem to not be related to finite groups or quantum groups via well-understood constructions.  Thus they are very important examples for further study since they could be ``exotic'' or ``exceptional.''

Of these three examples, the Haagerup subfactor is by far the best understood.  Three different constructions of this subfactor standard invariant have been given, each of a different flavor \cite{MR1686551,MR1832764,0902.1294}.  One of these constructions allows for practical calculations, for example of the $S$ and $T$ matrices of the Drinfel'd center of the even part \cite{MR1832764}.  Furthermore, the principal even part of the Haagerup subfactor has a collection of invertible objects forming the group $\mathbb{Z}/3\mathbb{Z}$, and all non-invertible objects in this even part are in a single orbit of the action of $\mathbb{Z}/3\mathbb{Z}$.  On the one hand, this suggests that there may be a sense in which the Haagerup subfactor can be constructed from the group $\mathbb{Z}/3\mathbb{Z}$ (or a closely related group like $S_3$, see \cite{MR2837122}), and on the other hand it suggests that the Haagerup subfactor could be generalized by replacing $\mathbb{Z}/3\mathbb{Z}$ by other finite abelian groups.   We call such a subfactor a $3^A$ subfactor, since its principal graph is a star graph with $\#A$ arms where each arm has $3$ edges. The second-named author determined a numerical invariant for such subfactors given by solutions to certain polynomial equations, and by solving these equations he constructed a $3^{\mathbb{Z}/5\mathbb{Z}}$ subfactor \cite{MR1832764}. Evans--Gannon constructed $3^G$ subfactors for several other cyclic groups of odd order by solving the polynomial equations using symmetries of the associated modular data \cite{MR2837122}.  Although there are still only finitely many examples known, it appears that that the Haagerup subfactor may not be exceptional after all.  

The main goal of this paper is to establish a similar story for the previously mysterious Asaeda--Haagerup subfactor.  The basic philosophy of our approach is that one should never study a single subfactor at a time, but instead should study all the subfactors which come from inclusions of algebra objects $A \subset B$ in a fixed unitary fusion category $\mathcal{C}$.  In other words, one fixes a finite collection of bimodules over a fixed factor $R$ and looks at all subfactors $N \subset M$ where each of $N$ and $M$ can be built as direct sums of these particular bimodules.  All of these subfactors fit together into a richer structure known as the maximal atlas or the Brauer--Picard groupoid \cite{MR2677836}.  This approach has two advantages.  First, the combinatorics of this whole collection is richer and more restrictive than the combinatorics of the individual subfactors.  For example, just by looking at fusion rules one can show that certain subfactors must exist which would be very difficult to construct directly, as illustrated in \cite{GSbp}.  Second, it may be that some of the other subfactors in the maximal atlas are simpler than others.  In this paper we we emphasize the second point of view, and we relate the Asaeda--Haagerup subfactor to a simpler subfactor also appearing in its maximal atlas.  That is, the complicated small index Asaeda--Haagerup subfactor is best thought of as a consequence of a simpler subfactor with larger index.  This gives a new construction of the the Asaeda--Haagerup subfactor, which still requires a difficult computation, but which is much more illuminating in terms of understanding this subfactor.

Here is a quick overview of our new construction of the Asaeda--Haagerup subfactor.  We first construct a specific $3^{\mathbb{Z}/4\mathbb{Z} \times \mathbb{Z}/2\mathbb{Z}}$ subfactor $\mathcal{S}$ by finding a solution to the appropriate polynomial equations, which were generalized to finite Abelian groups of possibly even order in \cite{IzumiNote}.  
%When the abelian group here has even order, one may need to use a slightly more general construction which allows for certain additional signs.  
%We construct $\mathcal{S}$ by providing an explicit solution to these more general equations.  
The subfactor $\mathcal{S}$ has index $5+\sqrt{17}$.  Then we take an orbifold (or de-equivariantization) quotient of $\mathcal{S}$ by adding an extra isomorphism between the objects corresponding to the trivial and non-trivial elements of $\mathbb{Z}/2\mathbb{Z}$.  This new subfactor $\mathcal{S}'$ has index $5+\sqrt{17}$, and we call its even part $\mathcal{AH}_4$.  We then explicitly calculate the fusion rules for the dual even part of this subfactor (i.e. the bimodules over the larger factor) and see that it contains an object $X$ with dimension $\frac{3+\sqrt{17}}{2}$ such that $X \otimes X \cong 1 \oplus X \oplus Y$ for some non-invertible simple object $Y$.  By a simple skein theory argument \cite[Thm 3.4]{GSbp}, the object $1\oplus X$ has a unique algebra structure and this algebra gives a new subfactor with index $\frac{5+\sqrt{17}}{2}$.  This subfactor must be the Asaeda--Haagerup subfactor $\mathcal{AH}$, since it is easy to see that there's at most one finite depth subfactor of index $\frac{5+\sqrt{17}}{2}$.  Altogether, the Asaeda--Haagerup subfactor comes as a natural consequence of the more symmetric and easier to understand subfactor $\mathcal{S}$.

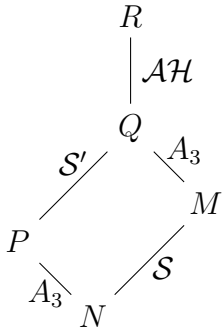
\begin{figure}
$$\begin{tikzpicture}
\node (N) at (0,0) {$N$};
\node (M) at (1.5,1.5) {$M$};
\node (P) at (-1,1) {$P$};
\node (Q) at (.5,2.5) {$Q$};
\node (R) at (.5,4) {$R$};
\path (N) edge node[below,pos=.7] {$\mathcal{S}$} (M);
\path (N) edge node[below,pos=.8] {$A_3$} (P);
\path (P) edge node[above, pos=.5] {$\mathcal{S}'$} (Q);
\path (M) edge node[above=.1cm, pos=0] {$A_3$}(Q);
\path (Q) edge node[right,pos=.5]{$\mathcal{AH}$} (R);
\end{tikzpicture}
$$
\caption{Our construction of the Asaeda--Haagerup subfactor can be summarized by the above inclusions of factors, where $\mathcal{S}$ denotes our new index $5+\sqrt{17}$ subfactor, $\mathcal{S}'$ is its orbifold quotient, $\mathcal{AH}$ is the Asaeda--Haagerup subfactor of index $\frac{5+\sqrt{17}}{2}$, and $A_3$ denotes index $2$ inclusions.}
\label{fig:outline}
\end{figure}

In particular, we see that $\mathcal{AH}_4$ lies in the same higher Morita equivalence class as the even parts of the Asaeda--Haagerup subfactor $\mathcal{AH}_1$ and $\mathcal{AH}_2$.  Many invariants of fusion categories, most notably the Drinfel'd center, are invariant under higher Morita equivalence.  Thus instead of doing a calculation in $\mathcal{AH}_1$ or $\mathcal{AH}_2$, one can instead do the calculation in $\mathcal{AH}_4$.  In particular, in a follow-up paper we will give an explicit description of the Drinfel'd center of the Asaeda--Haagerup fusion categories (a question which had remained open for 15 years) by following the method in \cite{MR1832764}.  Similarly, since Etingof-Nikshych-Ostrik's classification of $G$-extensions of fusion categories is Morita invariant, constructing extensions of $\mathcal{AH}_4$ yields extensions of the other fusion categories in its Morita class.  More specifically, in another follow-up paper we can determine the homotopy type of the Brauer--Picard $3$-groupoid of the Asaeda--Haagerup category, which splits as a product of Eilenberg-Maclane spaces $K(\mathbb{Z}/2\mathbb{Z} \times \mathbb{Z}/2\mathbb{Z}, 1) \times K(\mathbb{C}^\times, 3)$.  As a consequence, we get a plethora of interesting $\mathbb{Z}/2\mathbb{Z} \times \mathbb{Z}/2\mathbb{Z}$ extensions of Asaeda--Haagerup fusion categories generalizing the constructions in \cite{GJSext}, and we can use this to give a complete classification of all extensions of these fusion categories by any group.

It is natural to wonder why the Asaeda--Haagerup subfactor, which has only a $2$-fold symmetry, should be related to an $8$-fold symmetric $3^{\mathbb{Z}/4\mathbb{Z} \times \mathbb{Z}/2\mathbb{Z}}$ subfactor.  In fact, we discovered this connection by working in the other direction: starting with the Asaeda--Haagerup subfactor and trying to describe its entire higher Morita equivalence class.  An extensive combinatorial calculation suggested that $\mathcal{AH}_4$ was one of the only possible new tensor categories compatible with all the rich structure of the Asaeda--Haagerup subfactor.  If $\mathcal{AH}_4$ does exist, one can then determine that it must have an equivariantization which would give a $3^{\mathbb{Z}/4\mathbb{Z} \times \mathbb{Z}/2\mathbb{Z}}$ subfactor.  Please note that logically our construction does not require the extensive combinatorial calculations from \cite{GSbp}, but the motivation to consider $\mathcal{S}$ in the first place did come out of those calculations.   

In fact, constructing $\mathcal{AH}_4$ has allowed us to complete our description of the higher Morita equivalence class of the Asaeda--Haagerup subfactors.  Namely, there are exactly $6$ fusion categories, and between any two of them there are exactly $4$ Morita equivalences between them.  As shown in \cite{GSbp}, the group of Morita autoequivalences is the Klein $4$-group.  In fact, for $\mathcal{AH}_4$ all these autoequivalences are realized by outer automorphisms.  As a consequence of this complete description of the higher Morita equivalence class, we are able to answer many other questions about the subfactors related to the Asaeda--Haagerup subfactors.  For example, in a followup paper we will use this classification to find all lattices of intermediate subfactors related to Asaeda--Haagerup.

The outline of this paper is as follows.  We begin in Section \ref{sec:Prelim} with some background on fusion categories, algebra objects, the Brauer--Picard groupoid, and the Asaeda--Haagerup subfactor.  This section includes some expository material summarizing the key ideas of our project of understanding subfactors via the Brauer--Picard groupid.

In Section \ref{sec:BP} we further describe the Brauer--Picard groupoid of the Asaeda--Haagerup fusion categories following \cite{GSbp}.  In \cite{GSbp} we saw that there were at least three fusion categories in this Morita equivalence class, that there were exactly four bimodules between any two such fusion categories, and that the BP group was the Klein $4$-group.  Here we narrow things down further, showing that if $\mathcal{AH}_4$ exists then there are exactly six fusion categories in the equivalence class.  These calculations are combinatorial in nature and are quite similar to those in \cite{GSbp}.  However, we would like to remind the reader that our new construction of $AH$ is {\em motivated} by the results of Section \ref{sec:BP}, but strictly speaking does not {\em depend} on the results of Section \ref{sec:BP}.

Section \ref{sec:main} is the main heart of the paper.  We give a direct construction of $\mathcal{S}$ as endomorphisms of an algebra using results of \cite{IzumiNote}. We then show that it has a de-equivariantization $\mathcal{AH}_4$, and show that the Asaeda--Haagerup subfactor comes from an inclusion of algebras in $\mathcal{AH}_4$. 
% The initial construction uses similar techniques to \cite{???} to reduce the problem to solving some polynomial equations.  

In Section \ref{sec:apps}, we quickly sketch several applications of our main results.  The full details of these applications will appear in later papers.

%In Section \ref{sec:extension} we give an application of this construction.  The whole BP group of $\mathcal{AH}_4$ must be implemented by outer automorphisms.  In this section we explicitly construct these outer automorphisms and construct a larger fusion category containing $\mathcal{AH}_4$ where all of these outer automorphisms become inner.  This larger fusion category is a Klein $4$-group extension of $\mathcal{AH}_4$ in the sense of \cite{MR2677836}.  This implies that all of the Asaeda--Haagerup fusion categories have extensions by the Klein $4$-group and allows us to show that the $k$-invariant of the Brauer--Picard $3$-groupoid of AH vanishes, which determines its homotopy type.  \todo{Figure out what to do about this paragraph.}

\textbf{Acknowledgements:}
The authors would like to thank Scott Morrison, David Penneys, and Emily Peters for helpful conversations. Pinhas Grossman was partially supported by ARC grant DP140100732.  Masaki Izumi was supported in part by the Grant-in-Aid for Scientific Research (B) 22340032, JSPS. Noah Snyder was supported by DOD-DARPA grant HR0011-12-1-0009. The authors would also like to thank the American Institute of Mathematics for its hospitality during the workshop \textit{Classifying Fusion Categories}.

\section{Preliminaries}\label{sec:Prelim}

We begin with some background on fusion categories, subfactors, the Brauer--Picard groupoid, combinatorics of fusion rules, and the Asaeda--Haagerup subfactor.  We assume that the reader is either familiar with the theory of fusion categories or the theory of subfactors.  Thus the first subsection is aimed at readers who are familiar with subfactors but not with fusion categories, and the second subsection is aimed at readers familiar with fusion categories but not subfactors.

\subsection{Fusion categories, module categories, and bimodule categories}

\begin{definition} \cite{MR2183279}
A fusion category over an algebraically closed field $k $ is a semisimple $k$-linear rigid monoidal 
category with finitely many simple objects and finite-dimensional morphism spaces such that the identity object is simple. 
\end{definition}

In this paper $k$ will always be the field of complex numbers. An equivalence of fusion categories is a linear monoidal equivalence.

\begin{example}
Fix a finite collection of bifinite bimodules $M_i$ over a II${}_1$ factor $A$ which are closed under tensor product in the sense that $M_i \otimes_A M_j \cong \bigoplus_k M_k^{\oplus n_k}$.  Then there is a fusion category whose objects are the direct sums of the $M_i$, the morphisms are bimodule maps, and the tensor product is $\otimes_A$.
\end{example}

A fusion category can be thought of as a higher analogue of an algebra, where instead of multiplying elements you tensor objects.  Just as modules and bimodules play a crucial role in understanding algebras, module categories and bimodule categories play a similar role in understanding fusion categories.

\begin{definition}
A (left) module category ${}_{\mathcal{C}}{\mathcal{M}} $ over a fusion category $\mathcal{C} $ is a
 category $\mathcal{M} $ along with a biexact bifunctor from $\mathcal{C} \times \mathcal{M}$ to $ \mathcal{M}$, along with a collection of unit and associativity isomorphisms satisfying certain coherence relations; see \cite{MR1976459} for details. In this paper we will further assume that module categories over fusion categories are semisimple.
\end{definition}

\begin{example}
If $\mathcal{C}$ comes from a collection of bimodules over $A$ as in the previous example, then one can build a left module category in the same way by finding a finite collection of simple $A$-$B$ bimodules for a factor $B$ which are closed under left multiplication by the $A$-$A$ bimodules.
\end{example}

Similarly, one may define right module categories over fusion categories and bimodules categories over pairs of fusion categories. There is a natural notion of equivalence for module or bimodule categories. 
\begin{definition}
A bimodule category over a pair of fusion categories   ${}_{\mathcal{C}}{\mathcal{M}} {}_{\mathcal{D}}$ is invertible if ${}_{\mathcal{C}}{\mathcal{M}} {}_{\mathcal{D}}   \boxtimes_{\mathcal{D}}  {}_{\mathcal{D}}{\mathcal{M}}^{op} {}_{\mathcal{C}} \cong {}_{\mathcal{C} } {\mathcal{C} } {}_{\mathcal{C} }$ and ${}_{\mathcal{D}}{\mathcal{M}^{op}} {}_{\mathcal{C}}   \boxtimes_{\mathcal{C}}  {}_{\mathcal{C}}{\mathcal{M}} {}_{\mathcal{D}} \cong {}_{\mathcal{D} } {\mathcal{D} } {}_{\mathcal{D} }$, where $\mathcal{M}^{op} $ is the opposite bimodule category and $\boxtimes_{\mathcal{C}} $ and $ \boxtimes_{\mathcal{D}}$ are the relative tensor products of bimodule categories; see \cite{MR2677836} for details.

Two fusion categories are Morita equivalent if there is an invertible bimodule category between them; an invertible bimodule category is called a Morita equivalence.
\end{definition}

\begin{example}
If $A \subset B$ is a finite depth subfactor then the principal even part (certain $A$-$A$ bimodules) and the dual even part (certain $B$-$B$ bimodules) are Morita equivalent with the Morita equivalence given by a collection of $A$-$B$ bimodules.
\end{example}

\begin{example} \label{ex:standardmorita}
If $G$ is a finite group, let $\mathrm{Vec}(G)$ be the category of $G$-graded vector spaces and $\mathrm{Rep}(G)$ be the category of finite dimensional representations of $G$.  Then $\mathrm{Vec}$ gives a Morita equivalence between $\mathrm{Vec}(G)$ and $\mathrm{Rep}(G)$ where the individual actions are given by forgetting and tensoring, but the associator for simple objects is $V = (1_g \otimes 1) \otimes V \rightarrow 1_g \otimes (1 \otimes V)  = V$, with the map given by the action of $g$ on V.  

This Morita equivalence can be understood from the subfactor point of view by looking at the crossed product subfactor $N \subset M = N \rtimes G$ and thinking of $\mathrm{Vec}(G)$ as $N$-$N$ bimodules, $\mathrm{Vec}$ as $N$-$M$ bimodules, and $\mathrm{Rep}(G)$ as $M$-$M$ bimodules.
\end{example}

Any simple module category over a fusion category ${}_{\mathcal{C}} \mathcal{M} $ can be given the structure of an invertible
bimodule category ${}_{\mathcal{C}} \mathcal{M} {}_{\mathcal{D}} $, where $\mathcal{D} =({}_{\mathcal{C}} \mathcal{M} )^*$ is the dual category of module endofunctors of ${}_{\mathcal{C}} \mathcal{M} $. Conversely, for any invertible bimodule category ${}_{\mathcal{C}} \mathcal{M} {}_{\mathcal{D}} $ we have $\mathcal{D}  \cong ({}_{\mathcal{C}} \mathcal{M} )^*$.

%; the bimodule category is determined up to equivalence by ${}_{\mathcal{C}} \mathcal{M} $, up to outer autoequivalences of $ ({}_{\mathcal{C}} \mathcal{M} )^*$.

Module categories can also be characterized in terms of algebras.  This can be thought of as an algebraic substitute for subfactor theory where the role of factors is played by algebra objects.

\begin{definition}
An algebra in a fusion category is an object $A$ together with a unit map $I \rightarrow A $
and a multiplication map $A \otimes A \rightarrow A$  satisfying the usual relations; see \cite{MR1976459}. 
\end{definition}

\begin{example}
If $N \subset M$ is a finite index finite depth subfactor, and $\mathcal{C}$ is the category of $N$-$N$ bimodules generated by $M$, then $M$ itself is an algebra object in $\mathcal{C}$.  In this case, saying that $M$ is an algebra object means both that it is an algebra and that the multiplication map $M \otimes M \rightarrow M$ is a map of $N$-$N$ bimodules.
\end{example}

In a similar way one can define modules over algebras in fusion categories. If $A$ is an  algebra
in $\mathcal{C} $, then the category $A $-mod of left  $A$-modules in $ \mathcal{C}$ is a right module category over $\mathcal{C} $ (although not necessarily semisimple), and similarly the category of right $A$-modules is a left module category. 

An algebra $A$ is called simple if its module category $A$-mod is semisimple and indecomposable, and is called a division algebra if in addition $A$ is simple as an $A$-module \cite{GSbp}.   If $A$ is a simple algebra in a fusion category, then the category $A$-mod-$A$ of $A-A$ bimodules is the dual category of $A$-mod.  In fact every indecomposable module category arises this way. One can define the internal hom bifunctor $\underline{\text{Hom}}(M_1,M_2) $ from ${}_{\mathcal{C}} \mathcal{M}  \times {}_{\mathcal{C}} \mathcal{M}  $ to $\mathcal{C} $. The internal hom satsfies $\text{Hom}(\underline{\text{Hom}} (M_1,M_2),X)\cong \text{Hom}(X \otimes M_1,M_2)$ for all $M_1,M_2  \in  {}_{\mathcal{C}} \mathcal{M}$ and
$X \in   {}_{\mathcal{C}}$, and the internal end $\underline{\text{End}} (M)=\underline{\text{Hom}}(M,M) $ is an algebra for every $M \in {}_{\mathcal{C}} \mathcal{M} $.

\begin{theorem}\cite{MR1976459}
Let $M \in {}_{\mathcal{C}} \mathcal{M}  $ be a simple object is a simple module category over a fusion category. Then $ {}_{\mathcal{C}} \mathcal{M} $ is equivalent to the category of modules
over $ \underline{\text{End}}(M)$ in $\mathcal{C} $.
\end{theorem}

\begin{example}
If $M$ comes from an $A$-$B$ bimodule for two factors, then $\underline{\text{End}} (M)= M \otimes_B \bar{M}$ where $\bar{M}$ is the contragradient $B$-$A$ bimodule.
\end{example}

When $M$ is a simple object, the algebra $ \underline{\text{End}}(M)$ is a division algebra \cite{GSbp}. 

\begin{theorem}\cite{GSbp}
The algebras $\underline{\text{End}}(X)$ and $\underline{\text{End}}(Y)$ for two objects $X$ and $Y$ in ${}_\mathcal{C} \mathcal{M}$ are isomorphic if and only if $X \cong Y g$ for some invertible object $g$ in the dual category $({}_{\mathcal{C}} \mathcal{M} )^*$.
\end{theorem}

\subsection{Subfactors}
The theory of standard invariants of finite-index subfactors has been developed both for Type II$_1$ and for properly infinite factors. The classes of standard invariants which arise in the two settings are the same, so the choice of setting is in some sense a matter of taste; however certain types of calculations or constructions may be easier in one setting or the other. In this paper we use infinite factors in order to utilize a construction of subfactors from endomorphisms of Cuntz algebras introduced by the second-named author in \cite{MR1228532}, which uses Type III factors.

There are two standard ways of building tensor categories related to an algebra $A$.  The first is familiar algebraically, namely one considers a collection of bimodules over $A$ which is closed under tensor product $\otimes_A$.  When $A$ is a von Neumann algebra, one needs to be a bit careful about whether one is considering Hilbert bimodules with the Connes fusion product or algebraic bimodules with the algebraic tensor product.  Happily, in the finite index setting for II$_1$ factors, either choice leads to the same tensor category \cite{MR2501843}.

The second construction is called the category of sectors and is less familiar to algebraists.  Nonetheless it has a nice categorical description which we give here.  If $\mathcal{C}$ is a linear category then there is a monoidal category whose objects are the tensor automorphisms of $\mathcal{C}$ and whose morphisms are tensor natural transformations.  The simplest kind of category is one with only one object $M$.  In this setting the monoidal category of functors and natural transformations is called the category of sectors and written $\mathcal{C}(M) $.  

If the underlying object is an algebra $M$,  then $\mathcal{C}(M) $ can be described explicitly as follows.  The objects are the endomorphisms of $M$ and the morphisms are given by 
$$\text{Hom}(\rho, \sigma)=\{ v \in M : v \rho(x) = \sigma(x) v, \ \forall x \in M \}  $$
for $\rho, \sigma \in \text{End}(M) $.  The composition of morphisms is given by multiplication in $M$, and the tensor product of objects is given by composition of endomorphisms.  Finally, the tensor product of morphisms is given by the asymmetrical formula determined by composition of natural transformations, if $v \in (\rho, \sigma)$ and $u \in (\alpha, \beta)$ then $u \otimes v \in (\alpha \circ \rho, \beta \circ \sigma)$ is given by $\beta(v) u$.

%Let $M$ be a properly infinite factor with separable predual. Let $\text{End}_0(M) $ be the set of finite-index unital normal $*$-endomorphisms of $M$ (see \cite{MR829381,MR1027496} for the definition of index for subfactors of properly infinite factors).  Define $\mathcal{C}_0(M) $ to be the category whose objects are members of $\text{End}_0(M) $ and whose morphisms are given by 
%$$\text{Hom}(\rho, \sigma)=\{ v \in M : v \rho(x) = \sigma(x) v, \ \forall x \in M \}  $$
%for $\rho, \sigma \in \text{End}_0(M) $, with composition of morphisms given by multiplication in $M$. We often simply write $ (\rho, \sigma)$ for $\text{Hom}(\rho, \sigma)$. Define a tensor product on $\mathcal{C}_0(M) $ by $\rho \otimes \sigma = \rho \circ \sigma $. This tensor product is strictly associative and the identity endomorphism $id $ acts as a strict tensor identity. Then $\mathcal{C}_0(M) $ is a semisimple $\mathbb{C} $-linear rigid monoidal category with simple identity object and finite-dimensional morphism spaces. We write $\rho \sigma $ for $\rho \otimes \sigma $.

If $M$ is a Type III factor with separable predual, then any nice bimodule comes from an endomorphism, and so the category of sectors for $M$ is especially appropriate.
We use some standard definitions.  We consider the category $\mathcal{C}_0(M) $ of finite-index unital normal $*$-endomorphisms of $M$.
A sector in $\mathcal{C}_0(M) $ is an isomorphism class of objects; the sector associated
to an endomorphism $\rho \in \text{End}_0(M) $ is denoted by $[\rho] $. We often simply write $ (\rho, \sigma)$ for $\text{Hom}(\rho, \sigma)$.  
For a self-dual endomorphism $\rho \in \text{End}_0(M) $, let  $\mathcal{C}_{\rho} $
be the tensor category tensor generated by $\rho $; i.e. it is the full subcategory of 
$\mathcal{C}_0(M) $ whose objects are exactly those which are contained in some tensor power
of $\rho $. We say that $\rho $ has finite depth if $\mathcal{C}_{\rho}$ has finitely many simple objects up to isomorphism; in this case $\mathcal{C}_{\rho}$ is a fusion category.

%For each sector $[\rho]$, we let $[\bar{\rho}] $ be the dual sector. Note that the dual object $\bar{rho} $ is only determined up to isomorphism.

We will often use sector notation for arbitrary tensor categories and bimodule categories, so objects are denoted by lowercase Greek letters, tensor product symbols are suppressed and $(\kappa, \lambda) := \text{dim}(\text{Hom}( \kappa,\lambda))$.

\begin{definition}\cite{MR1257245}
A Q-system is an algebra $\gamma$ in $\mathcal{C}_0(M) $ such that the unit map 
$R \in (id,\gamma) $ is an isometry and the multiplication map $T \in (\gamma^2,\gamma) $ is a co-isometry.
\end{definition}

If $\gamma $ is a Q-system in $\mathcal{C}_0(M) $, then there is another properly infinite factor $N$ and a pair of finite-index unital $*$-homomorphisms $\iota: N \rightarrow M $ and $\bar{\iota}:
M \rightarrow N $ such that $\iota \bar{\iota} = \gamma $. Then $\sigma =\bar{\iota} \iota $ is a Q-system
in  $\mathcal{C}_0(N) $. The homomorphisms $ \rho^n \iota, n \in \mathbb{N} $ generate a subcategory $\mathcal{M}_{\iota} $ of the category $\text{Hom}_0(N,M)$ of finite-index homomorphisms from $N$ to $M$ (where morphisms are defined exactly as for $\text{End}_0 $). Then $\mathcal{M}_{\iota} $ is an invertible $\mathcal{C}_{\gamma}-\mathcal{C}_{\sigma}$ bimodule category where the tensor product is given by composition. If $\gamma $ is a finite-depth endomorphism, then $\mathcal{C}_{\gamma}$ is a fusion category, $\mathcal{M}_{\iota} $ is equivalent to the
category of $\gamma=\underline{\text{End}} (\iota)$ modules in $\mathcal{C}_{\gamma} $, and  $\mathcal{C}_{\sigma}$ is then equivalent to the category of $\gamma-\gamma $ bimodules in $\mathcal{C}_{\gamma} $.

Conversely, if $N \subseteq M $ is a finite-index subfactor with inclusion map $\iota: N \rightarrow  M$, then there is a dual homomorphism $\bar{\iota}: M \rightarrow N$ such that $\sigma=\bar{\iota}\iota $ is a Q-system in $\mathcal{C}_0(N)$, called the canonical endomorphism, and $ \gamma= \iota \bar{\iota} $ is a Q-system in $\mathcal{C}_0(M)$, called the dual canonical endomorphism. The subcategories of  $\mathcal{C}_0(N)$ and  $\mathcal{C}_0(M)$ tensor generated by $\sigma $ and $\gamma $ are called the principal and dual even parts of the subfactor, respectively.

Finally, we recall the following calculation from \cite{ MR2418197}. Suppose $\gamma \cong id \oplus \rho $ is a finite-index endomorphism with $\rho $ a self-conjugate irreducible endomorphism not isomorphic to the identity, and let $R \in (id,\rho^2 )$ be an isometry. Then Q-systems for $\gamma $ correspond to isometries $S \in (\rho,\rho^2 )$ satisfying
$$\rho(S)R=SR, \quad  \sqrt{d}R+(d-1)S^2=\sqrt{d}\rho(R)+(d-1)\rho(S)S,$$
where $d=d(\rho)=[M:\rho(M)]^{\frac{1}{2}} $ is the statistical dimension of $\rho $. Two such Q-systems corresponding to $S$ and $S'$ are equivalent iff $S=\pm S' $.

\subsection{The Brauer--Picard groupoid and related higher structures}

There are several closely related higher categorical constructions important to the study of subfactors, which we describe in this section.  All of them are motivated by trying to study not just a single division algebra in a fusion category $\mathcal{C}$, but all division algebras at once.  Ocneanu called this the ``maximal atlas'' and it can be formalized in several ways.

The simplest construction is to look at all division algebras in $\mathcal{C}$ (or, more generally, the Frobenius algebra objects in $\mathcal{C}$) and all bimodule objects between them.  These bimodules form a $2$-category as follows: the objects are division algebras in $\mathcal{C} $, the $1$-morphisms are bimodules between division algebras, and the $2$-morphisms are bimodule maps. The composition of $1$-morphisms in this $2$-category is given by the relative tensor product over the common algebra.  Furthermore, this $2$-category is rigid in the sense that every bimodule has a dual bimodule \cite{MR2075605}.  This rigidity is often called Frobenius reciprocity.  This construction has a strong subfactor flavor.  Indeed one can think of the objects of this $2$-category as a certain finite collection of von Neumann algebras, the $1$-morphisms as a finite collection of finite index bimodules between them which are closed under tensor product, and the $2$-morphisms as bimodule maps.

The second construction looks beyond the original fusion category $\mathcal{C}$ and thus is a little more difficult to understand from a pure subfactor point of view.  Any division algebra in $\mathcal{C}$ gives a Morita equivalence mod-$A$ between the fusion category $\mathcal{C}$ and the fusion category of $A$-$A$ bimodules.  Thus, it is natural to study all bimodule categories between all fusion categories in the Morita equivalence class of $\mathcal{C}$.  On the other hand, different algebras can give rise to the same module category, so we can consider a structure that is not defined in terms of specific algebras.

\begin{definition}\cite{MR2677836}
The Brauer-Picard $3$-groupoid of a fusion category $\mathcal{C} $ is the $3$-groupoid whose objects are fusion categories Morita equivalent to $ \mathcal{C}$, whose $1$-morphisms are invertible bimodule categories between such fusion categories, whose $2$-morphisms are equivalences between such bimodule categories, and whose  $3$-morphisms are isomorphisms of such equivalences.

This $3$-groupoid can be truncated to an ordinary groupoid whose points are the fusion categories which are Morita equivalence to $\mathcal{C}$ and whose arrows are equivalence classes of Morita equivalences.  The group of Morita autoequivalences of a fusion category, considered modulo equivalence of bimodule categories, forms a group, called the Brauer-Picard group. The Brauer-Picard group is an invariant of the Morita equivalence class.
\end{definition}

Typically groupoids are considered up to equivalence. For example the trivial $1$-point groupoid is equivalent to the groupoid with $n$ points and exactly one arrow between any pair of them even though they have a different number of points.  We often consider a finer invariant of the Brauer--Picard groupoid.  Namely, consider the groupoid whose points are fusion categories up to ordinary equivalence (not Morita equivalence!) and whose arrows are equivalence classes of Morita equivalences.  By Ocneanu rigidity \cite{MR2183279}, this gives a groupoid with a finite number of points.  When we say something like the Brauer--Picard groupoid has exactly $6$ points and exactly $4$ arrows between each of them, we are referring to this refined version.

An autoequivalence of a fusion category $\mathcal{C} $ is called inner
if it is equivalent as a monoidal functor to conjugation by an invertible object in $\mathcal{C} $. The group $\text{Out}(\mathcal{C}) $ of autoequivalences of a fusion category $\mathcal{C} $ modulo inner autoequivalences is a subgroup of the Brauer-Picard group, via the map sending $\alpha$ to the bimodule ${}_{\mathcal{C}} \mathcal{C} {}_{\alpha(\mathcal{C})}$ where the right action is twisted by $\alpha$.

\begin{example}
The Brauer--Picard groupoid of $\mathrm{Vec}(\mathbb{Z}/p\mathbb{Z})$ has exactly one point.  The group of outer automorphisms is the group of ordinary outer automorphisms of $\mathbb{Z}/p\mathbb{Z}$ acting in the obvious way.
The Brauer--Picard group is the dihedral group $(\mathbb{Z}/p\mathbb{Z})^\times \rtimes \mathbb{Z}/2\mathbb{Z}$ where the subgroup is the group of outer automorphisms and the other coset comes from realizing $\mathrm{Vec}$ as a bimodule category over $\mathbb{Z}/p\mathbb{Z}$.  Such bimodules can be realized as follows: start with the Morita equivalence between $\mathrm{Vec}(\mathbb{Z}/p\mathbb{Z})$ and $\mathrm{Rep}(\mathbb{Z}/p\mathbb{Z})$ from Example \ref{ex:standardmorita} and then pick an equivalence $\mathrm{Vec}(\mathbb{Z}/p\mathbb{Z}) \cong \mathrm{Rep}(\mathbb{Z}/p\mathbb{Z})$.
\end{example}

\begin{example}
If $G$ is a non-abelian group, then the Brauer--Picard groupoid of $\mathrm{Vec}(G)$ has at least two distinct points since $\mathrm{Vec}(G)$ and $\mathrm{Rep}(G)$ are Morita equivalent via the bimodule from Example \ref{ex:standardmorita}.  Specifically the Brauer--Picard groupoid of $\mathrm{Vec}(S_3)$ has exactly two points, and exactly two Morita equivalences between each points.  The nontrivial Morita autoequivalence of $\mathrm{Rep}(S_3)$ is given by $\mathrm{Rep}(S_2)$ (this corresponds to the $A_5$ subfactor with index $3$).
\end{example}

In general there may be more than one equivalence ${}_{\mathcal{C}}{\mathcal{M}} {}_{\mathcal{D}}   \boxtimes_{\mathcal{D}}  {}_{\mathcal{D}}{\mathcal{N}} {}_{\mathcal{E}} \cong {}_{\mathcal{C} } {\mathcal{L} } {}_{\mathcal{E} }$.  In fact, such choices of equivalence are a torsor for $\pi_2$ of the Brauer--Picard groupoid.  In particular, for the Asaeda--Haagerup subfactor there is no such ambiguity, since $\pi_2 $ is trivial \cite{  GJSext}.  Similarly, if $\pi_2$ were nontrivial then one has to be more careful about the associativity of these multiplication maps.   

There is a third  construction which unifies the above two points of view and elucidates the associativity of composition of bimodules.  Although we will not use this third construction in this paper it may clarify the above constructions conceptually.  One can consider a $3$-category whose objects are pairs $A \in \mathcal{C}$ of a Frobenius algebra object in a spherical fusion category, whose $1$-morphisms are pairs ${}_A M_B \in {}_\mathcal{C}\mathcal{M}_{\mathcal{D}}$ of a bimodule object in a bimodule category, whose $2$-morphisms are pairs of a bimodule functor $\mathcal{F}: \mathcal{M} \rightarrow \mathcal{N}$ and a bimodule map $f: \mathcal{F}(m) \rightarrow n$, and whose $3$-morphisms are pairs of bimodule natural transformations satisfying a compatibility condition.

\subsection{Combinatorics of fusion and module categories}

A fusion category contains a lot of information. At the first level there is the combinatorial information of how the objects tensor, but then at the second level there is the linear algebraic information of all the morphisms and their compositions and tensor product.  The combinatorial data is typically a rather weak invariant of the fusion category.  However, once one considers the whole Brauer--Picard groupoid, the combinatorics for tensoring all the objects in all the tensor categories and all the bimodules is a lot richer.

The fusion ring of a fusion category $\mathcal{C} $ is the based ring with basis indexed by the equivalence classes of  simple objects of $\mathcal{C} $, with addition given by direct sum and multiplication given by tensor product. There is also an involution, defined on basis elements by duality in $\mathcal{C} $.  There is a unique homomorphism from the fusion ring to $\mathbb{R} $ which takes basis elements to positive numbers, called the Frobenius-Perron dimension, and denoted by $d$.  

In a similar way, module categories and bimodule categories determine based modules and bimodules over the fusion rings of the corresponding fusion categories, which we call fusion modules and fusion bimodules.  These modules are assigned dimension functions, also denoted by $d$, as follows: for each basis element $m $, set $d(m)=\sqrt{d(\underline{\text{End}}(M))} $, where $M$ is the corresponding simple object in the module category. Then $d$ is multiplicative for the module multiplication as well, and does not depend on whether one takes the left or right internal end in invertible bimodule categories. 

If $A$ is an algebra in $\mathcal{C}$ then the principal graph of $A$ is the induction-restriction graph between $\mathcal{C}$ and $A$-mod, while the dual principal graph is the induction-restriction graph between $A$-mod and $A$-mod-$A$.  The principal graphs record only part of the fusion bimodule structure, namely the fusion rules for tensoring with the basic bimodules ${}_{1}A_A$ and ${}_{A}A_1$. The principal and dual graphs of a finite-depth subfactor $N \subset M $ are the graphs of the algebras $\bar{\iota} \iota $ and $ \iota \bar{\iota} $ in the principal and dual even parts, where $\iota: N \rightarrow M $ is the inclusion map.

We refer to \cite{GSbp} for the definitions of fusion modules and bimodules, and multiplicative compatibility for triples of fusion modules and bimodules; however we briefly summarize the basic idea.

Let $ {}_{\mathcal{C}}{\mathcal{M}} {}_{\mathcal{D}}$ and ${}_{\mathcal{D}}{\mathcal{N}} {}_{\mathcal{E}}$ be invertible bimodule categories over fusion categories, and let ${}_{\mathcal{C}}{\mathcal{M}} {}_{\mathcal{D}}   \boxtimes_{\mathcal{D}}  {}_{\mathcal{D}}{\mathcal{N}} {}_{\mathcal{E}} \cong {}_{\mathcal{C} } {\mathcal{L} } {}_{\mathcal{E} }$. Let ${}_C M {}_D$, $ {}_D N {}_E$, and $ {}_C  L {}_E $ be the 
corrseponding  fusion bimodules over fusion rings. Then the equivalence of bimodule categories induces a bimodule homomorphism from ${}_C M {}_D \otimes_D {}_D N {}_E  $ to $ {}_C  L {}_E  $ such that if $m $ and $n$ are basis elements in $ {}_C M {}_D$ and 
${}_D N {}_E  $ respectively, then the image of $m \otimes n $ is a non-negative combination
of basis elements of  $ {}_C  L {}_E  $; moreover, this homomorphism preserves Frobenius-Perron dimension and multiplication by duals in the fusion bimodules. The existence of such a homomorphism places strong combinatorial constraints on the triple $({}_C M {}_D, {}_D N {}_E, {}_C  L {}_E )$, which can be checked tediously by hand or quickly with a computer.

\begin{remark}
As mentioned above, in general one needs to be a bit careful about the associativity of compositions of bimodule categories, and this is true even at the level of fusion bimodules.  For AH, the vanishing of $\pi_2$ guarantees that the composition of bimodules is well-defined, and also that these compositions are associative at the fusion bimodule level.  (Whether one can give compatible associators is a more subtle question which involves $\pi_3$ as well.)  Nonetheless we will never need to use this uniqueness or associativity.
\end{remark}

%Finally, we recall the definition of the principal and dual graphs. The fusion graph of a basis element $m_0$ in a fusion module over a fusion ring is a bipartitie graph with even vertices indexed by the basis of the fusion ring and odd vertices indexed by the basis of the fusion module. If $r$ is an element of the fusion ring basis and $m$ is an element 
%of the fusion module basis, then the number of edges connecting the corresponding vertices is the coefficient of $ m$ in $rm_0 $. Similarly, an element in a fusion bimodule has two fusion graphs, one for the left action and one for the right action, called the principal and dual graphs. The principal and dual graphs of a finite-depth subfactor $N \subset M $ are, respectively, the principal and dual graphs of the basis element corresponding to the inclusion map $ \iota: N \rightarrow M$ in the bimodule category $\mathcal{M}_{\iota} $ over the principal and dual even parts of the subfactor.

\subsection{The Asaeda-Haagerup fusion categories}
The Asaeda-Haagerup subfactor, constructed in \cite{MR1686551}, is a finite depth subfactor with index $\frac{5+\sqrt{17}}{2} $
and principal and dual graph pair
$$\vpic{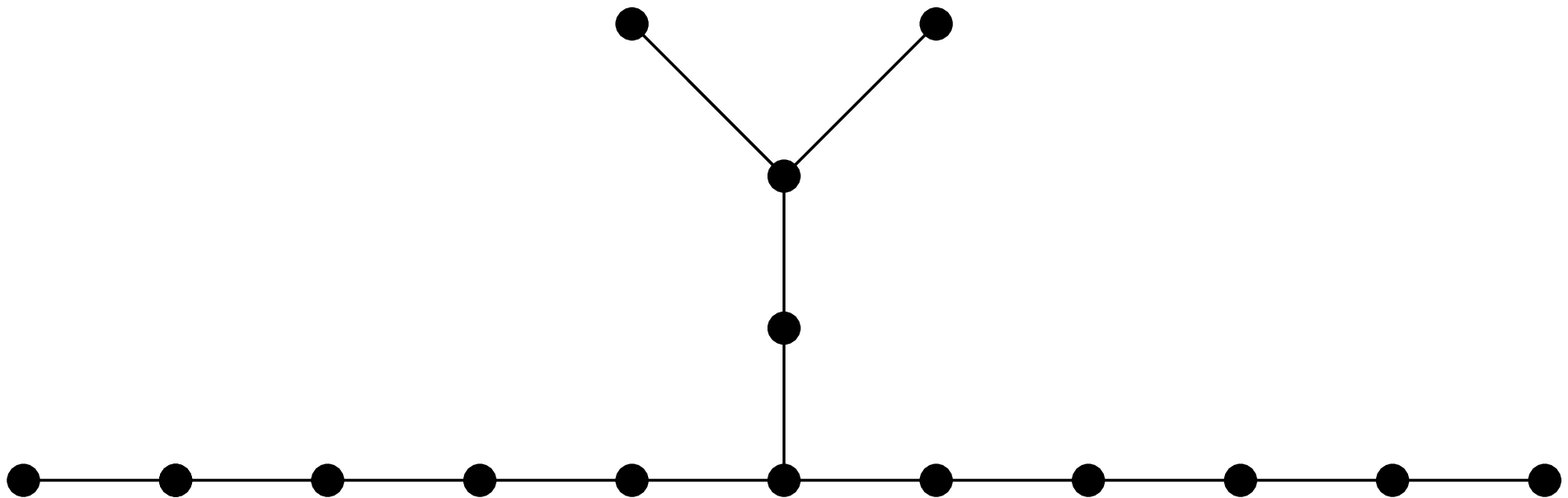} {1.4in}  , \vpic{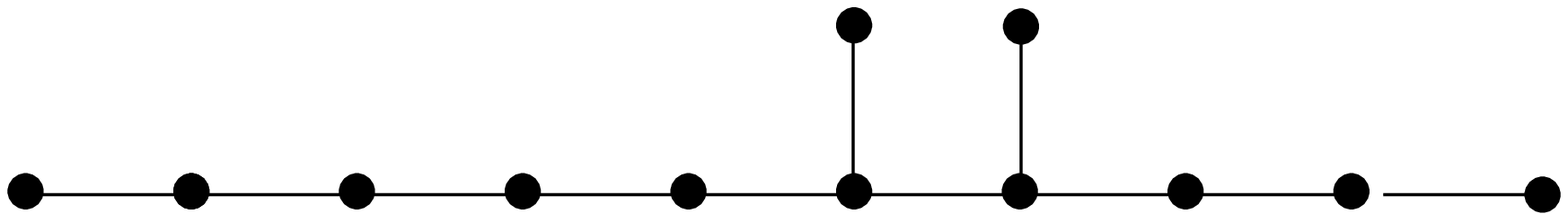} {1.4in}  .$$
A subfactor with index $\frac{7+\sqrt{17}}{2} $  and principal and dual graph pair $$ \vpic{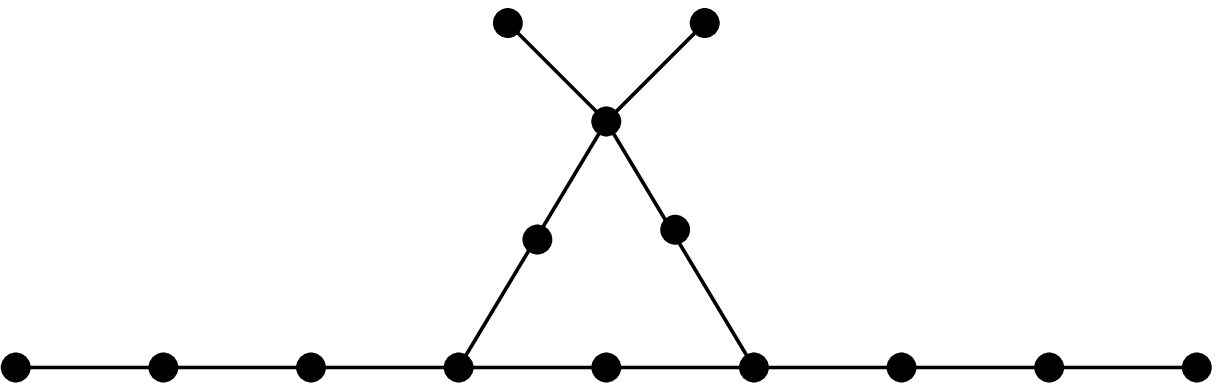} {1.4in}  , \vpic{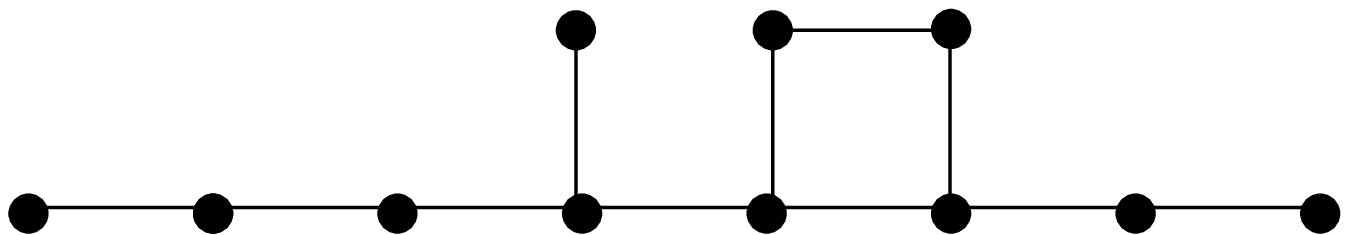} {1.4in}  $$
was constructed in \cite{MR2812458}, and a third subfactor with index $\frac{9+\sqrt{17}}{2} $  and principal and dual graph pair
$$ \vpic{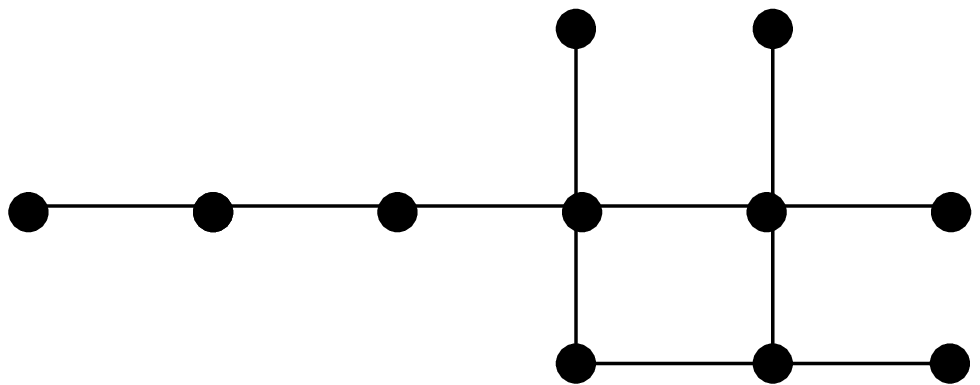 } {1.4in},  \vpic{AHp2dual } {1.4in}  $$ 
was constructed in \cite{GSbp}.

Let $\mathcal{AH}_2$, $\mathcal{AH}_3$, and $\mathcal{AH}_1  $ be the principal even parts of these three subfactors, respectively ($\mathcal{AH}_1 $ is also the dual even part of all three subfactors).    The fusion ring of $\mathcal{AH}_i$ is denoted $AH_i$ and the structure contants for these fusion rings are given in the Appendix.  In \cite{GSbp} all possible fusion modules and fusion bimodules over these $AH_i$ were computed.  Each of these fusion modules can be described by a small list of small matrices with small integer entries (here small means single digit) where each matrix describes the combinatorics of tensoring with one of the simple objects.  

%Since there are dozens of such modules we list them all in some tables in the online appendix to \cite{GSbp}.  We will use the same notation as in those appendices, so $k_i$ will denote the $k$th fusion module for $AH_i$ and $k_{ij}$ will denote the $k$th $AH_i$-$AH_j$ fusion bimodule from those tables.

Our main result from that paper was a description of all Morita equivalences between any two of the $\mathcal{AH}_i$.

\begin{theorem}
There are exactly four invertible bimodule categories over each not-necessarily-distinct pair $\mathcal{AH}_i-\mathcal{AH}_j$, up to equivalence. 

%These realize the following fusion bimodules, which are each realized uniquely:
%\begin{align*}
%&10_{11}, 12_{11}, 13_{11}, 14_{11}, 2_{12}, 5_{12}, 8_{12}, 9_{12},  2_{13}, 3_{13}, 6_{13}, 7_{13} \\
%& 2_{21}, 5_{21}, 8_{21}, 9_{21}, 8_{22}, 11_{22}, 12_{22}, 13_{22}, 1_{23}, 3_{23}, 4_{23}, 6_{23} \\
%& 2_{31}, 3_{31}, 6_{31}, 7_{31}, 1_{32}, 2_{32}, 4_{32}, 6_{32}, 8_{33}, 11_{33}, 12_{33}, 13_{33}.
%\end{align*}

Furthermore, the Brauer-Picard group of each $\mathcal{AH}_i$ is $\mathbb{Z}/2\mathbb{Z} \times \mathbb{Z}/2\mathbb{Z} $. 
\end{theorem}
The fusion rules for all $36 $ bimodule categories between the various $\mathcal{AH}_i-\mathcal{AH}_j $ were given in an online appendix to \cite{GSbp}. In addition, the combinatorial possibilities for any other fusion categories in the same Morita equivalence class were severely constrained.

\begin{theorem} \cite[Theorem 6.10]{GSbp} \label{cases}
 
 Let $\mathcal{C}^{(k)} $ be a fusion category which is Morita equivalent to the $\mathcal{AH}_i$ (for $i =1,2,3$), but not equivalent to any of them. Then there is a triple of fusion bimodules
 $(K_{AH_1}^{(k)} , L_{AH_2}^{(k)}, M_{AH_3}^{(k)} ) $ such that every $\mathcal{C}-\mathcal{AH}_1 $ Morita equivalence realizes $K_{AH_1}^{(k)} $, every $\mathcal{C}-\mathcal{AH}_2 $ Morita equivalence realizes $L_{AH_2}^{(k)}$, and every $\mathcal{C}-\mathcal{AH}_3 $ Morita equivalence realizes $L_{AH_3}^{(k)}$.
 
 Moreover, there are exactly four possibilities for the triple $(K_{AH_1}^{(k)} , L_{AH_2}^{(k)}, M_{AH_3}^{(k)}  )$. \end{theorem}

 Since we will need the details of the four possiblities for $(K_{AH_1}^{(k)} , L_{AH_2}^{(k)}, M_{AH_3}^{(k)}  )$, we include them in the Appendix for the convenience of the reader. We will refer to these cases using $k = 4, 5, 6, 7$.  We will see that cases $(4)$, $(5)$, and $(6)$ have a somewhat similar flavor to each other, while case $(7)$ is quite different.  These cases correspond to cases (c), (a), (b), and (d) respectively in \cite[Theorem 6.10]{GSbp}.

Note that the situation described by Theorem \ref{cases} is very different than what we see for $\mathcal{AH}_{1-3} $, for which the fusion bimodule of each  $\mathcal{AH}_i-\mathcal{AH}_j $ invertible bimodule category are distinct from each other either as a left $AH_i $ fusion module or as a right $AH_j $ fusion module.  

%This happens because $\mathcal{AH}_{4-6} $ have outer automorphisms while $\mathcal{AH}_{1-3} $ do not.
 
% 
% Then exactly one of the following four cases holds:
% 
% \begin{enumerate}[(a)]
%
%\item Every $\mathcal{C}-\mathcal{AH}_1 $ Morita equivalence realizes $9_1$, every $\mathcal{C}-\mathcal{AH}_2 $ Morita equivalence realizes $19_2$, and every $\mathcal{C}-\mathcal{AH}_3 $ Morita equivalence realizes $16_3$.
%
%\item Every $\mathcal{C}-\mathcal{AH}_1 $ Morita equivalence realizes $16_1$, every $\mathcal{C}-\mathcal{AH}_2 $ Morita equivalence realizes $4_2$, and every $\mathcal{C}-\mathcal{AH}_3 $ Morita equivalence realizes $18_3$.
%
%\item Every $\mathcal{C}-\mathcal{AH}_1 $ Morita equivalence realizes $21_1$, every $\mathcal{C}-\mathcal{AH}_2 $ Morita equivalence realizes $17_2$, and every $\mathcal{C}-\mathcal{AH}_3 $ Morita equivalence realizes $3_3$.
%
%\item Every $\mathcal{C}-\mathcal{AH}_1 $ Morita equivalence realizes $1_1$, every $\mathcal{C}-\mathcal{AH}_2 $ Morita equivalence realizes $3_2$, and every $\mathcal{C}-\mathcal{AH}_3 $ Morita equivalence realizes $2_3$.
%\end{enumerate}
%\end{theorem}
%

\section{Classification of the Morita equivalence class of the Asaeda-Haagerup fusion categories}
\label{sec:BP}
The main result of this section is the following. 

\begin{theorem} \label{thm:dichotomy}
One of the following two claims is true:
\begin{itemize}
\item There are no fusion categories realizing any of cases $(4)$, $(5)$, or $(6)$.
\item There is a unique fusion category realizing each of cases $(4)$, $(5)$, and $(6)$, but there are no fusion categories realizing case $(7)$.
\end{itemize}
\end{theorem}

\begin{remark}
In Section \ref{sec:main} we will see that the latter possibility is the correct one by realizing case $(4)$.
\end{remark}

Our first goal will be to constrain the possible fusion rings for the dual fusion categories of the fusion modules $(K_{AH_1}^{(k)} , L_{AH_2}^{(k)}, M_{AH_3}^{(k)} )$.

\subsection{Finding the possible dual fusion modules of a fusion module}

Given a module category over a fusion category  ${}_{\mathcal{C}} \mathcal{M}$, we would like to be able to compute the dual fusion category $\mathcal{D}=( {}_{\mathcal{C}} \mathcal{M})^*$, and in particular its fusion ring. This former computation can be difficult even when the module category is well understood. %and indeed impossible without explicit formulas for the data of ${}_{\mathcal{C}} \mathcal{M}$ such as $6j$-symbols.  
However, just the fusion module structure of ${}_{\mathcal{C}} \mathcal{M}$ can provide strong combinatorial constraints on the fusion rules of $\mathcal{D} $, and this is sometimes sufficient to calculate the fusion ring of $\mathcal{D} $ and the dual fusion module of $\mathcal{M} {}_{\mathcal{D}} $, or at least to narrow it down to a few choices.
%
%We now briefly describe an algorithm to enumerate possible dual fusion modules to a given fusion module over a fusion ring.

The idea is that Frobenius reciprocity relates the combinatorics of the two even parts of a subfactor.  Suppose that $\iota$ is the inclusion $N \rightarrow M$ and that $\psi$ and $\chi$ are other $M$-$N$ sectors.  By Frobenius reciprocity, we have that $(\iota \bar{\psi} , \iota \bar{\chi}) = (\bar{\iota} \iota,  \bar{\chi} \psi)$.  But $(\iota \bar{\psi} , \iota \bar{\chi}) $ counts the number of length two paths between $\psi$ and $\chi$ in the dual principal graph, while $(\bar{\iota} \iota,  \bar{\chi} \psi)$ is asking about fusing two odd bimodules into the principal even part.  Thus we can read off the number of length two paths on the dual principal graph just from understanding the principal half very well.  Just knowing the odd vertices and the number of length two paths between each pair of odd vertices is often enough to determine the even vertices as well.  More generally, the following lemma is an immediate consequence of Frobenius reciprocity.

\begin{lemma} \label{duallem}
Let $(K,T) $ be a fusion bimodule over the fusion rings $(A,R) $ and $(B,S)$ (where the second argument in each pair refers to the basis.) For fixed $\kappa \in T $, define the matrices $M^{\kappa}$ and $N^{\kappa}$ as follows: $$ M^{\kappa}_{ij}=  \sum_{\xi \in R} \limits (\xi \kappa, \kappa) ( \xi \mu_i, \mu_j),  \ \mu_i, \mu_j \in T$$  and $$ N^{\kappa}_{ij}= (\mu_i, \kappa \eta_j), \ \eta_j \in S, \mu_i \in T.$$ Then $N^{\kappa} (N^{\kappa})^t =M$.
\end{lemma}

\begin{proof}
We have $(N^{\kappa}(N^{\kappa})^t)_{ij} = \sum_k \limits (\mu_i, \kappa \eta_k)(\mu_j, \kappa \eta_k)=(\bar{\kappa} \mu_i ,\bar{ \kappa} \mu_j)=(\kappa \bar{\kappa}, \mu_j \bar{\mu_i})=\sum_{\xi \in R} \limits (\xi \kappa, \kappa) ( \xi \mu_i, \mu_j)=M^{\kappa}_{ij}$.
\end{proof}

The matrix $M^{\kappa}$ in Lemma \ref{duallem} is part of the data of the left fusion module ${}_A K $, while the matrix $N$ is the fusion matrix of $\kappa $ with respect to the right action of $B$. Thus given $ {}_A K$, we can compute all possibilities for the right fusion matrix of $\kappa $ in a fusion bimodule ${}_A K {}_B $ whose left fusion module is ${}_A K $ if we can find all decompositions of
 $M^{\kappa}$ into $NN^t $ for non-negative integer matrices $N$. Any choice of $N$ determines the dimensions of the basis elements in $B$, and we may immediately discard any choices which do not allow for a basis element of dimension $1$ (the identity) as well as any choices for which there is some other $\kappa'  \in T$ which does not admit a right fusion matrix corresponding to the same dimension values for $B$.
 
 Once we have found all possible right fusion matrices for each basis element of $K $, we can try to reconstruct all possible fusion bimodules ${}_A K {}_B $ whose left fusion modules are ${}_A K $, using similar combinatorial searches as in \cite{GSbp}.
 
 These can easily be made into computer algorithms all of which run quite quickly since the number of possibilities is small.  In fact, all the results of this section were originally checked by computer, but since the calculations are all simple enough we have also checked them all by hand.

\subsection{The structure of fusion categories realizing cases (4)-(6)}
The first step toward proving Theorem \ref{thm:dichotomy} is to understand the simple objects and fusion rules for any fusion category realizing cases (4)-(6).

\begin{lemma}\label{dualgraphlemma}
Let $\mathcal{C} $ be a fusion category realizing case (4),(5), or (6) of Theorem \ref{cases}. Then

\begin{enumerate}
\item $\mathcal{C} $ has $4$ invertible objects and $4$ simple objects of dimension $4+\sqrt{17} $. 

%The fusion ring of $\mathcal{C} $ satisfies $(\xi_i \xi_j,\xi_k)=2 $ where $\xi_i,\xi_j,\xi_k $ are any three non-invertible objects.
\item If $\mathcal{C}$ realizes case (4) (resp. case (5)), then there is a simple object in an invertible $\mathcal{C}-\mathcal{AH}_2$(resp. $\mathcal{C}-\mathcal{AH}_3$) bimodule category whose principal graph is 
$$\vpic{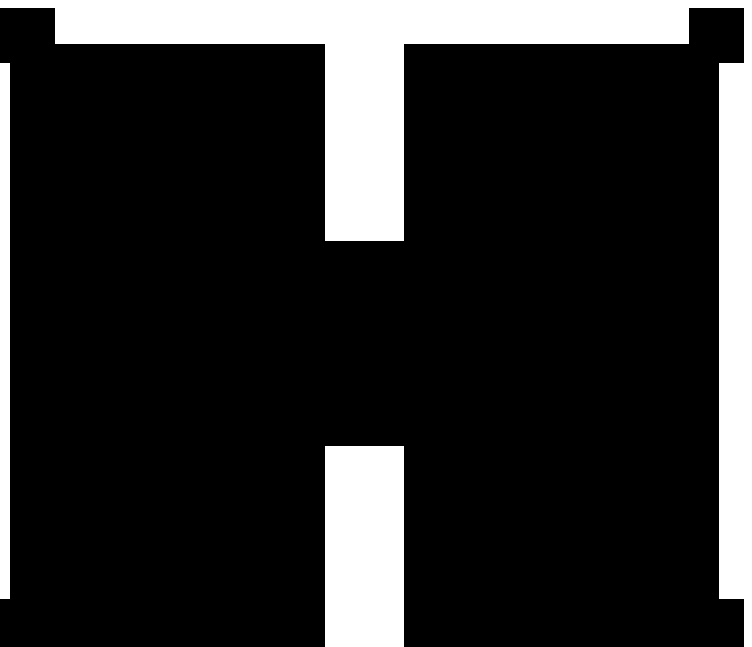} {1.5in}   $$
and the adjacency matrix of whose dual graph is the first matrix appearing in the fusion module $ L_{AH_2}$ for case (4) (resp. $M_{AH_3}$ for case (5)).

\end{enumerate}

\end{lemma}

\begin{proof}
In each of the three cases, we look for a basis element of small dimension in one of the fusion modules.
In case (4), $\theta^1_{42}$ has a $2$-supertransitive principal graph (the matrix given in the appendix is the adjacency matrix of this graph).  In case (5), $\theta_{53}^1$ also has a $2$-supertransitive principal graph.
% Cases (4) and (5) have basis elements with $2$-supertransitive fusion graphs in the fusion modules $L_{AH_2}^{(4)}$ and $M_{AH_3}^{(5)}$, respectively. 
We can apply Lemma \ref{duallem} to these basis elements to obtain the pictured graph; the dimensions of the simple objects in $\mathcal{C}$ are then given by the normalized Frobenius-Perron weights of the graph.
 
Since cases (4) and (5) are similar, we only consider case (4) in detail.  We refer to the table of fusion modules in the Appendix and use the notation from there. For the fusion module $L_{AH_2}^{(4)}$ we consider the object $\theta^1_{42} $ and compute the matrix
$$M^{\theta^1_{42}}_{ij}=\sum_{\phi} (\theta^1_{42} \phi,\theta^1_{42})(\theta^i_{24} \phi, \theta^j_{24})  $$ of Lemma \ref{duallem}, where $\phi $ ranges over the simple objects of $\mathcal{AH}_2 $. Since $(\theta^1_{42} \phi,\theta^1_{42}) =1$ if $\phi $ is either $1$ or $\alpha \pi $ and $(\theta^1_{42} \phi,\theta^1_{42}) =0$ otherwise, the rows of $M^{\theta^1_{42}}$ are obtained by adding the rows corresponding to $1$ and to $\alpha \pi $ in each submatrix of the fusion module $L_{AH_2}^{(4)}$. Thus we have
$$M^{\theta^1_{42}}=
\left(
\begin{array}{cccccc}
2 & 0 & 0 & 0 & 1 & 1 \\
0 & 2 & 0 & 0 & 1 & 1 \\
0 & 0 & 2 & 0 & 1 & 1 \\
0 & 0 & 0 & 2 & 1 & 1 \\
1 & 1 & 1 & 1 & 4 & 4 \\
1 & 1 & 1 & 1 & 4 & 4 \\
\end{array}
\right) .$$

Then from the relation $ NN^t=M$ it is easy to see that up to permutations of columns, the left fusion matrix of $ \theta^1_{42}$
must be 
$$N^{\theta^1_{42}}=
\left(
\begin{array}{cccccccc}
1 & 1 & 0 & 0 & 0 & 0 & 0 & 0 \\
0 & 0 & 1 & 1 & 0 & 0 & 0 & 0 \\
0 & 0 & 0 & 0 & 1 & 1 & 0 & 0 \\
0 & 0 & 0 & 0 & 0 & 0 & 1 & 1 \\
1 & 0 & 1 & 0 & 1 & 0 & 1 & 0 \\
1 & 0 & 1 & 0 & 1 & 0 & 1 & 0 \\
\end{array}
\right) , $$
which is the adjacency matrix of the pictured graph.

One can also think of the above calculation as counting paths on principal graphs, which can be an easier way to arrange the calculation if you are doing it by hand.  From the fusion matrix for $\theta^1_{42} $ we see that the graph for fusion with $\theta^1_{42}$ has one odd object at depth one ($\theta^1_{42}$), two at depth three ($\theta^5_{24}$ and $\theta^6_{24}$), and three at depth five ($\theta^2_{24}$, $\theta^3_{24}$, and $\theta^4_{24}$).  Thus the dual graph which we are trying to compute must also have the same number of odd objects at each depth.  Now, by Frobenius reciprocity, the matrix $M$ above counts the number of paths of length two between these odd vertices.  In particular, the vertex at depth one is $2$-valent and has exactly one length-two path to each of the depth three vertices.  Thus, the graph must agree with the pictured graph through depth three.  Then the vertices at depth three are $4$-valent and since one edge is already accounted for, they must each have three edges connected to depth four vertices.  Since there are four length-two paths between distinct depth three vertices (one of which is already accounted for), there must be exactly three vertices at depth four each of which is connected to both depth three vertices.  Thus the graph must agree with the given one through depth four.  Since each of the vertices at depth five is $2$-valent and only has one length-two path to each of the depth three vertices, each must have a single edge going to a vertex at depth six.  Since these three vertices at depth five have no length two paths between them there must be three such vertices at depth six.  Thus the graph must be the one in the theorem.

For case (6), we apply in a similar way Lemma \ref{duallem} to each basis element in the fusion module $M^{(6)}_{AH_3}$. The dimensions listed above are the only choice for which each basis element admits a consistent graph.

\end{proof}

\begin{remark}
In the rest of this section there are several elementary calculations like the one in the last lemma.  Since they are all straightforward and similar to the first calculation, we will not go through them in detail.
\end{remark}

\begin{remark} 
The principal graph of the smallest simple object in an invertible $\mathcal{C}-\mathcal{AH}_3$ bimodule category realizing case (6) must be 
$$\vpic{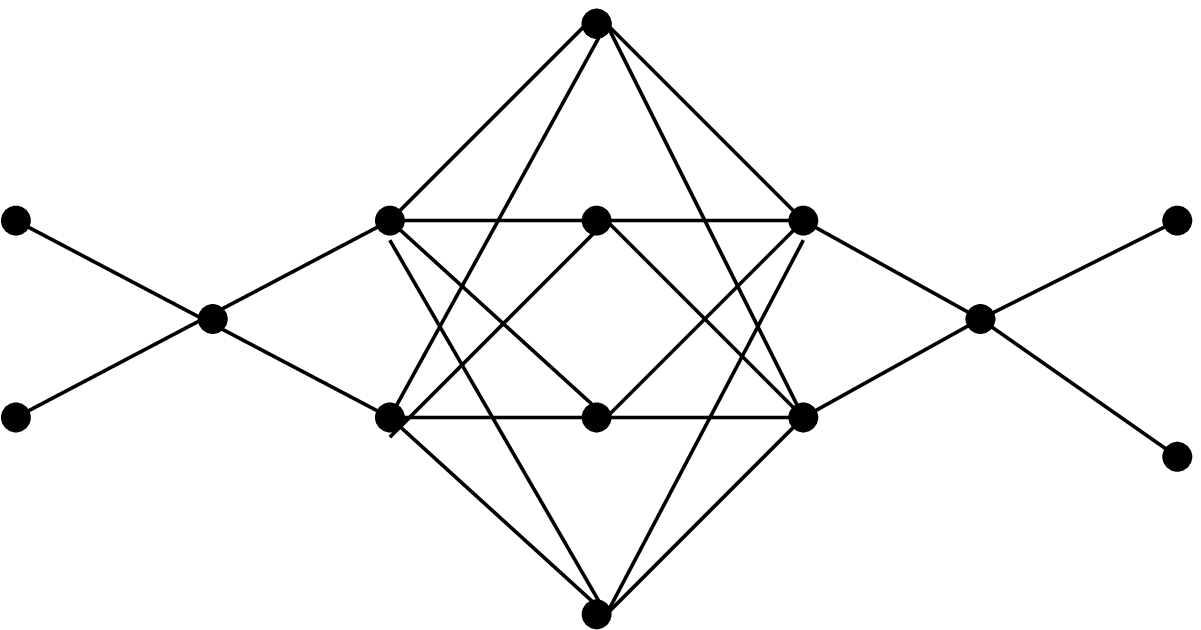} {3in} .$$
\end{remark}

\begin{lemma} \label{lemah4}
Let $\mathcal{C} $ be a fusion category which admits a module category with an object having the fusion graph pictured in Lemma \ref{dualgraphlemma}. Then there is an order $4$ group $G$ such that the basis for the Grothedieck ring of $\mathcal{C}$ is indexed by $
\{\alpha_g\}_{g \in G} \cup  \{\alpha_g \rho \}_{g \in G} $, satisfying $$\quad \alpha_g \alpha_h = \alpha_{g+h}, \quad \text{and }\rho^2=1+\sum_{g \in G} 2\alpha_g \rho .$$ There are exactly four possibilities for the fusion ring up to isomorphism, corresponding to four choices for the depth preserving duality involution on the basis: 
\begin{enumerate}
\item the trivial involution, where $G$ is the Klein $4$-group and $\alpha_g \rho = \rho \alpha_g$,
\item the involution which swaps two invertible basis elements, where $G$ is cyclic and $\alpha_g \rho = \rho \alpha_{g^{-1}}$,
\item the involution which swaps two non-invertible basis elements, where $G$ is the Klein $4$-group, and without loss of generaltiy $\alpha_{(1,0)} \rho = \rho \alpha_{(0,1)}$, and $\alpha_{(0,1)} \rho = \rho \alpha_{(1,0)}$,
\item the involution which swaps two invertible basis elements and two noninvertible basis elements, where $G$ is cylic and $\alpha_g \rho = \rho \alpha_g$.
\end{enumerate}
\end{lemma}
\begin{proof}
It is clear from the principal graph that there are four invertible objects (which must have the fusion rules for some group of size four) and four noninvertible objects.  Let $\rho$ be the non-invertible object at depth $2$.  Tensoring a simple object with $\rho$ counts the number of length-$2$ paths (from the vertex reprepresenting that object to the vertex representing $\rho$) minus the number of length $0$ paths.  Since each of the non-invertible objects has a length-$2$ path to a different invertible object, it follows that all non-invertible objects are of the form $\alpha_g \rho$.  Furthermore, $\rho^2$ has one length-$2$ path to the identity, three length-$2$ paths to itself, and two length-$2$ paths to each of the other non-invertible objects, so $\rho^2 = 1+\sum_{g \in G} 2\alpha_g \rho$.  Finally, it is clear that those are the only four possibilities for depth preserving involutions on the even vertices, and the fusion rules are determined by the duality using that $\overline{\alpha_g \rho} = \rho \overline{\alpha_g}$.
\end{proof}

%
%\begin{remark}
%
%In the following arguments, we use the algebraic language of algebras in fusion categories. An algebra in a fusion category is not exactly the same thing as a $Q$-system, but in unitary fusion categories these two notions often coincide. See the discussion before Lemma 2.18 in \cite{MR2909758}.
%\end{remark}

In the following arguments we will frequently need to find the list of objects in a fusion category which admit the structure of a division algebra whose category of modules realizes a given fusion module. Since every division algebra is realized as the internal endomorphisms of a simple object in its category of modules, this list can be read off the data of the fusion module.

We give an example of how this works. Suppose we would like to find the list of objects in $\mathcal{AH}_3 $ which admit an algebra structure whose category of modules realizes the fusion module $M_{AH_3}^{(4)}$.

The fusion module $M_{AH_3}^{(4)}$ has $3$ basis elements, with fusion matrices
$$\left(
\begin{array}{c|ccc|ccc|ccc}
& \multicolumn{2}{l}{\theta_{43}^1 \otimes \_} & & \multicolumn{2}{l}{\theta_{43}^2 \otimes \_}  & & \multicolumn{2}{l}{\theta_{43}^3 \otimes \_} & \\
& \theta_{43}^1 & \theta_{43}^2 & \theta_{43}^3 & \theta_{43}^1 & \theta_{43}^2 & \theta_{43}^3 & \theta_{43}^1 & \theta_{43}^2 & \theta_{43}^3 \\ \hline
 1 & 1 & 0 & 0 & 0 & 1 & 0 & 0 & 0 & 1 \\
 \beta & 1 & 0 & 0 & 0 & 1 & 0 & 0 & 0 & 1 \\
 \xi & 0 & 1 & 1 & 1 & 0 & 1 & 1 & 1 & 4 \\
 \beta \xi & 0 & 1 & 1 & 1 & 0 & 1 & 1 & 1 & 4 \\
 \xi \beta & 0 & 1 & 1 & 1 & 0 & 1 & 1 & 1 & 4 \\
 \beta \xi \beta & 0 & 1 & 1 & 1 & 0 & 1 & 1 & 1 & 4 \\
 \mu & 1 & 0 & 2 & 0 & 1 & 2 & 2 & 2 & 7 \\
 \beta \mu & 1 & 0 & 2 & 0 & 1 & 2 & 2 & 2 & 7 \\
 \nu & 0 & 0 & 2 & 0 & 0 & 2 & 2 & 2 & 6 \\
\end{array}
\right)$$
where the columns of each matrix are indexed by the basis of the fusion module and the rows by the basis of the fusion ring of $\mathcal{AH}_3 $.  %The basis for the fusion ring is in turn indexed by the set of simple objects of $\mathcal{AH}_3 $, of which there are $9$; in the notation of \cite{GSbp}, these are labeled by  $1, \beta, \xi, \beta \xi, \xi \beta, \beta \xi \beta, \mu, \beta \mu, \nu $, and have dimensions $1,1,\frac{d+1}{2},\frac{d+1}{2},\frac{d+1}{2},\frac{d+1}{2},d,d,d-1 $, respectively, where $d=4+\sqrt{17}$. 
To find the underlying object of the algebra $\overline{\theta_{43}^{(i)}} \theta_{43}^{(i)}$, we just look at $\theta_{43}^{(i)}$ column in the fusion matrix for $\theta_{43}^{(i)}$.  Thus for any (right) module category $\mathcal{M} $ over $\mathcal{AH}_3$ which realizes the fusion module $M_{AH_3}$, there are two algebras with underlying objects $1+\beta +\mu+\beta \mu$ and one algebra with underlying object $1+\beta+4(\xi+\beta \xi+\xi\beta+\beta\xi\beta)+7(\mu+ \beta \mu)+6\nu $ whose categories of (left) modules are each equivalent to $\mathcal{M} $. The first two algebras each have dimension $2d+2$ and the latter algebra has dimension $28d+4$, where $d = 4+\sqrt{17}$. The dimensions of the basis elements of $M_{AH_3}^{(1)}$ are thus $\sqrt{2d+2}$, $\sqrt{2d+2}$, and $\sqrt{28d+4}$. The algebras $\theta_{43}^{(i)}\overline{\theta_{43}^{(i)}}$ in the dual category $ (\mathcal{M}_{\mathcal{C}} )^*$ hence also have dimensions $2d+2$ or $28d+4$.

We introduce the following notation. Let $\gamma $ be a simple algebra in a fusion category $\mathcal{C} $. Then we denote the category of $\gamma-\gamma $ bimodules in $\mathcal{C} $ by $\mathcal{C}^{\gamma} $. Recall that  $C^{\gamma} $ is Morita equivalent to $\mathcal{C} $.

\begin{lemma}\label{div}
\begin{enumerate}
\item Let $\gamma $ be an algebra in a fusion category $\mathcal{C} $, and let $\delta $ be the subobject of $\gamma $ which contains each invertible object of $\mathcal{C} $ with the same multiplicity as $\gamma $ does and which does not contain any non-invertible objects. Then $\delta$ inherits an algebra structure from $\gamma $.

\item Let $\mathcal{C} $ be a fusion category and let $\gamma $, $\gamma'$ be division algebras such that  $(\gamma,\gamma')=1 $ (as objects of $\mathcal{C} $). Then there is a division algebra $\delta$ in $\mathcal{C}^{{\gamma}} $ with $d(\delta)=d(\gamma)d(\gamma') $ such that $(\mathcal{C}^{\gamma} )^{\delta}$ is equivalent to $\mathcal{C}^{\gamma'} $. 

\item Let $\mathcal{C} $ be a fusion category and let $ \gamma$ be a division algebra in $\mathcal{C} $ with a subalgebra $ \delta$. Then there is a division algebra $\gamma'$ in $\mathcal{C}^{\delta}$ with dimension $d(\gamma')=\frac{d(\gamma)}{d(\delta)} $ such that 
$(\mathcal{C}^{\delta})^{\gamma'} $ is equivalent to $\mathcal{C}^{\gamma}$.
\end{enumerate}
\end{lemma}
\begin{proof}
\begin{enumerate}
\item This follows from the fact that the tensor product of any two invertible objects is invertible.
\item Recall that bimodules over division algebras in $\mathcal{C} $ form a rigid $2$-category; we will use multiplicative notation for composition of $1$-morphisms. Consider the bimodule
$\kappa={}_{\gamma} \gamma {}_1 \cdot {}_1 \gamma' {}_{\gamma'} $. By Frobenius reciprocity,
$(\kappa,\kappa)_{\gamma\text{-mod-}\gamma'} =(\gamma,\gamma')=1$, so $ \kappa$ is simple as a $\gamma-\gamma' $ bimodule. Then $\delta=\kappa \bar{\kappa} $ is a simple algebra in $\mathcal{C}^{\gamma} $, and $\bar{\kappa} \kappa $ is a simple algebra in $\mathcal{C}^{\gamma'} $. By multiplicativity of dimension in the $2$-category, we have $d(\delta)=d(\gamma)d(\gamma') $.

\item Let $\kappa=  {}_{\gamma} \gamma {}_{\delta} $ . Then $\gamma'=\bar{\kappa} \kappa$ is a simple algebra in $\mathcal{C}^{\delta} $, $\kappa \bar{\kappa} $ is a simple algebra in $\mathcal{C}^{\gamma} $, and from the relation ${}_{\gamma} \gamma {}_1=  {}_{\gamma} \gamma {}_{\delta} \cdot {}_{\delta} \delta {}_{1} $ it follows that $d(\gamma)=d(\gamma')d(\delta) $.

\end{enumerate}
\end{proof}

\begin{remark}
In the case of a fusion category coming from a subfactor, Lemma \ref{div}(3) can be interpreted as follows: the algebra $\gamma $ corresponds to a subfactor $ N \subseteq M$, the subalgebra $\delta $ corresponds to $N \subseteq P $ for an intermediate subfactor $N \subseteq P \subseteq M$, and then the  ``quotient'' algebra $\gamma' $ corresponds to $P \subseteq M $. 
\end{remark}

The goal of this subsection is to determine, for each of cases (4)-(6), which of the four possibilities the fusion ring must be.  A priori it might be that there was more than one fusion category realizing one of the cases and that these two fusion categories have different fusion rings, but it turns out that we can determine the fusion rules from be the bimodules.  First we concentrate on the invertible objects, and then the non-invertible objects.

A fusion category $\mathcal{C} $ all of whose objects are invertible is necessarily equivalent to $Vec_G^{\omega} $, where $G$ is a finite group and $\omega \in H^3(G,\mathbb{C}^{\times}) $ determines the associator. If $H$ is a subgroup of $G$, then $\bigoplus_{h \in H}  k_h $ admits an algebra structure iff $\omega|_H$ is trivial; the algebra structures are then parametrized by  $H^2(H, \mathbb{C}^{\times} )$.  In fact, these are the only division algebras in $\omega \in H^3(G,\mathbb{C}^{\times})$ (this can be seen directly, or by applying \cite{MR1976459}).

We will need to know a little about the dual categories to the division algebras in $Vec_G^{\omega}$ when $G$ has size $4$.  For $\text{Vec}_{\mathbb{Z}/4\mathbb{Z} }$ with trivial associator, the dual category over the $2$-dimensional simple algebra is $\text{Vec}_{\mathbb{Z}/2\mathbb{Z} \times \mathbb{Z}/2\mathbb{Z}}^\xi$ with a specific non-trivial associator which we denote $\xi$ \cite[Ex. 4.10]{MR2362670}.  The associator $\xi$ is trivial when restricted to each of the $\mathbb{Z}/2\mathbb{Z}$ factors, but is non-trivial when restricted to the diagonal.  In particular, $\text{Vec}_{\mathbb{Z}/2\mathbb{Z} \times \mathbb{Z}/2\mathbb{Z}}^\xi$ has exactly two non-trivial division algebras both of dimension $2$.  On the other hand, $\text{Vec}_{\mathbb{Z}/2\mathbb{Z} \times \mathbb{Z}/2\mathbb{Z}}$ with trivial associator has three different simple $2$-dimensional algebras. By Lemma \ref{div}(2), the dual category over any one of the $2$-dimensional algebras in $\text{Vec}_{\mathbb{Z}/2\mathbb{Z} \times \mathbb{Z}/2\mathbb{Z}}$ has a $4$-dimensional simple algebra, and therefore must again be $\text{Vec}_{\mathbb{Z}/2\mathbb{Z} \times \mathbb{Z}/2\mathbb{Z}}$ with trivial associator.

%
%For $G=\mathbb{Z}/4\mathbb{Z} $, there are two possibilities: either the associator is trivial and there are three simple algebras corresponding to the subgroups $\{0\}$, $\mathbb{Z}/2\mathbb{Z}$, and $\mathbb{Z}/4\mathbb{Z}  $, or the associator is nontrivial and the only algebra is the trivial one.
%
%For $G=\mathbb{Z}/2\mathbb{Z} \times \mathbb{Z}/2\mathbb{Z} $, there are again two possibilities. Either the associator is trivial, and there are simple algebras associated to all subgroups, as well as to the nontrivial central extension of $\mathbb{Z}/2\mathbb{Z} \times \mathbb{Z}/2\mathbb{Z}$, or the associator is nontrivial, and there are three simple algebras: the trivial algebra and two algebras associated to order two subgroups. In the latter case, the category of bimodules over each of the $2$-dimensional simple algebras is equivalent to $\text{Vec}_{\mathbb{Z}/4\mathbb{Z}} $ with trivial associator.

If $\mathcal{C}$ is a fusion category, we denote by $\text{Inv}(\mathcal{C})$ the subcategory generated by the invertible objects. Note that by Lemma \ref{dualgraphlemma}, if $\mathcal{C} $ realizes case (4), (5), or (6), then $\text{Inv}(\mathcal{C})=\text{Vec}_{\mathbb{Z}/4\mathbb{Z}}^\omega$ or  $\text{Inv}(\mathcal{C})=\text{Vec}_{\mathbb{Z}/2\mathbb{Z}\times \mathbb{Z}/2\mathbb{Z}}^\omega$, for some associator $\omega$.

\begin{lemma} \label{lemc}
Let $\mathcal{C}$ be a fusion category realizing case (4). 
\begin{enumerate}
\item The associator on $\text{Inv}(\mathcal{C}) $ is trivial.
\item There is a  $4$-dimensional division algebra $\gamma_4 $ in $\mathcal{C} $ such that $\mathcal{C}^{\gamma_4} $ realizes case (6).
\item For any  $2$-dimensional division algebra $\gamma_2 $ in $\mathcal{C} $, the fusion category $\mathcal{C}^{\gamma_2} $ realizes case (5).
\end{enumerate}
\end{lemma}
\begin{proof}
Since $\mathcal{C}$ realizes case (4),  it is the dual category to a module cateogry over $\mathcal{AH}_3$ which realizes the corresponding fusion module $M_{AH_3}^{(4)}$; therefore it contains a division algebra $  \theta_{43}^3 \overline{\theta_{43}^3}= \gamma $ of dimension $28d+4$ as in the discussion preceding Lemma \ref{div}. All objects in $\mathcal{C} $ have dimension equal to $d$ or to $1$, so $\gamma$ must contain four invertible objects.  
If $\psi$ is invertible, then since $\theta_{43}^3$ and $\psi \theta_{43}^3$ are simple, $(\psi, \gamma) = (\psi\theta_{43}^3 ,\theta_{43}^3) \leq 1$.  Hence, $\gamma$ must contain each of the invertible objects exactly once. The direct sum of the invertible objects then forms a subalgebra of $\gamma $ by Lemma \ref{div}.
\begin{enumerate}
\item Since there is an algebra structure on the sum of all four invertible objects, the associator on $\text{Inv}(\mathcal{C}) $ must be trivial. 
\item Let $\gamma_4$ be the subalgebra of $\gamma $ corresponding to the direct sum of the invertible objects. Then $C^{\gamma_4} $ also contains four invertible objects, so it is not equivalent to $\mathcal{AH}_1-\mathcal{AH}_3$ and therefore must realize one of the four cases (4)-(7). Moreover, by Lemma \ref{div} there must be a division algebra in $\mathcal{C}^{\gamma_4} $ of dimension $\frac{28d+4}{4}=7d+1$ whose category of bimodules is equivalent to $\mathcal{AH}_3 $. But of the $M_{AH_3}^{(i)}$ for $i = 4, 5, 6, 7$, only $M_{AH_3}^{(6)}$ has a basis element with dimension $\sqrt{7d+1} $.  Therefore, $\mathcal{C}^{\gamma_4} $ must realize case (6).

\item Since $\mathcal{C} $ realizes case (4), it is the dual category to a module category over $\mathcal{AH}_2$ which realizes the corresponding fusion module $L_{AH_2}^{(4)}$. Therefore, there is a division algebra in $\mathcal{C}$ with dimension $1+d$, and a non-invertible simple object $\rho$ in $\mathcal{C}$ such that $\mathcal{C}^{1+\rho }$ is equivalent to  $\mathcal{AH}_2 $. Let $\gamma_2 $ be a $2$-dimensional division algebra in $\mathcal{C} $. Since $(\gamma_2,1+\rho)=1 $, by Lemma \ref{div} there is a division algebra of dimension $2(1+d) $ in $\mathcal{C}^{\gamma_2} $ whose dual category is $\mathcal{AH}_2 $. But of the four choices for $L_{AH_2}^{(i)}$, only that of case (5) has a basis element with dimension $\sqrt{2(1+d)} $. Therefore $\mathcal{C}^{\gamma_2} $ must realize case (5).
\end{enumerate}

\end{proof}

\begin{lemma} \label{lema}
Let $\mathcal{C}$ be a fusion category realizing case (5). 
\begin{enumerate}
\item The associator on $\text{Inv}(\mathcal{C}) $ is non-trivial.
\item There is a  $2$-dimensional division algebra $\gamma_2 $ in $\mathcal{C} $ such that $\mathcal{C}^{\gamma_2} $ realizes case (4).
\end{enumerate}
\end{lemma}
\begin{proof}
Since $\mathcal{C} $ realizes case (5), it is the dual category to a module category over $\mathcal{AH}_2$ which realizes the corresponding fusion module $L_{AH_2}^{(5)}$.
\begin{enumerate}
\item Looking at the fusion module  $L_{AH_2}^{(5)}$, we see that $\mathcal{C}$ has a division algebra $\gamma $ with dimension $ 7d+1$ such that $\mathcal{C}^{\gamma} $ is equivalent to $\mathcal{AH}_2 $. Suppose that the associator on $\text{Inv}(\mathcal{C}) $ is trivial. Let $\gamma_4 $ be a simple algebra of dimension $4$ in $\mathcal{C} $. Since all simple objects in $\mathcal{C} $ have dimension $1$ or dimension $d$, $\gamma $ does not contain any non-trivial invertible objects which implies that $(\gamma,\gamma_4) =1$.   By Lemma \ref{div}, $\mathcal{C}^{\gamma_4} $ must have a simple algebra $\delta $ of dimension $4(7d+1)$ such that $(\mathcal{C}^{\gamma_4})^{\delta} $ is equivalent to  $\mathcal{AH}_2 $. Also, since $\mathcal{C}^{\gamma_4} $ has four invertible objects, it must realize one of the cases (4)-(7). But looking at the four possibilites for $L_{AH_2}^{(i)}$, we find that none of them contain a basis element with dimension $\sqrt{4(7d+1)} $. Therefore, $\gamma_4$ cannot exist and the associator on $\text{Inv}(\mathcal{C}) $ must be non-trivial.

\item Again looking at  $L_{AH_2}^{(5)}$, we see that there is a division algebra $\gamma$ in $\mathcal{C} $ of dimension $2(d+1)$ such that $\mathcal{C}^{\gamma} $ is equivalent to $\mathcal{AH}_2 $. Then $\gamma $ must contain $2$ invertible objects, and therefore contains a $2$-dimensional subalgebra $\gamma_2 $. Then by Lemma \ref{div}, $\mathcal{C}^{\gamma_2} $ has a division algebra $\gamma' $ of dimension $ d+1$  such that $(\mathcal{C}^{\gamma_2} )^{\gamma'}$ is equivalent to  $ \mathcal{AH}_2$. Therefore, $\mathcal{C}^{\gamma_2} $ realizes case (4).
\end{enumerate}
\end{proof}

\begin{lemma}\label{lemb}
Let $\mathcal{C}$ be a fusion category realizing case (6). 
\begin{enumerate}
\item The associator on $\text{Inv}(\mathcal{C}) $ is trivial.
\item There is a  $4$-dimensional division algebra $\gamma_4 $ in $\mathcal{C} $ such that $\mathcal{C}^{\gamma_4} $ realizes case (4).
\end{enumerate}
\end{lemma}
\begin{proof}
The proof is similar to the previous lemmas. (Namely, look at the fusion module $L_{AH_2}^{(6)}$  to find a division algebra in $ \mathcal{C}$ of dimension $4(d+1)$ which has a $4$-dimensional subalgebra $\gamma_4 $. Then there is a division algebra $\gamma' $ in $\mathcal{C}^{\gamma_4} $ of dimension $d+1$ such that $\mathcal{C}^{\gamma'} $ is equivalent to $\mathcal{AH}_2 $.) 
\end{proof}

\begin{lemma}
\begin{enumerate}
\item If $\mathcal{C} $ is a fusion category realizing case (4) or case (6), then $\text{Inv}(\mathcal{C}) $ is equivalent to $Vec_{\mathbb{Z}/4\mathbb{Z}} $ with trivial associator.
\item  If $\mathcal{C} $ is a fusion category realizing case (5), then $\text{Inv}(\mathcal{C}) $ is equivalent to $Vec_{\mathbb{Z}/2\mathbb{Z}\times\mathbb{Z}/2\mathbb{Z}}^\xi $.
\end{enumerate}
\end{lemma}
\begin{proof}
\begin{enumerate}
\item  First assume that $\mathcal{C} $ realizes case (4). By Lemma \ref{lemc}, $ \text{Inv}(\mathcal{C}) $ has a trivial associator, but we need to determine whether the group is cyclic or not.  Furthermore, we know that there is a $2$-dimensional algebra in $\mathcal{C}$ which gives a Morita equivalence between the category $\text{Inv}(\mathcal{C})$ and $\text{Inv}(\mathcal{D})$ where $\mathcal{D} $  is a fusion category realizing case (5).  But $\text{Inv}(\mathcal{D})$ has nontrivial associator by Lemma \ref{lema}.  In $\text{Vec}_{\mathbb{Z}/2\mathbb{Z} \times \mathbb{Z}/2\mathbb{Z}}$ any $2$-dimensional algebra has dual category $\text{Vec}_{\mathbb{Z}/2\mathbb{Z} \times \mathbb{Z}/2\mathbb{Z}}$ with trivial associator,  thus $\text{Inv}(\mathcal{C}) $ must be equivalent to $\text{Vec}_{\mathbb{Z}/4\mathbb{Z}}$.

Now assume $\mathcal{C}$ realizes case (6). Then by Lemma \ref{lemb}, $\text{Inv}(\mathcal{C}) $ is Morita equivalent to $\text{Inv}(\mathcal{D}) $, where $\mathcal{D} $ is a fusion category realizing case (4) and the Morita equivalence comes from a $4$-dimensional division algebra. Therefore $\text{Inv}(\mathcal{C}) $ is also equivalent to $\text{Vec}_{\mathbb{Z}/4\mathbb{Z}} $.

\item By Lemma \ref{lema},  $\text{Inv}(\mathcal{C}) $ is Morita equivalent to $\text{Inv}(\mathcal{D}) $, where $\mathcal{D} $ realizes case (4) and the Morita equivalence comes from a $2$-dimensional algebra.  Since $\text{Inv}(\mathcal{C}) $ is equivalent to  $\text{Vec}_{\mathbb{Z}/4\mathbb{Z}} $ with trivial associator, $\text{Inv}(\mathcal{C}) $ is equivalent to $\text{Vec}_{\mathbb{Z}/2\mathbb{Z} \times \mathbb{Z}/2\mathbb{Z}}^\xi$ by \cite{MR2362670}.
\end{enumerate}
\end{proof}

Now we turn our attention to the non-invertible objects.

\begin{lemma} \label{graphpair}
Let $\mathcal{C} $ be a fusion category realizing case (4). Then all non-invertible simple objects in $\mathcal{C} $ are self-dual.
\end{lemma}

\begin{proof}
This can be deduced by looking at the principal/dual graph pair for $\theta_{24}^{1}$ in $L_{AH_2}^{(4)}$, where the top row corresponds to objects in $\mathcal{C}$ and the bottom row corresponds to objects in $\mathcal{AH}_2 $ (we use the label $\rho' $ in $\mathcal{AH}_2 $ to distinguish the object from the similarly named $\rho $ in $\mathcal{C} $):

$$\begin{tikzpicture}
\draw [fill] (0,0) circle [radius=.05];
\node [below] at (0,0) {$1$};
\draw [fill] (1,0) circle [radius=.05];
\node [below] at (1,0) {$\alpha$};
\draw [fill] (2,0) circle [radius=.05];
\node [below] at (2,0) {$\alpha \pi$};
\draw [fill] (3,0) circle [radius=.05];
\node [below] at (3,0) {$\pi$};
\draw [fill] (4,0) circle [radius=.05];
\node [below] at (4,0) {$\eta$};
\draw [fill] (5,0) circle [radius=.05];
\node [below] at (5,0) {$\rho' \alpha$};
\draw [fill] (6,0) circle [radius=.05];
\node [below] at (6,0) {$\alpha \rho' \alpha$};
\draw [fill] (7,0) circle [radius=.05];
\node [below] at (7,0) {$\rho'$};
\draw [fill] (8,0) circle [radius=.05];
\node [below] at (8,0) {$\alpha \rho'$};

\draw [fill] (1.5,2) circle [radius=.05];
\node [left] at (1.5,2) {$\theta_{42}^1$};
\draw [fill] (2.5,2) circle [radius=.05];
\node [left] at (2.5,2) {$\theta_{42}^2$};
\draw [fill] (3.5,2) circle [radius=.05];
\node [left] at (3.5,2) {$\theta_{42}^3$};
\draw [fill] (4.5,2) circle [radius=.05];
\node [left] at (4.5,2) {$\theta_{42}^4$};
\draw [fill] (5.5,2) circle [radius=.05];
\node [left] at (5.5,2) {$\theta_{42}^5$};
\draw [fill] (6.5,2) circle [radius=.05];
\node [left] at (6.5,2) {$\theta_{42}^6$};

\draw [fill] (.5,4) circle [radius=.05];
\node [above] at (.5,4) {$1$};
\draw [fill] (1.5,4) circle [radius=.05];
\node [above] at (1.5,4) {$\alpha_1$};
\draw [fill] (2.5,4) circle [radius=.05];
\node [above] at (2.5,4) {$\alpha_3$};
\draw [fill] (3.5,4) circle [radius=.05];
\node [above] at (3.5,4) {$\alpha_2$};
\draw [fill] (4.5,4) circle [radius=.05];
\node [above] at (4.5,4) {$\rho$};
\draw [fill] (5.5,4) circle [radius=.05];
\node [above] at (5.5,4) {$\alpha_1 \rho$};
\draw [fill] (6.5,4) circle [radius=.05];
\node [above] at (6.5,4) {$\alpha_3 \rho$};
\draw [fill] (7.5,4) circle [radius=.05];
\node [above] at (7.5,4) {$\alpha_2 \rho$};

\path (0,0) edge (1.5,2);
\path (2,0) edge (1.5,2);
\path (2,0) edge (5.5,2);
\path (2,0) edge (6.5,2);
\path (3,0) edge (5.5,2);
\path (4,0) edge (5.5,2);
\path (5,0) edge (5.5,2);
\path (6,0) edge (5.5,2);
\path (3,0) edge (6.5,2);
\path (4,0) edge (6.5,2);
\path (7,0) edge (6.5,2);
\path (8,0) edge (6.5,2);
\path (4,0) edge (2.5,2);
\path (4,0) edge (3.5,2);
\path (1,0) edge (4.5,2);
\path (3,0) edge (4.5,2);

\path (1.5,2) edge (.5,4);
\path (2.5,2) edge (1.5,4);
\path (3.5,2) edge (2.5,4);
\path (4.5,2) edge (3.5,4);
\path (1.5,2) edge (4.5,4);
\path (2.5,2) edge (5.5,4);
\path (3.5,2) edge (6.5,4);
\path (4.5,2) edge (7.5,4);
\path (5.5,2) edge (4.5,4);
\path (5.5,2) edge (5.5,4);
\path (5.5,2) edge (6.5,4);
\path (5.5,2) edge (7.5,4);
\path (6.5,2) edge (4.5,4);
\path (6.5,2) edge (5.5,4);
\path (6.5,2) edge (6.5,4);
\path (6.5,2) edge (7.5,4);

\end{tikzpicture}$$

%$$\vpic{newgraph2} {4in}   .$$

By the usual associativity condition (cf. \cite[\S3.1]{1007.1730}), for any two vertices $\theta^i,\theta^j$ on the middle row, the number of paths between $\theta^i$ and $\theta^j$ formed by taking an edge from $\theta^i$ to the top row, switching to the dual of that vertex in the top row, and then taking an edge down to $\theta^j$, must be the same as the number of similar paths taken through the bottom row.

Since the two vertices connected to $\theta_{42}^2$ in the bottom row ($\alpha$ and $\eta$) both correspond to selfdual objects in $\mathcal{AH}_2$, there are exactly two paths from $\theta_{42}^2$ to itself through the bottom graph.  Thus there must also be two paths from $\theta_{42}^2$ to itself through the top graph.  The only way this can happen is if $\alpha_2 \rho$ is self-dual.  Similarly, looking at $\theta_{42}^1$, we see that $\rho$ must be self-dual.   There is exactly one path through the bottom graph between $\theta_{42}^2$ and $\theta_{42}^3$, so there must also be only one graph between them in the top graph.  Since $\alpha_1$ and $\alpha_3$ are dual to each other there is a path going through those vertices.  Thus there can not be a path going through $\alpha_1 \rho$ and $\alpha_3 \rho$, and so those vertices can not be dual to each other and so must be self-dual.
%Since the two vertices connected to $M$ in the bottom row, $u$ and $w$, both correspond to self-dual objects in $\mathcal{AH}_2 $, the two vertices connected to $M$ in the top row, $c$ and $g$, must correspond to self-dual objects as well. Similarly, $d$ and $h$ must correspond to self-dual objects. Then since the invertible objects in $\mathcal{C} $ form the category $Vec_{\mathbb{Z}/4\mathbb{Z}} $,  the invertible objects corresponding to $e$ and $f$ must be dual to each other. Then the objects corresponding to $a$ and $b$ must be self-dual, in order for $K$ and $L$ to each have a path to itself through the top row.
\end{proof}

By Lemma \ref{lemah4}, this completely determines the fusion ring of any fusion category realizing case (4); we call this fusion ring $AH_4$. Let $AH_5 $ be the fusion ring satisfying the fusion rules of Lemma \ref{lemah4} for $G=\mathbb{Z}/2\mathbb{Z}  \times \mathbb{Z}/2\mathbb{Z}  $ with all non-invertible basis elements self-dual. 

%
%
%
%Finally, we determine the fusion ring of $\mathcal{AH}_5 $. There are two possibilites: either all four non-invertible objects are self-dual or two of them are dual to each other. We compute all the fusion modules for each of these two rings and then all of the fusion bimodules with $AH_{1-3} $. Only one of these choices has a fusion bimodule containing $16_3 $.

\begin{lemma}
\begin{enumerate}
\item The fusion ring of any fusion category realizing cases (4) or (6) is $AH_4$.
\item The fusion ring of any fusion category realizing case (5) is $AH_5$.
\end{enumerate}
\end{lemma}
\begin{proof}
For case (5), a similar argument as in Lemma \ref{graphpair} works to show that all the non-invertible objects are self-dual. For case (6) it is a bit more complicated but it can be checked that if two of the non-invertible objects were dual to each other, the fusion ring would not admit a fusion bimodule with $AH_3 $ realizing the fusion module $M_{AH_3}^{(6)}$.
\end{proof}
%We denote the fusion ring of $\mathcal{AH}_5 $ by $AH_5$. 
%
\begin{remark}
Using a computer we can compute all the fusion modules over $AH_4$, of which there are $28$ up to isomorphism.  Similarly, there are $126$ fusion modules over $AH_5$.  We initially found many of the results below by doing computer case analysis for all these fusion modules similar to the arguments in \cite{GSbp}.  In order to make this paper widely accessible, we have not used these computer arguments below and instead given alternate by-hand proofs that don't split up into so many cases.
\end{remark}

\subsection{Uniqueness for cases (4)-(6)}

\begin{theorem}\label{unq}
If there exists a fusion category realizing any of cases (4)-(6), then there exists a unique fusion category realizing each of these three cases.
\end{theorem} 

\begin{proof}
If case (4) can be realized then so can case (5) and (6) as bimodules over algebras of dimension $2$ and $4$ respectively.  Similarly, if case (5) or (6) can be realized, then so can (4) as bimodules over algebras of dimension $2$ and $4$ respectively.  So we need only show that there is at most one fusion category realizing each case.

First we consider case (4). Let $\mathcal{C}_1, \mathcal{C}_2 $ be fusion categories realizing case (4), and let ${}_{\mathcal{C}_1} \mathcal{K} {}_{\mathcal{C}_2}$ be a Morita equivalence between them. Then there exist Morita equivalences ${}_{\mathcal{C}_1}  \mathcal{K}_1 {}_{\mathcal{AH}_2}$ and ${}_{\mathcal{AH}_2} \mathcal{K}_2 {}_{\mathcal{C}_2}$ and a bimodule equivalence ${}_{\mathcal{C}_1} \mathcal{K}_1 {}_{\mathcal{AH}_2} \boxtimes_{\mathcal{AH}_2} {}_{\mathcal{AH}_2} \mathcal{K}_2 {}_{\mathcal{C}_2} \rightarrow {}_{\mathcal{C}_1} \mathcal{K} {}_{\mathcal{C}_2}$. Since any bimodule between $\mathcal{AH}_2 $ and  $\mathcal{C}_1$ or $\mathcal{C}_2$ corresponds to the fusion module $L_{AH_2}^{(4)}$, there are simple objects $\kappa_1 \in \mathcal{K}_1, \kappa_2 \in \mathcal{K}_2 $ corresponding to the basis element $\theta_{42}^1$ such that $\bar{\kappa_1} \kappa_1 = 1+\rho = \kappa_2 \bar{\kappa_2}$ in $\mathcal{AH}_2$.
By Frobenius reciprocity, $\kappa_1 \kappa_2$ in $\mathcal{K} $ also has two simple summands and dimension $1+d $. Since each simple object in $\mathcal{K} $ has dimension equal to $\sqrt{i+jd } $ for some non-negative integers $i \leq 4, \ j \leq 32$, this is only possible if the dimensions of the two simple summands are $1$ and $d$, or $2$ and $d-1$.  In the latter case, the internal endomorphisms of the $2$-dimensional object, would give a $4$-dimensional division algebra in $\mathcal{C}_1$ whose dual was $\mathcal{C}_2$.  But this is impossible because such a dual category would have to realize case (6) and not case (4).  Thus, $\mathcal{K}$ has an object of dimension one whose internal endomorphisms must be the trivial algebra.  Hence, $\mathcal{C}_2$ is equivalent to the category of bimodules over the trivial algebra in $\mathcal{C}_1$ and thus is equivalent to $\mathcal{C}_1$.

It follows that cases (5) and (6) are also each realized by at most one fusion category.  Any fusion category $\mathcal{C} $ which realizes case (6) has a $4$-dimensional division algebra $ \gamma_4$ such that $\mathcal{C}^{\gamma_4} $ realizes case (4).  In particular, $\mathcal{C}$ must be the category of bimodules for a $4$-dimensional algebra in a the unique category realizing case (4) and thus is itself unique. Similarly, any fusion category $ \mathcal{C}$ realizing case (5) has a $2$-dimensional algebra $ \gamma_2$ such that $\mathcal{C}^{\gamma_2} $ realizes case (4), so it also must be the category of bimodules for the unique $2$-dimensional algebra in the unique fusion category realizing case (4).
\end{proof}

We will call the three (thus far hypothetical, but unique if they exist) fusion categories which realize these three cases $\mathcal{AH}_4$, $\mathcal{AH}_5$, and $\mathcal{AH}_6$.

\begin{theorem} \label{autothm}
If $\mathcal{AH}_{4-6} $ exist, then every Morita autoequivalence of each of these fusion categories is realized by an outer automorphism.
\end{theorem}

\begin{proof}
We first consider $\mathcal{AH}_4 $, which has fusion ring $ AH_4$. As in the proof of the previous lemma, any Morita autoequivalence must come from the trivial algebra, and is therefore realized by an outer automorphism. 

Next we consider $ \mathcal{AH}_5$. There is a Morita equivalence between $\mathcal{AH}_4 $ and $\mathcal{AH}_5 $ which is implemented by a $2$-dimensional algebra. Since every Morita autoequivalence of  $\mathcal{AH}_4 $ is realized by an outer automorphism, in fact every $\mathcal{AH}_4 -\mathcal{AH}_5 $  autoequivalence is implemented by a $2$-dimensional algebra. Again as in the proof of the previous lemma, by Frobenius reciprocity this means that every Morita autoequivalence of $\mathcal{AH}_5 $ contains an object with dimension $2$ and two simple summands. This means the summands each have dimension $1$, and the autoequivalence comes from the trivial module category, and is therefore realized by an outer automorphism. For $\mathcal{AH}_6 $ we use the same argument, starting with the $\mathcal{AH}_5 -\mathcal{AH}_6 $  Morita equivalence coming from a $2$-dimensional algebra.
\end{proof}
\begin{corollary}\label{outgroup}
If $\mathcal{AH}_{4-6} $ exist, then $\text{Out}(\mathcal{AH}_i)\cong \mathbb{Z} /2\mathbb{Z} \times  \mathbb{Z} /2\mathbb{Z} $ for $i =4,5,6 $.
\end{corollary}

\subsection{Non-existence for case (7)}
In this subsection we show that if $\mathcal{AH}_{4-6}$ exist (as we will prove in Section \ref{sec:main}), then case (7) cannot be realized.

\begin{lemma}\label{15d}
Let $\mathcal{C} $ be a fusion category realizing case (7). Then every invertible $ \mathcal{C}-\mathcal{AH}_5$ bimodule category has a simple object of dimension $\sqrt{1+5d} $.
\end{lemma}

\begin{proof}
Every invertible $ \mathcal{C}-\mathcal{AH}_5$ bimodule category $\mathcal{K} $ can be expressed as a relative tensor product $\mathcal{K}=\mathcal{K}_1 \boxtimes_{\mathcal{AH}_3}  \mathcal{K}_2$  of an invertible $\mathcal{C}-\mathcal{AH}_3$ bimodule category $\mathcal{K}_1$ and an invertible $\mathcal{AH}_3-\mathcal{AH}_5 $ bimodule category $\mathcal{K}_2$. Since $\mathcal{C} $ realizes case (7) and $\mathcal{AH}_5 $ realizes case (5), there are simple objects $\theta_{73}^1  \in \mathcal{K}_1$ and $\overline{\theta_{53}^1} \in \mathcal{K}_2$ such that $\overline{\theta_{73}^1} \theta_{73}^1 = 1+\beta \xi+\xi \beta+\mu  $ and $\theta_{53}^1 \overline{\theta_{53}^1}= 1+ \mu $ (both objects in $\mathcal{AH}_3 $). Therefore by Frobenius reciprocity, the object $\theta_{73}^1 \overline{\theta_{53}^1} \in \mathcal{K}$ has two simple summands, and its dimension is
$ \sqrt{(2+2d)(1+d)}=\sqrt{2}(1+d)$. Since every object in $\mathcal{K} $ must have dimension
$\sqrt{i + j d} $ for nonnegative integers $i,j$, and neither summand can have integer dimension, 
the only possibility is that each summand has dimension $\sqrt{1+5d} = \frac{\sqrt{2}(1+d)}{2}$.
\end{proof}

In the following two lemmas we need to discuss properties of $\mathcal{AH}_5 $. We label the simple objects by $\alpha_{g} $ and $\alpha_{g} \rho $, $g \in G=\mathbb{Z}/2\mathbb{Z} \times  \mathbb{Z}/2\mathbb{Z} $, where $\rho $ is a non-invertible simple object.

\begin{lemma}\label{algah5}
The object $1+ \alpha_g \rho $ in $\mathcal{AH}_5 $ admits two algebra structures whose dual categories are  $\mathcal{AH}_3$ for each $g \in G $.
\end{lemma}
\begin{proof}
Since $\mathcal{AH}_5 $ satisfies case (5), each of the four $\mathcal{AH}_5-\mathcal{AH}_3$ Morita equivalences is implemented by an algebra of dimension $1+d$ in $\mathcal{AH}_5$. Let $1+\alpha_{g_i} \rho $, $1 \leq i \leq 4 $ be algebras corresponding to the four Morita equivalences, and suppose that $g_i=g_j $ for some $i \neq j $. Then as in the proof of Theorem \ref{unq}, we can compose the two bimodule objects corresponding to the algebras $1+\alpha_{g_i} \rho $ and $1+\alpha_{g_j} \rho $ to obtain an object in a $\mathcal{AH}_3 $-Morita auto-equivalence which has dimension $1+d$ and two simple summands. Moreover, the Morita autoequiavelence must be nontrivial since $i \neq j $. However, there are no such objects in the non-trivial Morita autoequivalences of $\mathcal{AH}_3 $. Therefore all of the $g_i $ are distinct, and each $1+\alpha_{g_i} \rho $ admits an algebra structure whose dual category is $\mathcal{AH}_3 $. Since the fusion module $M_{AH_3}^{(5)}$ contains $4$ distinct basis elements of dimension $\sqrt{1+d} $, which form two equivalence classes under the action of the group of invertible objects of $\mathcal{AH}_3 $, there are in fact two such algebras for each $g \in G $.
\end{proof}
\begin{corollary}\label{outah5}
The outer automorphism group of $\mathcal{AH}_5 $ acts transitively on the noninvertible simple objects.
\end{corollary}
\begin{remark}
The situation is quite different for $ \mathcal{AH}_4$. If we let $\beta_h $, $h \in H=\mathbb{Z}/4\mathbb{Z} $ be the invertible objects of  $ \mathcal{AH}_4$ and $\xi $ be a noninvertible simple object, then two of the $\mathcal{AH}_4-\mathcal{AH}_2$ Morita equivalences correspond to pairs of algebra structures on $1+\xi $ and $ 1+ \beta_2 \xi $ and the other two Morita equivalences correspond to pairs of algebra structures on  $1+\beta_1 \xi $ and $ 1+ \beta_3 \xi $. The argument in the proof of the previous lemma fails here since, unlike for $\mathcal{AH}_3$, there is a nontrivial Morita autoequivalence of $\mathcal{AH}_2 $ with an object of dimension $1+d$ and two simple summands. Specifically, there is a Morita autoequivalence of $\mathcal{AH}_2 $ with simple objects of dimension $\frac{d+1}{2},\frac{d+1}{2},\frac{d+1}{2},\frac{d+1}{2},d-1,d-1,d+1 $ (see \cite{GSbp}).
\end{remark}

\begin{theorem} \label{no7}
If $\mathcal{AH}_{4-6} $ exist, there do not exist any fusion categories in the Morita equivalence class of the $\mathcal{AH}_i $ realizing case (7).
\end{theorem}
\begin{proof}
Let $\mathcal{C} $ be a fusion category realizing case (7). By Lemma \ref{15d}, there exists an object $\kappa $ of dimension $\sqrt{1+5d} $ in an 
invertible $ \mathcal{C}-\mathcal{AH}_5$ bimodule category. 
Then $\bar{\kappa} \kappa$ contains five summands with dimension $d$, so at least one simple summand has multiplicity at least $2$. Since the outer automorphism group acts transitively on the noninvertible simple objects,  we may assume without loss of generality that that summand is $\rho$, and we have
 $$(\bar{\kappa} \kappa, \rho) \geq 2 .$$ 

Let  $\theta^1_{53}  $ be a simple object in an invertible $ \mathcal{AH}_5-\mathcal{AH}_3$ bimodule category whose internal end is $1+\rho$. Then $$(\kappa \boxtimes \theta^1_{53},\kappa \boxtimes \theta^1_{53})=(\bar{\kappa}\kappa, \theta^1_{53}\overline{ \theta^1_{53}}) \geq 3$$ and $$d(\kappa \boxtimes \lambda)=\sqrt{(1+d)(1+5d)}=\sqrt{46d+6}  .$$ However, the dimensions of the basis elements in  $N_{AH_3} $ are 
 $$d(\theta^1_{73})=d(\theta^2_{73}) =\sqrt{2d+2},  \quad d(\theta^3_{73})=\sqrt{28d+4} .$$ 
 Therefore there cannot be an object $ \mathcal{C}-\mathcal{AH}_3$ bimodule category with at least three simple summands and dimension $\sqrt{46d+6}=d(\theta^1_{73})+d(\theta^3_{73})$.
%
%
%First note that $(\kappa, \kappa \alpha_g \rho) \geq 1 $ for all $g \in G $, since otherwise by Lemma \ref{algah5} we could compose $\kappa $ with an object $\lambda $ whose internal end is $1+\alpha_g \rho $ to obtain an irreducible object in an $\mathcal{C}-\mathcal{AH}_3 $ bimodule category with dimension $$\sqrt{(1+d)(1+5d)} =\sqrt{46d+6}.$$ But the basis elements in the fusion module $N_{AH_3} $ corresponding to case (7) have dimensions $$d(\theta^1_{73})=d(\theta^2_{37}) =\sqrt{2d+2},  \quad d(\theta^3_{37})=\sqrt{28d+4} ,$$ so this is impossible.
%Then since the outer automorphism group acts transitively on the noninvertible simple objects,  we may write without loss of generality $$\bar{\kappa} \kappa=1+ \rho +\sum_{g \in G} \limits \alpha_g \rho .$$ Let  $\lambda  $ be a simple object in an invertible $ \mathcal{AH}_5-\mathcal{AH}_3$ bimodule category whose internal end is $1+\rho$. Then $(\kappa \lambda,\kappa \lambda )=(\bar{\kappa}\kappa,\lambda\bar{\lambda})=3$ and and $d(\kappa \lambda)=\sqrt{(1+d)(1+5d)}  $. Again looking at the dimensions of the basis elements in  $N_{AH_3} $, we see that there cannot be an object $ \mathcal{C}-\mathcal{AH}_3$ bimodule category with three simple summand and this dimension.

\end{proof}

\section{Existence of $\mathcal{AH}_{4-6} $} \label{sec:main}
In the previous section we established that if any of cases (4)-(6) of Theorem \ref{cases} occur, then there are exactly three additional fusion categories in the Morita equivalence class of $\mathcal{AH}_{1-3} $, which we called $\mathcal{AH}_{4-6} $.
In this section we will show that $\mathcal{AH}_{4-6} $ exist by explicitly constructing $\mathcal{AH}_{4} $ as a category of finite-index endomorphisms of a Type III factor $M$. By Lemmas \ref{lemah4}, \ref{lemc},  and \ref{graphpair},  $\mathcal{AH}_{4}$ must contain a copy of $\text{Vec}_{\mathbb{Z}/4\mathbb{Z}} $, with simple objects $\{\alpha_g\}_{g \in \mathbb{Z}/4\mathbb{Z}} $, along with a self-dual simple object $\rho $, satisfying the fusion rules  
$$\alpha_g \rho \cong \rho \alpha_{-g}, \ \forall g \in \mathbb{Z}/4\mathbb{Z},   \quad \text{and }\rho^2\cong 1+2\sum_{g \in \mathbb{Z}/4\mathbb{Z}} \alpha_g \rho .$$
A general theory of such categories for arbitrary finite Abelian groups, but with the multiplicities of $\alpha_g \rho $ in $\rho^2 $ all equal to $1$ instead of $2$, has been developed by the second-named author in \cite{MR1228532,  MR1832764,IzumiNote}. He determined a complete numerical invariant  of such categories given by solutions of certain polynomial equations determined by the group $G$; these solutions give structure constants for endomorphisms $\{\alpha_g\}_{g \in G}$ and $ \rho $ of the Cuntz algebra $\mathcal{O}_{|G|+1} $, which can then be extended to endomorphisms of a Type III factor.  In this construction one Cuntz algebra generator belongs to each intertwiner space $(\alpha_g,\rho^2) $
and one generator belongs to $(1,\rho^2) $.

We could try to perform a similar construction for $\mathcal{AH}_{4} $ by realizing $\rho $ as an endomorphism of $\mathcal{O}_{2|G|+1}=\mathcal{O}_9 $, again with one generator belonging to $(1,\rho^2) $, but this time with two generators belonging to each $(\alpha_g,\rho^2) $. However, because the intertwiner spaces are two-dimensional, it is difficult to write down equations for the structure constants of $\rho $. We therefore take a different approach, and construct instead an equivariantization of $\mathcal{AH}_{4} $, which is technically easier to deal with. We then recover $\mathcal{AH}_{4} $ by de-equivariantizing.

\subsection{Equivariantization and the orbifold construction}
We recall the notion of equivariantization of a fusion category by the action of a finite group of automorphisms. For simplicity we define equivariantization for strict categories, which include categories of endomorphisms (though as with all constructions for monoidal categories, strictness is not really important).
Let $\mathcal{C} $ be a strict fusion category, and let $G$ be a finite group acting by automorphisms  $\{ \phi_g  \}_{g \in G}$  on $\mathcal{C} $. A $G$-equivariant object of $\mathcal{C} $ is an object $\xi \in \mathcal{C} $, together with isomorphisms $\{ u_g: \phi_g(\xi) \rightarrow \xi \}_{g \in G}$, such that $$u_g \phi_g(u_h)=u_{gh}    . $$
Morphisms and tensor product for $G$-equivariant objects can be defined in a natural way, and the $G$-equivariant objects form a fusion category, called the $G$-equivariantization of $\mathcal{C} $.   

Suppose  $ \mathcal{C}$ is a fusion category of finite-index endomorphisms of a Type III factor $M$, and suppose $\alpha $ is an order two automorphism of $M$  which commutes with a set of representative endomorphisms of the simple objects of $\mathcal{C} $. Since $\mathcal{C} $ is equivalent to its full subcategory of endomorphisms which commute with $\alpha $, we assume without loss of generality that $\alpha $ commutes with all endomorphisms in $\mathcal{C} $. Then $\alpha $ induces a $\mathbb{Z}/2\mathbb{Z} $-action on $\mathcal{C} $, where the action of $\alpha $ fixes objects and takes each morphism $u \in (\xi,\sigma) $ to $\alpha(u) \in (\xi, \sigma ) $. The equivariant objects with respect to this action are of the form $(\xi,u)$, with $u \in (\xi , \xi )$ satisfying $u \alpha(u)=1$. In particular for each simple object $\xi  \in \mathcal{C}$, there are two inequivalent objects in the $\mathbb{Z}/2\mathbb{Z} $-equivariantization, corresponding to $u=\pm1$.

Now suppose that $\alpha$ is not in the same sector as any object of $\mathcal{C} $.  Then the $ \mathbb{Z}/2\mathbb{Z}$-equivariantization of $\mathcal{C} $ induced by $\alpha $ can be realized as a category of endomorphisms on a larger factor via an orbifold construction, as follows. Let $M_1$ be the crossed product $M \rtimes_{\alpha}  \mathbb{Z}/2\mathbb{Z}$, which is the von Neumann algebra generated by $M  $ and a self-adjoint unitary $\lambda  $ satisfying $$\lambda x \lambda=\alpha(x), \ \forall x \in M.$$
 Then every equivariant object $(\xi,u)$ can be extended to an endomorphism $\tilde{\xi} $ on  $M_1$ by setting
 $\xi(\lambda)=u \lambda$, and the morphims between pairs of such extended endomorphisms $\tilde{\xi}_1 $ and $\tilde{\xi}_2 $  in $ \text{End}_0(M_1) $ are precisely the equivariant morphisms between $\xi_1 $ and $ \xi_2$ in $\text{End}_0(M) $.

\subsection{ $\mathbb{Z}/2\mathbb{Z}$-equivariantization of $\mathcal{AH}_4 $}
This section is purely motivational and explains why if $\mathcal{AH}_4$ exists, then it must have a $\mathbb{Z}/2\mathbb{Z}$-equivariantization $\mathcal{S}$ which is easier to understand.  Logically this is not necessary, since we will directly construct this $\mathcal{S}$ and use it to construct $\mathcal{AH}_4$ via the inverse process of de-equivariantization.

The following can be shown with a considerable amount of work.  

\begin{lemma}\label{lemmaeq}
Suppose $\mathcal{AH}_4  $ exists.  Let $H=\mathbb{Z}/4 \mathbb{Z} \times \mathbb{Z}/2\mathbb{Z} $. Then the $\mathbb{Z}/2\mathbb{Z}$-equivariantization with respect to the action coming from conjugation by $\alpha_2 $ contains a copy of $\text{Vec}_{H} $, with simple objects $
\{\alpha_g\}_{g \in H} $, along with a simple object $\rho $, satisfying the fusion rules  $$\alpha_g \rho \cong \rho \alpha_{-g}, \ \forall g \in H,   \quad \text{and }\rho^2\cong 1+\sum_{g \in H} \alpha_g \rho .$$

\end{lemma}

% For simple objects in the equivariantization, corresponding to $u_{\text{Ad} \alpha}=\pm1 $. 
%

We do not prove this lemma here, since the result is not required for anything that follows and the proof is difficult and detailed. Rather the lemma functions as motivation for constructing $\mathcal{AH}_4  $ by first constructing its $\mathbb{Z}/2\mathbb{Z}$-equivariantization. However, we note several features of this proof which will motivate some choices in the next section.

\begin{remark}
\begin{enumerate}
\item The proof proceeds by realizing (the hypothetical) $\mathcal{AH}_4  $ as the fusion category generated by automorphisms $\{\alpha_g\}_{g \in \mathbb{Z}/4\mathbb{Z}} $, along with an endomorphism $\rho $, on a Type III facor $M$ containing a Cuntz algebra $\mathcal{O}_9 $, with one generator belonging to $(1,\rho^2) $
and two generators belonging to each $(\alpha_g \rho,\rho^2) $.  Then proving the lemma amounts to showing that the eigenvalues of the action of $\alpha_2 $ on each $(\alpha_g \rho,\rho^2) $ are $\{1,-1\} $. 

\item Finding the eigenvalues of $\alpha_2 $ required using properties of $\mathcal{AH}_4  $ beyond just the fusion rules and the presence of  a $\text{Vec}_{ \mathbb{Z}/4\mathbb{Z}} $ subcategory, namely that $1+ \rho $ admits a Q-system, that the dual category with respect to this Q-system contains a non-trivial invertible object, and that the center of $\mathcal{AH}_4  $ contains no non-trivial invertible objects (and therefore conjugation by $\alpha_2 $ is a non-trivial automorphism). These properties follow from the fact that $ \mathcal{AH}_4$ is assumed to be the dual category of $\mathcal{AH}_2 $ with repect to a Q-system of dimension $1+d$.

\item Since $\alpha_2 $ belongs to $\mathcal{AH}_4  $, we cannot simply use a crossed product by $\alpha_2 $ for the orbifold construction, as that would result in $\alpha_2 $ being identified with the identity automorphism on the Type III factor $M$. Rather, we need to twist $\alpha_2$ by an automorphism of $M$ to obtain an automorphism $ \alpha'$ which is not in the same sector as any automorphism of $\mathcal{AH}_4 $, but such that conjugation by $\alpha' $ is equivalent to conjugation by $\alpha_2 $ as an automorphism of $\mathcal{AH}_4  $.

\item From the orbifold construction, we can deduce some important structure constants of the 
$\mathbb{Z}/2\mathbb{Z}$-equivariantization, which will be useful in the following subsection.

\end{enumerate}
\end{remark}

\subsection{Constructing $\mathcal{S}$}

By Lemma \ref{lemmaeq}, the $\mathbb{Z}/2\mathbb{Z}$-equivariantization of $\mathcal{AH}_4$ would have nicer fusion rules than $\mathcal{AH}_4 $ itself. We now directly construct a subfactor $\mathcal{S}$ whose principal even half has these fusion rules.

We recall the following construction from \cite{IzumiNote}, which completely characterizes unitary fusion categories which are tensor generated by $\text{Vec}_G $ for a finite Abelian group $G$ (which for us is $\mathbb{Z}/4\mathbb{Z} \times \mathbb{Z}/2\mathbb{Z}$) and a simple object $ \rho$, satisfying  $$\alpha_g \rho \cong \rho \alpha_{-g}, \ \forall g \in G, \quad \text{and }\rho^2\cong 1+\sum_{g \in G} \alpha_g \rho .$$

Let $G$ be a finite Abelian group, and let $$A: G \times G \times G \rightarrow \mathbb{C} $$ $$\epsilon: G \times G \rightarrow \{-1,1\}$$ 
$$\eta: G \rightarrow \{  1, e^{\frac{2\pi i }{3}}, e^{\frac{-2\pi i }{3} }  \} $$ be functions. Set $A_g(h,k) =A(g,h,k)$, $\epsilon_g(h)=\epsilon(g,h)$, and $\eta_g=\eta(g) $ for $g,h,k \in G $.  Let $n=|G| $ and let $$d=\displaystyle \frac{n+\sqrt{n^2+4}}{2} .$$ 
Consider the following system of equations:

\begin{equation} \label{e1}
\epsilon_{h+k}(g)=\epsilon_h(g)\epsilon_k(g+2h), \ \epsilon_h(0)=1
\end{equation}
\begin{equation}
\eta_{g+2h}=\eta_g, \ \eta_g^3=1
\end{equation}
\begin{equation}\label{e3}
\sum_{h \in G} A_g(h,0)=-\frac{\overline{\eta_g}}{d}
\end{equation}
\begin{equation}\label{e4}
\sum_{h \in G} A_g(h-g,k)\overline{A_{g'}(h-g',k)}=\delta_{g,g'}-\frac{\overline{\eta_g}\eta_{g'}}{d} \delta_{k,0}
\end{equation}
\begin{equation}\label{e5}
A_{g+2h}(p,q)=\epsilon_h(g)\epsilon_h(g+p)\epsilon_h(g+q)\epsilon_h(g+p+q)A_g(p,q)
\end{equation}
\begin{equation}
A_g(h,k)=\overline{A_g(k,h)}
\end{equation}
\begin{align} \label{e8}
A_g(h,k) &=A_g(-k,h-k)\eta_g \epsilon_{-k}(g+h)\epsilon_{-k}(g+k)\epsilon_{-k}(g+h+k) \\
&=A_g(k-h,-h) \overline{\eta_g} \epsilon_{-h}(g+h)\epsilon_{-h}(g+k)\epsilon_{-h}(g+h+k) \nonumber
\end{align}
\begin{align}\label{e9}
A_g(h,k)&=A_{g+h}(h,k) \eta_g \eta_{g+k} \overline{\eta_{g+h} \eta_{g+h+k}} \epsilon_h(g)\epsilon_h(g+k) \\
&= A_{g+k}(h,k)  \overline{\eta_g \eta_{g+h}} \eta_{g+k} \eta_{g+h+k}\epsilon_k(g) \epsilon_k(g+h) \nonumber
\end{align}
\begin{align} \label{e10}
\sum_{l \in G} \limits & A_g(x+y,l)A_{g-p+x}(-x,l+p)A_{g-q+x+y}(-y,l+q) \\
 &=A_g(p+x,q+x+y)A_{g-p}(q+y,p+x+y)  \nonumber \\
 & \times \eta_{g} \eta_{g+q+x} \eta_{g+p+q+y}    \overline{ \eta_{g+p} \eta_{g+x+y} \eta_{g+q+x+y} } \nonumber \\
 &\times \epsilon_p(g-p+x)\epsilon_{p+x}(g-p+q+y)
 \epsilon_q(g-q+x+y)\epsilon_{q+y}(g-q+x) \nonumber \\
 & -\displaystyle \frac{\delta_{x,0}\delta_{y,0,}}{d} \eta_g \eta_{g+p} \eta_{g+q} \nonumber
\end{align}

\begin{theorem}\label{izthm} \cite{IzumiNote}
Suppose $A$ and $\epsilon $ satisfy (\ref{e1})-(\ref{e10}). 
Then there is a Type III factor $M$, an outer action $\alpha $ of $ G$ on $M$, an irreducible finite-index endomorphism $\rho $ of $M$, and $n+1$ isometries $S \in (id,\rho^2) $ and $T_g \in (\alpha_g \circ  \rho,\rho^2)$ satisfying the Cuntz algebra relations such that:\\
\begin{equation} \label{e11}
\alpha_g \circ \rho = \rho \circ \alpha_{-g}, \, \forall g \in G
\end{equation}
 \begin{equation}\label{e12}
 \alpha_g (S)=S, \ \ \alpha_g(T_h)=\epsilon_g(h)T_{h+2g}, \, \forall g,h \in G
 \end{equation}
 \begin{equation}\label{e12_2}
 \rho(s)=\frac{1}{d}S+\frac{1}{\sqrt{d}} \sum_{g \in G} \limits T_gT_g
 \end{equation}
 \begin{multline}\label{e13}
\rho(T_g)=\epsilon_g(-g)[\eta_{-g}T_{-g}SS^*+\frac{\overline{\eta_{-g}}}{\sqrt{d}}ST_{-g}^* \\ +\sum_{h,k \in G}A_{
-g}(h,k)T_{h-g}T_{h+k-g}T_{k-g}^*] ,\,  \forall g\in G.
\end{multline}

The sector $[id] \oplus[\alpha_g \rho] $ admits a Q-system iff 
\begin{equation} \label{e14}
A_g(h,0)=\delta_{h,0}-\frac{1}{d-1} 
\, \forall h \in G.
\end{equation}

\end{theorem}

These equations were solved in \cite{IzumiNote} for all groups of order less than or equal to $5$, and solutions have been found for some groups of higher order. For all known solutions for which $[id] \oplus [\rho] $ admits a Q-system,
$[id] \oplus [\alpha_g \rho ] $ admits Q-systems for all $g \in G $. In this case we get some very useful extra relations.

\begin{lemma}\cite{IzumiNote}
If $[id] \oplus [\alpha_g \rho] $ admits a Q-system for each $g \in G $, then $\eta_g$ is identically $1$ and  $\epsilon $ restricted to $G_2 \times  G$ is a bicharacter, where $G_2 $ is the subgroup of order $2$ elements of $G$. 
Moreover, we have
\begin{equation}\label{eq14}
|A_g(h,k) |^2=\delta_{h,0} \delta_{k,0} - \frac{\delta_{h,0}+\delta_{k,0}+\delta_{h,k}}{d-1}+\frac{d}{(d-1)^2} 
\end{equation}
and
\begin{align}\label{eq15}
 & A_g(-l,h)A_g(l,k)\epsilon_l(g-l)\epsilon_l(g+h-l) \nonumber \\
& -A_g(h+l,k)A_g(k-l,h) \epsilon_l(g+k-l) \epsilon_l (g+h+k-l)  \nonumber \\
& = \frac{\delta_{h-k+l,0}}{d-1}  \epsilon_l (g+h) -\frac{\delta_{l,0}}{d-1}  \nonumber \\
\end{align}
for all $g,h,k,l  \in G$. 
\end{lemma}

The structure of a solution to (\ref{e1})-(\ref{e10}) seems to depend in a fundamental way on $\epsilon $, which is determined up to gauge equivalence by its restriction to $G_2 \times G $. From the orbifold construction in the preceding subsection, we can determine that in order for $\mathcal{S}$ to be related to the Asaeda--Haagerup subfactor, the gauge equivalence class of $\epsilon$ 
on $G=\mathbb{Z}/4\mathbb{Z} \times \mathbb{Z}/2\mathbb{Z} $ for the $ \mathbb{Z}/2\mathbb{Z}$-equivariantization of $\mathcal{AH}_4 $ should satisfy 
$$\epsilon_{(0,1)}((a,b))=1,  \  \epsilon_{(2,0)}((a,b))=(-1)^b, \quad \forall (a,b) \in   \mathbb{Z}/4\mathbb{Z} \times \mathbb{Z}/2\mathbb{Z} .$$

We fix $\epsilon$ by setting  $\epsilon_{(1,0)} ((2,1) ) =\epsilon_{(1,0)}((3,1))=-1$ and
$\epsilon_{(1,0)} ((a,b) )=1$ otherwise. Together with $\epsilon_{(0,1)} ((a,b))=1$ and (\ref{e1}) this determines all the $\epsilon$.

Then we can try to solve for $A_g(h,k) $ using $(\ref{e3})-(\ref{e10}) $ and possibly $(\ref{eq14})-(\ref{eq15})$.  We know of no general technique for solving these equations, but nonetheless were able to find a solution in this particular case.
 
\begin{theorem}
There exists a solution to (\ref{e1})-(\ref{e10}) and (\ref{e14}) for the group 
$G=\mathbb{Z}/4\mathbb{Z} \times \mathbb{Z}/2\mathbb{Z} $
with $\epsilon_{(2,0)}((a,b))= (-1)^b$ and $\epsilon_{(0,1)} ((a,b))=1$ for all $(a,b) $.
\end{theorem}
\begin{proof}
We explicitly construct a solution is as follows. We define constants

$$c=\frac{1}{4}(1-d+i\sqrt{10d-2}) ,  \quad f=\sqrt{\frac{1}{2}(d-1-i\sqrt{26d+2})}, $$ $$ g=\frac{1}{2}\sqrt{-3d-1+i\sqrt{50d+6} },   \quad h=\frac{1}{4}(d+3-i(\sqrt{2d-10})).$$

Define $G \times G$ matrices

$$A=\frac{1}{d-1}\left(
\begin{array}{cccccccc}
 d-2 & -1 & -1 & -1 & -1 & -1 & -1 & -1 \\
 -1 & -1 & c & c & -f & f & -g & -g \\
 -1 & \bar{c} & -1 & c & i \sqrt{d} & h & -i \sqrt{d} & \bar{h} \\
 -1 & \bar{c} & \bar{c} & -1 & -\bar{f} & -\bar{g} & \bar{g} & -\bar{f} \\
 -1 & -\bar{f} & -i \sqrt{d} & -f & -1 & -f & i \sqrt{d} & -\bar{f} \\
 -1 & \bar{f} & \bar{h} & -g & -\bar{f} & -1 & g & -\bar{h} \\
 -1 & -\bar{g} & i \sqrt{d} & g & -i \sqrt{d} & \bar{g} & -1 & -g \\
 -1 & -\bar{g} & h & -f & -f & -h & -\bar{g} & -1 \\
\end{array}
\right)$$

%$$
%%\resizebox{\linewidth}{!}{%
%B_{(0,0)}=\left(
%\begin{array}{cccccccc}
% 1 & 1 & 1 & 1 & 1 & 1 & 1 & 1 \\
% 1 & 1 & 1 & 1 & 1 & 1 & 1 & 1 \\
% 1 & 1 & 1 & 1 & 1 & 1 & 1 & 1 \\
% 1 & 1 & 1 & 1 & 1 & 1 & 1 & 1 \\
% 1 & 1 & 1 & 1 & 1 & 1 & 1 & 1 \\
% 1 & 1 & 1 & 1 & 1 & 1 & 1 & 1 \\
% 1 & 1 & 1 & 1 & 1 & 1 & 1 & 1 \\
% 1 & 1 & 1 & 1 & 1 & 1 & 1 & 1 \\
%\end{array}
%\right)
%$$
$$
B_{(1,0)}=\left(
\begin{array}{cccccccc}
 1 & 1 & 1 & 1 & 1 & 1 & 1 & 1 \\
 1 & 1 & 1 & 1 & 1 & 1 & -1 & -1 \\
 1 & 1 & 1 & 1 & 1 & -1 & 1 & -1 \\
 1 & 1 & 1 & 1 & -1 & 1 & 1 & -1 \\
 1 & 1 & 1 & -1 & 1 & 1 & 1 & -1 \\
 1 & 1 & -1 & 1 & 1 & 1 & 1 & -1 \\
 1 & -1 & 1 & 1 & 1 & 1 & 1 & -1 \\
 1 & -1 & -1 & -1 & -1 & -1 & -1 & 1 \\
\end{array}
\right)
%}
$$

%$$B_{(2,0)}=\left(
%\begin{array}{cccccccc}
% 1 & 1 & 1 & 1 & 1 & 1 & 1 & 1 \\
% 1 & 1 & 1 & 1 & 1 & -1 & 1 & -1 \\
% 1 & 1 & 1 & 1 & -1 & -1 & -1 & -1 \\
% 1 & 1 & 1 & 1 & -1 & 1 & -1 & 1 \\
% 1 & 1 & -1 & -1 & 1 & 1 & -1 & -1 \\
% 1 & -1 & -1 & 1 & 1 & 1 & -1 & -1 \\
% 1 & 1 & -1 & -1 & -1 & -1 & 1 & 1 \\
% 1 & -1 & -1 & 1 & -1 & -1 & 1 & 1 \\
%\end{array}
%\right)$$
%
%$$B_{(3,0)}=\left(
%\begin{array}{cccccccc}
% 1 & 1 & 1 & 1 & 1 & 1 & 1 & 1 \\
% 1 & 1 & 1 & 1 & -1 & 1 & 1 & -1 \\
% 1 & 1 & 1 & 1 & -1 & 1 & -1 & 1 \\
% 1 & 1 & 1 & 1 & -1 & -1 & 1 & 1 \\
% 1 & -1 & -1 & -1 & 1 & -1 & -1 & -1 \\
% 1 & 1 & 1 & -1 & -1 & 1 & 1 & 1 \\
% 1 & 1 & -1 & 1 & -1 & 1 & 1 & 1 \\
% 1 & -1 & 1 & 1 & -1 & 1 & 1 & 1 \\
%\end{array}
%\right)$$
%
$$B_{(0,1)}\left(
\begin{array}{cccccccc}
 1 & 1 & 1 & 1 & 1 & 1 & 1 & 1 \\
 1 & 1 & 1 & -1 & 1 & 1 & 1 & -1 \\
 1 & 1 & -1 & -1 & 1 & 1 & -1 & -1 \\
 1 & -1 & -1 & -1 & 1 & -1 & -1 & -1 \\
 1 & 1 & 1 & 1 & 1 & 1 & 1 & 1 \\
 1 & 1 & 1 & -1 & 1 & 1 & 1 & -1 \\
 1 & 1 & -1 & -1 & 1 & 1 & -1 & -1 \\
 1 & -1 & -1 & -1 & 1 & -1 & -1 & -1 \\
\end{array}
\right)$$ $$B_{(1,1)}=\left(
\begin{array}{cccccccc}
 1 & 1 & 1 & 1 & 1 & 1 & 1 & 1 \\
 1 & 1 & -1 & 1 & 1 & 1 & 1 & -1 \\
 1 & -1 & -1 & 1 & 1 & 1 & -1 & -1 \\
 1 & 1 & 1 & -1 & -1 & 1 & 1 & 1 \\
 1 & 1 & 1 & -1 & 1 & 1 & 1 & -1 \\
 1 & 1 & 1 & 1 & 1 & 1 & -1 & -1 \\
 1 & 1 & -1 & 1 & 1 & -1 & -1 & -1 \\
 1 & -1 & -1 & 1 & -1 & -1 & -1 & -1 \\
\end{array}
\right)$$
%
%$$B_{(2,1)}=\left(
%\begin{array}{cccccccc}
% 1 & 1 & 1 & 1 & 1 & 1 & 1 & 1 \\
% 1 & -1 & 1 & 1 & 1 & 1 & 1 & -1 \\
% 1 & 1 & -1 & -1 & -1 & -1 & 1 & 1 \\
% 1 & 1 & -1 & 1 & -1 & 1 & 1 & 1 \\
% 1 & 1 & -1 & -1 & 1 & 1 & -1 & -1 \\
% 1 & 1 & -1 & 1 & 1 & -1 & -1 & -1 \\
% 1 & 1 & 1 & 1 & -1 & -1 & -1 & -1 \\
% 1 & -1 & 1 & 1 & -1 & -1 & -1 & 1 \\
%\end{array}
%\right)$$
%
%$$B_{(3,1)}=\left(
%\begin{array}{cccccccc}
% 1 & 1 & 1 & 1 & 1 & 1 & 1 & 1 \\
% 1 & -1 & -1 & -1 & -1 & -1 & -1 & 1 \\
% 1 & -1 & -1 & 1 & -1 & -1 & 1 & 1 \\
% 1 & -1 & 1 & 1 & -1 & 1 & 1 & 1 \\
% 1 & -1 & -1 & -1 & 1 & -1 & -1 & -1 \\
% 1 & -1 & -1 & 1 & -1 & -1 & -1 & -1 \\
% 1 & -1 & 1 & 1 & -1 & -1 & -1 & 1 \\
% 1 & 1 & 1 & 1 & -1 & -1 & 1 & 1 \\
%\end{array}
%\right)$$

We set $$A_{(0,0)}(h,k)=A(h,k), \quad A_{(1,0)}(h,k)=B_{(1,0)}(h,k) A(h,k), $$ $$\quad A_{(0,1)}(h,k)=B_{(0,1)}(h,k) A(h,k), \quad A_{(1,1)}(h,k)=B_{(1,1)}(h,k) A(h,k) $$
and use (\ref{e5}) to define the remaining $A_g(h,k)$.

We need only check that these numbers satisfy the finitely many polynomial equations (\ref{e1})-(\ref{e10}) and (\ref{e14}).  Although it is possible to check any one of these equations by hand, since there are 40600 total equations (mostly from Equation \ref{e10}) we checked them by computer. 
\end{proof}

\begin{remark}
The computer calculation in the above theorem has been checked twice independently.  Our original program used a mix of C and mathematica and ran in several hours.  A second program was written by Morrison, Penneys, and Peters in mathematica alone and is more easily human-readable but ran in a couple days.  We have included along with the arxiv source to this article a improved version of their code as CheckingSolution.nb which is much faster and runs in under 3 minutes on a  2.4GHz Intel Core i5.  We are grateful to Morrison, Penneys, and Peters for their help improving this code.
\end{remark}

\begin{remark}
The initial step in solving (\ref{e1})-(\ref{e10}) is determining the cocycle $\epsilon $ up to gauge equivalence. In our case we were able to determine what $\epsilon $ should be by assuming that the resulting category had an orbifold quotient which is Morita equivalent to the Asaeda-Haagerup categories. Then (\ref{e5})-(\ref{e9}) can be used to express the $A_g(h,k)$ in terms of a relatively small number of variables, which can be solved using (\ref{e4}) and (\ref{eq15}) for $g=0$. Although we were not assuming that $1+\alpha_g\rho $ admits a Q-system for all $g$, we were assuming that $1+\rho $ admits a $Q$-system, so (\ref{e14}) has to hold for $g=0$ and consequently so does (\ref{eq15}) for $g=0$. If one assumes that (\ref{e14}) holds for all $g$ (which indeed turns out to be the case), then the equations are easier to solve.
\end{remark}

We will call the fusion category corresponding to this solution $\mathcal{AH}^{\mathbb{Z}/2\mathbb{Z}}$ and the subfactor coming from $1+\rho$ we will call $\mathcal{S}$.

\subsection{De-equivariantization and the orbifold construction}
We can now construct $\mathcal{AH}_4 $ by de-equivariantizing $\mathcal{AH}^{\mathbb{Z}/2\mathbb{Z}} $; the de-equivariantization is implemented by another orbifold construction.

Let $G$ be a finite Abelian group and let $ A$ and $\epsilon$ be a solution to (\ref{e1})-(\ref{e10}).  Let $0\neq z \in G_2 $ be such that $\epsilon_g(\cdot) $ is a character and $ \epsilon_{z}(z)=1$.
Let $P=  M \rtimes_{\alpha_z} \mathbb{Z}/2\mathbb{Z}$ be the crossed product, which is the von Neumann algebra generated by $M$ and a unitary $\lambda $ satisfying $\lambda^2=1 $
and $\lambda x \lambda=\alpha_z(x) $ for all $x \in M $. 
%Since $\alpha_z $ is outer, $P$ is a factor.
We extend $\alpha_g $ and $\rho $ to $P$ by setting
   $\tilde{\alpha}_g(\lambda)=\epsilon_z(g) \lambda $ and $\tilde{\rho}(\lambda)=\lambda $.
   Then $\tilde{\alpha}$ is a $G$-action on $P $ and $\tilde{\alpha}_z=\text{Ad } \lambda $.
   
   \begin{theorem}\cite{IzumiNote}
   \begin{enumerate}
   \item $[\tilde{\rho}] $ is irreducible and self-dual. $[id] \oplus [\tilde{\rho}]  $ admits a Q-system
   if $[id] \oplus [\rho] $ does.
   \item $[\tilde{\alpha}_g]=[\tilde{\alpha}_h] $ iff $g-h \in \{ 0,z \} $.
   \item $$ [\tilde{\rho}^2]=[id]\oplus \bigoplus_{\dot{g} \in G / \{0,z\} }2 [\tilde{\alpha}_{\dot{g}} \tilde{\rho}] .$$
   \end{enumerate}
   \end{theorem}

We can apply this construction to $\mathcal{AH}^{\mathbb{Z}/2\mathbb{Z} } $.

\begin{theorem} \label{ahcons}
There is a Type III factor $M$ with an outer action $\alpha$ of $\mathbb{Z}/4\mathbb{Z} $, an irreducible finite-index endomorphism $\rho $ and isometries $S \in (id, \rho^2 )$ and $T_{(a,b)} \in (\alpha_{a} \rho, \rho^2 ) ,\  (a,b)  \in \mathbb{Z}/4\mathbb{Z}  \times \mathbb{Z}/ 2\mathbb{Z}$ 
such that
\begin{enumerate}
\item $\alpha_g  \rho=   \rho \alpha_{-g}  $ and $[\rho]^2=[id]\oplus 2\sum_{g \in  \mathbb{Z}/4\mathbb{Z}  } \limits [\alpha_g \rho] $
\item  (\ref{e12})-(\ref{e13}) hold with $A$ and $\epsilon $ as in the previous subsection for all $\alpha_g, \ g \in \mathbb{Z}/4\mathbb{Z} \subset \mathbb{Z}/4\mathbb{Z} \times \mathbb{Z}/2\mathbb{Z} $.
\end{enumerate}
 \end{theorem}

\begin{proof}
Apply the orbifold construction to $\mathcal{AH}^{\mathbb{Z}/2\mathbb{Z} } $
 with $z=(0,1)$. Note that 
%quotient  $G/\{0,z\} $ can be identified with the subgroup $\mathbb{Z}/4\mathbb{Z} $ and that 
$\alpha_{z} $ acts trivially on the original Cuntz algebra, since $z$ has order $2$ and $\epsilon_z $ is identically $1$. After replacing $T_{(a,1)} $ with $T_{(a,1)} \lambda$ for each $a \in \mathbb{Z}/4\mathbb{Z} $, it is easy to check that (\ref{e12})-(\ref{e13}) still hold.
\end{proof}

We will call the fusion category tensor generated  by $\rho$ in the above theorem $\mathcal{AH}_{4'}$, since we will show below that it is $\mathcal{AH}_4$.

\begin{remark}
By looking at the details of the proof of Theorem \ref{izthm}, one can give a more direct construction of $\mathcal{AH}_{4'}$.  The way that Theorem \ref{izthm} is proven is by using the structure constants $ A$, $ \epsilon$, and $\eta $ to define an endomorphism  $\rho $ and a $G$-action $\alpha $ of a Cuntz algebra and then completing this Cuntz algebra to a von Neumann algebra with respect to an appropriate state. Then all of the properties in Theorem \ref{izthm} are satisfied except possibly for the action of $G$ being outer.  One then needs to make an extra modification in order to enforce outerness.  In our case, this extra modification is exactly cancelled out by the de-equivariantization.  Since $\epsilon_z $ is trivial (which means that $\alpha_{(0,1)} $ acts trivially on the Cuntz algebra) the category that one gets on the von Neumann algebra completion of the Cuntz algebra without any modification is already $\mathcal{AH}_{4'} $. Thus it is not strictly necessary to go through the orbifold construction here.

On the other hand, without using the fact that $\mathcal{AH}_4 $ would have to come from a de-equivariantization of a generalized Haagerup category for $ \mathbb{Z}/4\mathbb{Z} \times \mathbb{Z}/2\mathbb{Z}$, we would have no way of knowing that $\mathcal{AH}_4$ comes from a solution to (\ref{e1})-(\ref{e10}).  Thus we found the two-step construction more natural.
\end{remark}

\subsection{Reconstruction of the Asaeda-Haagerup subfactor}
To show that the fusion category $\mathcal{AH}_{4'}$ constructed in the previous subsection satisfies the defining characterization of $\mathcal{AH}_4 $ (and therefore show that $\mathcal{AH}_{4-6}$ exist), it remains to show that  $\mathcal{AH}_{4'} $ is Morita equivalent to  $\mathcal{AH}_{1-3}$. We will show this by deducing that the dual category of $\mathcal{AH} _{4'}$ with respect to a Q-system for $[id] \oplus [\rho] $ is $\mathcal{AH}_2$. In the process, we will arrive at a new construction of the Asaeda-Haagerup subfactor. 

Let $M$ be a Type III factor with endomorphisms $\rho $ and $\alpha_g, \ g \in \mathbb{Z}/4\mathbb{Z} $ realizing $\mathcal{AH}_{4'}$, as in Theorem \ref{ahcons}.
Then since the sector $[id] \oplus [\rho] $ admits a Q-system, there is a Type III factor $N$ and a finite-index homomorphism $\iota: N \rightarrow M $
with a dual homomorphism  $\bar{\iota}: M \rightarrow N $ such that $[\iota \bar{\iota}]=[id]\oplus [\rho] $. The fusion ring of $\mathcal{AH}_{4'}$ is $AH_4 $. 

\begin{lemma} \label{nti}
 The fusion category $\mathcal{AH}_{2'}$ of endomorphisms of $N $ tensor generated by $\bar{\iota} \iota $ has a non-trivial invertible object $\alpha$ satisfying $\alpha_2 \iota = \iota \alpha$.
\end{lemma}
\begin{proof}
 It suffices to show that $\alpha_2 \iota $ determines an equivalent $Q$-system to that of $\iota $, since the only way two different sectors can determine equivalent $Q$-systems is if they are in the same orbit under the action of the invertible objects of the dual category \cite[\S3.3]{GSbp}. Since conjugation by $\alpha_2 $ fixes $\rho $ and 
 $\alpha_2 $ fixes $S$ and $T_0$, this is indeed the case.
\end{proof}

\begin{lemma} \label{2graphs}
The principal graph pair of $ \iota$ is either 
$$\begin{tikzpicture}
\draw [fill] (0,0) circle [radius=.05];
\node [below] at (0,0) {$1$};
\draw [fill] (1,0) circle [radius=.05];
\node [below] at (1,0) {$\alpha$};
\draw [fill] (2,0) circle [radius=.05];
\node [below] at (2,0) {$\pi_1$};
\draw [fill] (3,0) circle [radius=.05];
\node [below] at (3,0) {$\pi_2$};
\draw [fill] (4,0) circle [radius=.05];
\node [below] at (4,0) {$\eta$};
\draw [fill] (5,0) circle [radius=.05];
\node [below] at (5,0) {$\sigma_1$};
\draw [fill] (6,0) circle [radius=.05];
\node [below] at (6,0) {$\sigma_2$};
\draw [fill] (7,0) circle [radius=.05];
\node [below] at (7,0) {$\sigma_3$};
\draw [fill] (8,0) circle [radius=.05];
\node [below] at (8,0) {$\sigma_4$};

\draw [fill] (1.5,2) circle [radius=.05];
\node [left] at (1.5,2) {$\iota$};
\draw [fill] (2.5,2) circle [radius=.05];
\node [left] at (2.5,2) {$\alpha_1 \iota$};
\draw [fill] (3.5,2) circle [radius=.05];
\node [left] at (3.5,2) {$\alpha_3 \iota$};
\draw [fill] (4.5,2) circle [radius=.05];
\node [left] at (4.5,2) {$\alpha_2 \iota$};
\draw [fill] (5.5,2) circle [radius=.05];
\node [left] at (5.5,2) {$\kappa$};
\draw [fill] (6.5,2) circle [radius=.05];
\node [left] at (6.5,2) {$ \alpha_1 \kappa$};

\draw [fill] (.5,4) circle [radius=.05];
\node [above] at (.5,4) {$1$};
\draw [fill] (1.5,4) circle [radius=.05];
\node [above] at (1.5,4) {$\alpha_1$};
\draw [fill] (2.5,4) circle [radius=.05];
\node [above] at (2.5,4) {$\alpha_3$};
\draw [fill] (3.5,4) circle [radius=.05];
\node [above] at (3.5,4) {$\alpha_2$};
\draw [fill] (4.5,4) circle [radius=.05];
\node [above] at (4.5,4) {$\rho$};
\draw [fill] (5.5,4) circle [radius=.05];
\node [above] at (5.5,4) {$\alpha_1 \rho$};
\draw [fill] (6.5,4) circle [radius=.05];
\node [above] at (6.5,4) {$\alpha_3 \rho$};
\draw [fill] (7.5,4) circle [radius=.05];
\node [above] at (7.5,4) {$\alpha_2 \rho$};

\path (0,0) edge (1.5,2);
\path (2,0) edge (1.5,2);
\path (2,0) edge (5.5,2);
\path (2,0) edge (6.5,2);
\path (3,0) edge (5.5,2);
\path (4,0) edge (5.5,2);
\path (5,0) edge (5.5,2);
\path (6,0) edge (5.5,2);
\path (3,0) edge (6.5,2);
\path (4,0) edge (6.5,2);
\path (7,0) edge (6.5,2);
\path (8,0) edge (6.5,2);
\path (4,0) edge (2.5,2);
\path (4,0) edge (3.5,2);
\path (1,0) edge (4.5,2);
\path (3,0) edge (4.5,2);

\path (1.5,2) edge (.5,4);
\path (2.5,2) edge (1.5,4);
\path (3.5,2) edge (2.5,4);
\path (4.5,2) edge (3.5,4);
\path (1.5,2) edge (4.5,4);
\path (2.5,2) edge (5.5,4);
\path (3.5,2) edge (6.5,4);
\path (4.5,2) edge (7.5,4);
\path (5.5,2) edge (4.5,4);
\path (5.5,2) edge (5.5,4);
\path (5.5,2) edge (6.5,4);
\path (5.5,2) edge (7.5,4);
\path (6.5,2) edge (4.5,4);
\path (6.5,2) edge (5.5,4);
\path (6.5,2) edge (6.5,4);
\path (6.5,2) edge (7.5,4);

\end{tikzpicture} $$ or $$ \begin{tikzpicture}
\draw [fill] (0,0) circle [radius=.05];
\node [below] at (0,0) {$1$};
\draw [fill] (1,0) circle [radius=.05];
\node [below] at (1,0) {$\alpha$};
\draw [fill] (2,0) circle [radius=.05];
\node [below] at (2,0) {$\pi_1$};
\draw [fill] (3,0) circle [radius=.05];
\node [below] at (3,0) {$\pi_2$};
\draw [fill] (4,0) circle [radius=.05];
\node [below] at (4,0) {$\eta$};
\draw [fill] (6.5,0) circle [radius=.05];
\node [below] at (6.5,0) {$\sigma$};

\draw [fill] (1.5,2) circle [radius=.05];
\node [left] at (1.5,2) {$\iota$};
\draw [fill] (2.5,2) circle [radius=.05];
\node [left] at (2.5,2) {$\alpha_1 \iota$};
\draw [fill] (3.5,2) circle [radius=.05];
\node [left] at (3.5,2) {$\alpha_3 \iota$};
\draw [fill] (4.5,2) circle [radius=.05];
\node [left] at (4.5,2) {$\alpha_2 \iota$};
\draw [fill] (5.5,2) circle [radius=.05];
\node [left] at (5.5,2) {$\kappa$};
\draw [fill] (6.5,2) circle [radius=.05];
\node [left] at (6.5,2) {$ \alpha_1 \kappa$};

\draw [fill] (.5,4) circle [radius=.05];
\node [above] at (.5,4) {$1$};
\draw [fill] (1.5,4) circle [radius=.05];
\node [above] at (1.5,4) {$\alpha_1$};
\draw [fill] (2.5,4) circle [radius=.05];
\node [above] at (2.5,4) {$\alpha_3$};
\draw [fill] (3.5,4) circle [radius=.05];
\node [above] at (3.5,4) {$\alpha_2$};
\draw [fill] (4.5,4) circle [radius=.05];
\node [above] at (4.5,4) {$\rho$};
\draw [fill] (5.5,4) circle [radius=.05];
\node [above] at (5.5,4) {$\alpha_1 \rho$};
\draw [fill] (6.5,4) circle [radius=.05];
\node [above] at (6.5,4) {$\alpha_3 \rho$};
\draw [fill] (7.5,4) circle [radius=.05];
\node [above] at (7.5,4) {$\alpha_2 \rho$};

\path (0,0) edge (1.5,2);
\path (2,0) edge (1.5,2);
\path (2,0) edge (5.5,2);
\path (2,0) edge (6.5,2);
\path (3,0) edge (5.5,2);
\path (4,0) edge (5.5,2);
\path (6.5,0) edge (5.5,2);
%\path (6,0) edge (5.5,2);
\path (3,0) edge (6.5,2);
\path (4,0) edge (6.5,2);
\path (6.5,0) edge (6.5,2);
%\path (8,0) edge (6.5,2);
\path (4,0) edge (2.5,2);
\path (4,0) edge (3.5,2);
\path (1,0) edge (4.5,2);
\path (3,0) edge (4.5,2);

\path (1.5,2) edge (.5,4);
\path (2.5,2) edge (1.5,4);
\path (3.5,2) edge (2.5,4);
\path (4.5,2) edge (3.5,4);
\path (1.5,2) edge (4.5,4);
\path (2.5,2) edge (5.5,4);
\path (3.5,2) edge (6.5,4);
\path (4.5,2) edge (7.5,4);
\path (5.5,2) edge (4.5,4);
\path (5.5,2) edge (5.5,4);
\path (5.5,2) edge (6.5,4);
\path (5.5,2) edge (7.5,4);
\path (6.5,2) edge (4.5,4);
\path (6.5,2) edge (5.5,4);
\path (6.5,2) edge (6.5,4);
\path (6.5,2) edge (7.5,4);
\end{tikzpicture}.$$

\end{lemma}

\begin{proof}
The principal graph (which is the upper half of these graphs) is easy to deduce. We discuss the calculation of the dual graphs. Let $\alpha $ be the nontrivial automorphism determined by Lemma \ref{nti}.  Hence, $\bar{\iota} \alpha_2 \iota=\bar{\iota}\iota \alpha$. For any $g,h \in \mathbb{Z}/4\mathbb{Z}  $
we have $$(\bar{\iota}\alpha_g \iota,\bar{\iota} \alpha_h \iota )=(\alpha_g \iota \bar{\iota},\iota \bar{\iota} \alpha_h)=(\alpha_g+\alpha_g \rho,\alpha_h + \rho \alpha_h)=\delta_{g,h}+\delta_{g,-h}.$$
This means that $\iota $ and $\alpha_2 \iota $ are each connected to two vertices with no overlap and $\alpha_1 \iota $ and $\alpha_3 \iota $ are each connected to the same single vertex.  It is clear that $1$ is connected to $\iota$ and $\alpha$ is connected to $\alpha_2 \iota$.  We label the other vertex touching $\iota$ by $\pi_1$ and the other vertex touching $\alpha_2 \iota$ by $\pi_2$ (in the end the dual will be $\mathcal{AH}_2$, so we have named our variables appropriately).  Note that, by the defining property of $\alpha$, we have that $\pi_1 \alpha = \pi_2$.

We also compute $$(\iota \pi_1,\kappa)=(\iota(1+\pi_1),\kappa)=(\iota\bar{\iota} \iota,\kappa)=1, $$ so $\pi_1 $ is connected to $\kappa $ by a single edge. In a similar way, we find that $\pi_1$, $\pi_2$, and $ \eta$ are each connected to each of $ \kappa $ and $\alpha_1 \kappa $ by a single edge. 

Next we note that we must have $ (\rho\kappa ,\kappa   )+(\rho\kappa,\alpha_1 \kappa )=7$, since $(\rho \kappa,\alpha_g \iota)=1 $ for each $g$ and $d(\rho)=4d(\iota) +7d(\kappa)$.
Then $$(\bar{\iota}\kappa,\bar{\iota} \kappa )+(\bar{\iota}\kappa,\bar{\iota} \alpha_1 \kappa )=(\iota\bar{\iota},\kappa \bar{\kappa} )+(\iota \bar{\iota}, \alpha_1 \kappa \bar{\kappa} )=(1+\rho,\kappa \bar{\kappa})+(1+\rho,\alpha_1 \kappa \bar{\kappa})=8.$$
Therefore the total number of paths from $\kappa $ to itself in the dual graph plus the total number of paths from $\kappa $ to $\alpha_1 \kappa $ must equal $8$, and the only two possibilities are pictured. 

\end{proof}

We refer to these two cases as the ``four-sigma case'' and ``one-sigma case'', respectively.  We would like to rule out the one sigma case.

%We would like to rule out the second possibility in Lemma \ref{2graphs}. From the graph, we can deduce the fusion rules for the dual category corresponding to this possibility as follows. All the objects are self-dual so the fusion ring is commutative. Counting paths in the usual way, we have $$\pi^2=(\alpha \pi)^2=1+2\alpha \pi+2\pi+2\eta+2\sigma .$$ By Frobenius reciprocity, we now know how many copies of $\pi $ and $\alpha \pi $ any given product contains, in addition to whether the product contains $1$ or $\alpha $. Then we can figure out how many copies of $\eta $ and $\sigma $ the product contains by using $d(\pi)=d $, $d(\eta) =d+1$, and $d(\sigma)=d-1 $. We will denote this fusion ring by $R$.

\begin{remark}
It is not difficult to deduce the fusion rules for $\mathcal{AH}_{2'}$ in the one-sigma case, getting a fusion ring $R$.  There does exist a consistent $AH_4-R$ fusion bimodule giving Lemma \ref{2graphs}. Therefore it is impossible to rule out this possibility from the fusion rules for this bimodule and the two fusion rings alone.  As usual, we exploit the additional combinatorial structure coming from looking at multiple subfactors simultaneously.  In particular, we look at the compatibility between this bimodule and the bimodule coming from the algebra $1+\alpha_1 \rho$.
\end{remark}

\begin{remark}
If one only knows the fusion rules for $\mathcal{AH}_{4'}$ and the fusion module corresponding to $\iota$ but not Lemma \ref{nti}, then there are two additional possibilities for the dual graph beyond the above two.
\end{remark}

%
%
%Let  $\kappa$ be one of the depth three objects in the $\mathcal{AH}_{4'}$-module category $\mathcal{K} $ generated by $\iota $. The fusion graph of $\iota $ determines the $AH_4$-fusion module of $\mathcal{K}$ up to two possibilities,
%corresponding to whether $(\rho \kappa, \kappa )=3 $ or $(\rho \kappa, \kappa)=4 $. In either case,
%$(\rho \kappa, \kappa)+(\alpha_1 \rho \kappa, \kappa )=7 $.
%
%The dual graph can then be easily computed up to two possibilities using Lemma \ref{duallem}  and the order $2$ symmetry provided by Lemma \ref{nti}. One possibility is the graph pictured in 
%Lemma \ref{graphpair}, and the other possibility has the vertices $p,q,r,s$ replaced by single vertex connected to the vertices $I,J$ by a single edge each. The Frobenius-Perron dimension vector 
%of the even vertices in the dual graph is $\{1,1,\frac{d-1}{2},\frac{d-1}{2},\frac{d-1}{2},\frac{d-1}{2},
%d,d,d+1 \} $ for the first possibility and $\{1,1,d-1,d,d,d+1 \} $ for the second. In either case we can compute the dual fusion ring from the graph. The first possibility gives the Asaeda-Haagerup fusion ring $AH_2$; we would like to rule out the second possibility. 

%There are three possibilities for the fusion ring of $\mathcal{AH}' $, but only two of them
%have non-trivial invertible basis elements. 

Since $[id] \oplus [\alpha_1 \rho]  $ has a Q-system, there is another module over $\mathcal{AH}_{4'}$ with a simple object $\lambda$ such that $\lambda \bar{\lambda} \cong id\oplus \alpha_1 \rho$.  In other words, there is another Type III factor $N_1$ and a finite-index homomorphism $\lambda: N \rightarrow M $ with dual homomorphism  $\bar{\lambda}: M \rightarrow N $ such that $\lambda \bar{\lambda} \cong id\oplus \alpha_1 \rho$.

\begin{lemma} \label{prev}
\begin{enumerate}
\item
The $d+1$-dimensional homomorphisms $ \bar{\iota}\alpha_1 \iota \in \text{End}_0(N) $ and $\bar{\lambda} \iota \in \text{Hom}_0(N,N_1) $ are irreducible. 
\item The $4d$-dimensional homomorphism $ \bar{\iota} \kappa \in \text{End}_0(N)$ and $ \bar{\lambda} \kappa \in \text{Hom}_0(N,N_1)$ satisfy either:
(a) $( \bar{\iota} \kappa, \bar{\iota} \kappa)=5$ and $( \bar{\lambda} \kappa, \bar{\lambda} \kappa)=4$; or
(b) $( \bar{\iota} \kappa, \bar{\iota} \kappa)=4$ and $( \bar{\lambda} \kappa, \bar{\lambda} \kappa)=5$.
\end{enumerate}
\end{lemma}
\begin{proof}
By Frobenius reciprocity, $$( \bar{\iota} \alpha_1 \iota, \bar{\iota}\alpha_1 \iota)=
(\alpha_1 \iota \bar{\iota},\iota \bar{\iota} \alpha_1)
=(\alpha_1+\alpha_1\rho,\alpha_1+ \rho \alpha_1)=1$$ and $$(\bar{\lambda} \iota,\bar{\lambda} \iota)=
(\iota \bar{\iota}, \lambda \bar{\lambda})=(1+\rho, 1+ \alpha_1 \rho) =1.$$
 Similarly, $$( \bar{\iota} \kappa, \bar{\iota} \kappa)=( \iota \bar{\iota}, \kappa\bar{\kappa})=1+(\rho \kappa, \kappa)$$ and $$( \bar{\lambda} \kappa, \bar{\lambda} \kappa)=
(\lambda \bar{\lambda}, \kappa \bar{\kappa})=1+ (\alpha_1 \rho \kappa, \kappa).$$ 
\end{proof}

\begin{lemma}\label{killposs}
The one-sigma case of Lemma \ref{2graphs} does not occur.
\end{lemma}
\begin{proof}
Since the objects in $\mathcal{AH}_{2'}$ would have dimensions $(1,1,d-1,d,d,d+1 )$, there is no way to combine them to get an object with dimension $4d$ and with $5$-dimensional endomorphism space.  Hence, by Lemma \ref{prev}, we must have that $\bar{\iota} \kappa$ satisfies $(\bar{\iota} \kappa, \bar{\iota} \kappa) = 4$ and thus $(\bar{\lambda} \kappa, \bar{\lambda} \kappa) = 5$.  We will examine the bimodule category corresponding to $\bar{\lambda} \iota$ more closely.

First note that $d(\bar{\lambda} \iota) = 5+\sqrt{17}$ which lies in the field $\mathbb{Q}(\sqrt{17})$.  Since eigentheory works over any field, it follows that all the other odd vertices must also have dimension lying in $\mathbb{Q}(\sqrt{17})$.  We now look more closely at $\bar{\lambda} \kappa$.  Frobenius reciprocity shows that $(\bar{\lambda} \kappa, \bar{\iota} \kappa) = 1$ and $(\bar{\lambda} \kappa, \bar{\iota} \alpha \kappa) = 1$, thus $\bar{\lambda} \kappa$ breaks up as  $[\bar{\iota} \kappa]$ plus $[\bar{\iota} \alpha \kappa]$ plus three other simple objects.  These three objects have total dimension $(3+\sqrt{17})$ so at least one of them has dimension below $\frac{2}{3}(3+\sqrt{17})$.  Thus, there's a nontrivial division algebra in $\mathcal{AH}_{2'}$ which has dimension smaller than $\frac{4}{9}(3+\sqrt{17})^2$ and whose dimension is a square in $\mathbb{Q}(\sqrt{17})$.  In any division algebra the trivial appears with multiplicity one and invertible objects appear with multiplicity at most one.  Thus, from the dimensions of the simple objects in $\mathcal{AH}_{2'}$ the list of possible dimensions of algebra objects in $\mathcal{AH}_{2'}$ below $\frac{4}{9}(3+\sqrt{17})^2$ is:
$$1, 2, d, d+1, d+2, d+3, 2d-1, 2d,2d+1,2d+2,2d+3,2d+4.$$
Clearly $2$ is not a square in $\mathbb{Q}(\sqrt{17})$, and none of the other numbers are squares either since their norms are the nonsquare numbers $-1$, $8$, $19$, $32$, $-19$, $-4$, $13$, $32$, $53$,  and $76$, respectively.  Thus, the one-sigma case cannot occur.
%Suppose the second possibility occurs.The dimension vector of $R$ is $(1,1,d-1,d,d,d+1 ) $, so $R$ does not have five distinct  basis elements whose dimensions sum to $4d$, nor two distinct basis elements such that twice one plus the other has dimension $4d$. Therefore, by Lemma \ref{prev},  we have $( \bar{\lambda} \kappa, \bar{\lambda} \kappa)=5$.  Consider the fusion module over $R$ corresponding to the module category over $\mathcal{AH}_{2'} $ generated by $\bar{\lambda} \iota $.  This fusion module has a basis element of dimension $d+1$, and either five distinct basis elements whose dimensions sum to $4d$, or two distinct basis elements such that twice one plus the other has dimension $4d$.  Using a computer, as in \cite{GSbp}, we find all the fusion modules over $R$, of which there are $209$. However,  besides for the trivial module, the only fusion modules with basis elements of dimension $d+1$ have dimension vector $(d-1,d-1,d+1,d+1) $. 
\end{proof}

Now that we know that we're in the four-sigma case, we want to work out the fusion rules for the $\sigma_i$ to show that $\mathcal{AH}_{2'}$ must have a $Q$-system giving the Asaeda--Haagerup subfactor.

\begin{lemma}
We have $\alpha \pi_1 = \pi_1 \alpha = \pi_2$.
\end{lemma}
\begin{proof}
Counting paths from $\iota$ to itself through the upper and lower graphs, we see that $\pi_1$ and $\pi_2$ are both selfdual.  Since $\alpha$ is also self-dual, this means that $\alpha$ commutes with $\pi_1$ and $\pi_2$.  Recall that $1+\pi_1 = \bar{\iota} \iota$, so $\alpha + \pi_1 \alpha = \bar{\iota} \iota \alpha = \bar{\iota} \alpha_2 \iota = \alpha + \pi_2$.  Thus $\pi_1 \alpha = \pi_2$.
\end{proof}

In keeping with the suggestive notation from $AH_2$, we will let $\pi = \pi_2$ and now denote $\pi_1$ by $\alpha \pi$.

\begin{lemma}
The category $ \mathcal{AH}_{2'}$ has a self-dual simple object $\sigma$ with dimension $d(\sigma)=\displaystyle \frac{d-1}{2} $ such that $\sigma^2=1+\sigma+ \pi $.
\end{lemma}
\begin{proof}
Each of the four $\sigma_i $ has dimension $d(\sigma) =\displaystyle \frac{d-1}{2} = \frac{3+\sqrt{17}}{2}$.
Note that of $\sigma_1 $ and $\sigma_2 $, one object must be self-dual while the other one is not, in order for $\kappa $ to admit the same number of length two paths through the upper graph and through the lower graph. Similarly one of $\sigma_3 $ and $\sigma_4$ is self-dual and the other is not.  Without loss of generality,  suppose $\sigma_1 $ and $ \sigma_3$ are the self-dual objects while $\sigma_2$ and $\sigma_4$ are dual to each other.  We would like to understand how tensoring with $\alpha$ on the left or right acts on the $\sigma_i$.  We have
$$\alpha \bar{\iota}\kappa=\overline{\iota \alpha }\kappa=\overline{\alpha_2\iota}\kappa=\bar{\iota}\alpha_2 \kappa=\bar{\iota} \kappa ,$$
so tensoring with  $\alpha$ on the left fixes $\sigma_1+\sigma_2 $, and similarly fixes $\sigma_3+\sigma_4$.

%The dimension of $\sigma_1^2 $ is $1+d(\sigma)+d$, so $ \sigma_1^2$ must contain $1$ and at least one of the $\sigma_i$, as well as another (not-necessarily-simple) object of dimension $d$.

From the graph we have $$\alpha \pi+\pi+\eta+\sigma_1+\sigma_2= \bar{\iota}\iota  \sigma_1=(1+\alpha\pi)\sigma_1=\sigma_1+\alpha\pi \sigma_1 ,$$
and hence $\alpha \pi \sigma_1=\alpha \pi+\pi+\eta+\sigma_2$.
By Frobenius reciprocity, it follows that $(\pi \sigma_1, \alpha \sigma_2 )=1$ and $(\pi, \alpha \sigma_2 \sigma_1) = 1$.   

We claim that $\alpha \sigma_2=\sigma_1 $. Suppose to the contrary that $\alpha \sigma_2=\sigma_2$. Then $\sigma_2 \sigma_1$ contains $\pi $. Thus the remaining part of $\pi$ has dimension $d(\sigma)^2-d(\pi) = \frac{5+\sqrt{17}}{2}$, which means it must be the sum of an invertible object and one of the $\sigma_i$.  But this is impossible because $\sigma_2 \sigma_1$ does not contain $1$ since they are not dual to each other, and does not contain $\alpha$ because $\alpha \sigma_2 = \sigma_2$ is not dual to $\sigma_1$.  Therefore $\alpha \sigma_2=\sigma_1 $, and similarly $\alpha \sigma_3=\sigma_4 $.   

We set $\sigma=\sigma_1 $, so that $\alpha \sigma=\sigma_2 $, $\sigma \alpha=\overline{\sigma_2}=\sigma_4 $, and $\alpha \sigma  \alpha =\alpha \sigma_4 = \sigma_3 $.  Now return to $\sigma^2$, we now know that $1=(\pi \sigma_1, \alpha \sigma_2) = (\pi \sigma, \sigma)$ so we must have that $\pi$ occurs in $\sigma^2$.  Thus, by dimension considerations, we must have that $\sigma^2 = 1+\sigma_i+\pi$ and since $\sigma$ is selfdual while $\alpha \sigma$ and $\sigma \alpha$ are not, we must have $\sigma^2 = 1+\sigma+\pi$ or $\sigma^2 = 1+\alpha\sigma\alpha+\pi$.  So our final task is the rule out the latter possibility.

Suppose that $x=\alpha \sigma \alpha $. Then from the relation
$$\sigma^2\sigma=\sigma \sigma^2 $$
we get $$\sigma \pi+ \sigma \alpha \sigma \alpha=\pi \sigma+\alpha \sigma \alpha \sigma .$$
Since $ \pi \sigma= \pi+\alpha\pi+\eta+\sigma$ and $\pi $ and  $\sigma$ are both self-dual, we have $\sigma\pi= \overline{\sigma \pi} = \bar{\pi} \bar{\sigma} = \pi \sigma $.  Therefore,  $\sigma \alpha \sigma \alpha=  \alpha \sigma \alpha \sigma$.  So consider $\sigma \alpha \sigma$, since its dimension is half odd (that is it is not in $\mathbb{Z}[\sqrt{17}]$), it must have an odd number of $\sigma_i$ in it.  But conjugating by $\alpha$ acts faithfully on the $\sigma_i$, so $\sigma \alpha \sigma \alpha$ and $\alpha \sigma \alpha \sigma$ cannot have the same $\sigma_i$ summands.  This is a contradiction, so we must have $x = \sigma$ and $\sigma^2=1+\pi+\sigma $.
\end{proof}

%\begin{corollary}
% $\mathcal{AH}_{2'} $ has the same fusion ring as $\mathcal{AH}_2$.
%\end{corollary}
%\begin{proof}
%The only other possibility is a ring $R$ with dimension vector $\{1,1,d-1,d,d,d+1 \} $. $R$
%does not have five distinct  basis elements whose dimensions sum to $4d$, nor two distinct basis elements such that twice one plus the other has dimension $4d$. Therefore, by the previous lemma, if $R$ were the fusion ring of $\mathcal{AH}' $,  possibility (a) in Lemma \ref{prev} would have to be satisfied.
%But then $R$ would have to admit a fusion module corresponding to the module category generated by $\bar{\eta} \iota $  which has a basis element of dimension
%$d+1$, and either five distinct basis elements whose dimensions sum to $4d$, or two distinct basis elements such that twice one plus the other has dimension $4d$. But $R$ does not admit any such fusion module. 
%\end{proof}

This gives an alternative proof for the existence of the Asaeda-Haagerup subfactor \cite{MR1686551} by applying a short skein theoretic argument from \cite{GSbp}.

\begin{theorem}
The object $1 \oplus \sigma$ in $\mathcal{AH}_{2'}$ has an algebra structure giving the Asaeda--Haagerup subfactor.
\end{theorem}

\begin{proof}
 Since the fusion ring contains a self-dual simple object $\sigma $ with $$d(\sigma)=\frac{d-1}{2}=\frac{3+\sqrt{17}}{2} $$ satisfying $\sigma^2=1+\sigma+\pi $ with $\pi $
 simple, $1+\sigma $ admits a Q-system by \cite[\S3.2]{GSbp}. It is easy to see that the only possible graph pair for such a Q-system is given by the Asaeda-Haagerup principal graphs.
\end{proof}

\begin{corollary}
 $\mathcal{AH}_{2'} $ is equivalent to $\mathcal{AH}_2 $.
 \end{corollary}

\begin{proof}
 This follows from the uniqueness of the connection on the Asaeda-Haagerup principal graph pair.
\end{proof}
\begin{theorem}
The fusion categories $\mathcal{AH}_{4-6} $ exist. The Morita equivalence class contains exactly six fusion categories up to equivalence.
\end{theorem}
\begin{proof}
Since $\mathcal{AH}_{4'} $ is Morita equivalent to $\mathcal{AH}_2 $ via a fusion category with  fusion module is $L_{AH_2}^4$, it satisfies the characterizing property of $\mathcal{AH}_4$.  Thus $\mathcal{AH}_4$ exists, and so by Theorem \ref{unq} each of $\mathcal{AH}_{4-6}$ exist and are unique.  By Theorem \ref{no7}, there are no fusion categories realizing case (7), so by Theorem \ref{cases} the only fusion categories in the Asaeda--Haagerup Morita equivalence class are $\mathcal{AH}_i$ for $1 \leq i \leq 6$.
\end{proof}

\section{Applications} \label{sec:apps}

In this section we outline several applications of our main theorems.  These applications will appear in subsequent papers.

\subsection{Drinfeld center}
The Drinfeld center, or quantum double, $Z(\mathcal{C}) $ of a spherical fusion category $\mathcal{C} $ is the braided fusion category whose objects are are pairs of an object in $X \in \mathcal{C}$ endowed with a natural half-braiding $\eta_Y: X \otimes Y \rightarrow Y \otimes X$.  In the subfactor setting, the Drinfeld center can be realized as one of the even parts of the asymptotic inclusion or the Longo-Rehren construction.  The category $Z(\mathcal{C}) $ is a modular tensor category, in the sense that the $S$-matrix given by the (normalized) traces of the Hopf links coming from the braiding is invertible. The $S$-matrix, along with the diagonal $T$-matrix which gives the twist of each object in the modular tensor category, is called the modular data. 

Since the discovery of the ``exotic'' Haagerup and Asaeda-Haagerup in the 1990's, there has been significant interest in computing their quantum doubles. Izumi developed a general method to describe the quantum double of a fusion category from Cuntz algebra models in terms of tube algebras \cite{MR1782145}; he used this to describe the quantum doubles of the Haagerup subfactor and its generalizations \cite{MR1832764}. The modular data of the Haagerup subfactor was simplified and generalized by Evans and Gannon \cite{MR2837122}.  By contrast, for the Asaeda--Haagerup subfactor, the quantum double has remained elusive.

Theorem \ref{ahcons} gives a Cuntz algebra model for the fusion category $\mathcal{AH}_4 $, and we can apply the method of \cite{MR1782145} to describe its tube algebra.  
In particular, we can compute the simple objects, find their dimensions, and compute the modular data of the Asaeda-Haagerup subfactor. This calculation is complicated, so we will defer the details to a subsequent paper and provide a summary of the results here.

The rank of the quantum double of the Asaeda--Haagerup subfactor is 22.  In addition to the trivial object, there are six objects of dimension $1+4(4+\sqrt{17})$, eight objects of dimension $8(4+\sqrt{17})$,  one of dimension $1+8(4+\sqrt{17})$, and six of dimension $2+8(4+\sqrt{17})$.  Following Evans--Gannon, as in the Haagerup case one can choose a nice ordering of the objects so that the modular data breaks up cleanly into blocks.  In this ordering, the first $14$ entries of the diagonal of the $T$-matrix are given by the vector $(1,1,1,-1,1,-1,i,-i,1,1,1,1,1,1) $; the final eight entries are given by $e^{\frac{6l^2 \pi i}{17}}, \ 1 \leq l \leq 8 $.  Similarly, the $S$-matrix is given by the following blocks.  There is a $14 \times 14$ block given by
	%\renewcommand{\arraystretch}{1.5} 
	%\resizebox{\linewidth}{!}{%
	$$\frac{1}{8} \left(
	\begin{array}{cccccccccccccc}
	  \frac{8}{\Lambda}&  \frac{8d^2}{\Lambda} & 2 &
	   2 & 2 & 2 & 2 & 2 & 1 & 1 & 1 & 1 & 1 & 1 \\
	 \frac{8d^2}{\Lambda}&  \frac{8}{\Lambda}& 2 &
	   2 & 2 & 2 & 2 & 2 & 1 & 1 & 1 & 1 & 1 & 1 \\
	 2 & 2 & 4 & -4 & 0 & 0 & 0 & 0 & 2 & 2 & -2 & -2 & -2 & -2 \\
	 2 & 2 & -4 & 4 & 0 & 0 & 0 & 0 & 2 & 2 & -2 & -2 & -2 & -2 \\
	 2 & 2 & 0 & 0 & 4 & -4 & 0 & 0 & -2 & -2 & 2 & 2 & -2 & -2 \\
	 2 & 2 & 0 & 0 & -4 & 4 & 0 & 0 & -2 & -2 & 2 & 2 & -2 & -2 \\
	 2 & 2 & 0 & 0 & 0 & 0 & -4 & 4 & -2 & -2 & -2 & -2 & 2 & 2 \\
	 2 & 2 & 0 & 0 & 0 & 0 & 4 & -4 & -2 & -2 & -2 & -2 & 2 & 2 \\
	 1 & 1 & 2 & 2 & -2 & -2 & -2 & -2 & 5 & -3 & 1 & 1 & 1 & 1 \\
	 1 & 1 & 2 & 2 & -2 & -2 & -2 & -2 & -3 & 5 & 1 & 1 & 1 & 1 \\
	 1 & 1 & -2 & -2 & 2 & 2 & -2 & -2 & 1 & 1 & 5 & -3 & 1 & 1 \\
	 1 & 1 & -2 & -2 & 2 & 2 & -2 & -2 & 1 & 1 & -3 & 5 & 1 & 1 \\
	 1 & 1 & -2 & -2 & -2 & -2 & 2 & 2 & 1 & 1 & 1 & 1 & 5 & -3 \\
	 1 & 1 & -2 & -2 & -2 & -2 & 2 & 2 & 1 & 1 & 1 & 1 & -3 & 5 \\
	\end{array}
	\right),$$
where $d=4+\sqrt{17}$ and $\Lambda=4(1+d^2) $ is the global dimension.
Then there is an $8 \times 8 $ block given by 
$$S_{14+k,14+l}=-\frac{2}{\sqrt{17}} \text{cos}\left(\frac{12 \pi l k }{17}\right) , \quad 1 \leq k,l \leq 8 .$$
These blocks only interact in the first two rows/colums. For $ 15 \leq j \leq 22 $, we have
$$S_{1j}=\frac{1}{\sqrt{17}} $$
$$S_{2j}=-\frac{1}{\sqrt{17}} $$
and
$$S_{ij}=0, \ 1 \leq i \leq 14 .$$

This modular data has a very similar form to that of the Haagerup subfactor, and there is hope that it can be generalized into a series as Evans and Gannon did for the Haagerup modular data.

We note that Morrison and Walker have found a purely combinatorial method to find the dimensions of the simple objects in the quantum double of the Asaeda-Haagerup subfactor as well as the induction functors to $\mathcal{AH}_1 $ and $ \mathcal{AH}_2$ \cite{1404.3955}. However their method does not provide an explicit description of the quantum double or give the modular data.

\subsection{Graded extensions}
A grading on a fusion category $\mathcal{D}$ by a group $G$ is a decomposition $\mathcal{D} \cong \bigoplus_{g \in G}\mathcal{D}_g$ such that the tensor product of an object in $\mathcal{D}_g$ and an object in $\mathcal{D}_h$ lands in $\mathcal{D}_{gh}$.  Such a grading is called faithful if the $\mathcal{D}_g$ are all nonzero.  A graded extension of a fusion category $\mathcal{C}$ by a group $G$ is a faithfully $G$-graded fusion category $\mathcal{D}$ such that the identity component $\mathcal{D}_0$ is $\mathcal{C}$.  

Graded extensions by the group $\mathbb{Z}/2\mathbb{Z}$ are important in subfactor theory, where it is natural to ask whether there is a $\mathbb{Z}/2\mathbb{Z}$ extension of the even part of the subfactor such that $\mathcal{D}_1$ is the odd part of the subfactor.  That is, we want to know whether one can find an isomorphism of factors $M$ and $N$ which identifies the two even parts together and the two odd parts together.   In planar algebraic language, this is asking about whether there is an unshaded version of the planar algebra attached to a subfactor.  For example, for group subfactors, such a $\mathbb{Z}/2\mathbb{Z}$-extension is exactly a Tambara--Yamagami category.

Graded extensions of a fusion category $\mathcal{C}$ by a finite group $G$ were classified in \cite{MR2677836} using homotopy theory.  The Brauer--Picard $3$-group (also called a categorical $2$-group) of $\mathcal{C}$, is the full subgroupoid of the Brauer--Picard $3$-groupoid whose only object is $\mathcal{C}$.  (In other words, its only object is $\mathcal{C}$, its $1$-morphisms are Morita auto-equivalences, and its $2$ and $3$ morphisms are defined as in the Brauer--Picard groupoid.)  Then, giving a $G$ extension of $\mathcal{C}$ is the same thing as giving a map of $3$-groups from $G$ to the Brauer--Picard $3$-group of $\mathcal{C}$.  Furthermore, using obstruction theory from algebraic topology, such maps can be classified homologically.  By the usual change-of-basepoint conjugation, a Morita equivalence between $\mathcal{C}$ and $\mathcal{D}$ induces an equivalence between their Brauer--Picard $3$-groups. Therefore, to understand the graded extensions of $\mathcal{C} $, it suffices to understand the graded extensions of any fusion category Morita equivalent to $\mathcal{C}$.

In \cite{GJSext}, we gave a partial description of the Brauer--Picard $3$-group of the Asaeda--Haagerup fusion categories.   In particular, its homotopy groups are $\pi_1 = \mathbb{Z}/2\mathbb{Z} \times \mathbb{Z}/2\mathbb{Z}$, $\pi_2 = 0$, and $\pi_3 = \mathbb{C}^\times$.  For Brauer--Picard $3$-groups of fusion categories, $\pi_1$ automatically acts trivially on $\pi_3$.  Thus in order to completely describe the Brauer--Picard $3$-group of Asaeda--Haagerup, we need only understand a single Postnikov $k$-invariant living in $H^4(\mathbb{Z}/2\mathbb{Z} \times \mathbb{Z}/2\mathbb{Z}, \mathbb{C}^\times)$.  Furthermore, given a map $G \rightarrow \mathbb{Z}/2\mathbb{Z} \times \mathbb{Z}/2\mathbb{Z}$, whether a $G$-extension compatible with that map exists depends only on whether this $k$-invariant becomes trivial in $H^4(G, \mathbb{C}^\times)$.  In particular, since $H^4(\mathbb{Z}/2\mathbb{Z}, \mathcal{C}^\times)$ is trivial, the obstruction automatically vanishes for $\mathbb{Z}/2\mathbb{Z}$ extensions.  Therefore, any Morita autoequivalence of one the Asaeda-Haagerup fusion categories gives rise to a $\mathbb{Z}/2\mathbb{Z} $-graded extension.  (There are two such extensions, since they form a torsor for $H^3(\mathbb{Z}/2\mathbb{Z}, \mathcal{C}^\times)) \cong \mathbb{Z}/2\mathbb{Z} $).  A natural question is whether there exist extensions by the full Brauer-Picard group, $ \mathbb{Z}/2\mathbb{Z} \times \mathbb{Z}/2\mathbb{Z}$.  In \cite{GJSext}, we were unable to resolve this question, since unlike in the case of cyclic groups, $H^4( \mathbb{Z}/2\mathbb{Z} \times \mathbb{Z}/2\mathbb{Z} , \mathbb{C}^{\times})$ is nontrivial.  In particular, we need to know the actual value of the $k$-invariant in this nontrivial homology group.

The results of this paper will allow us to answer these question.  Instead of trying to understand $\mathbb{Z}/2\mathbb{Z} \times \mathbb{Z}/2\mathbb{Z}$ extensions of the original Asaeda--Haagerup fusion categories, we can instead consider extensions of the new category $\mathcal{AH}_4 $ which has a computationally accessible model as endomorphisms of the Cuntz algebra $\mathcal{O}_9 $.   Furthermore, for $\mathcal{AH}_4 $, every Morita autoequivalence is realized by an outer automorphism by Theorem \ref{autothm}, so the graded components of any extension are all trivial as $\mathcal{AH}_4$-module categories.  To realize such an extension, we only need to find a collection of endomorphisms of a factor containing a copy of $\mathcal{AH}_4$ and four additional automorphisms which implement the outer automorphisms of $\mathcal{AH}_4$.  

In a subsequent paper we will explicitly construct the outer automorphisms of $\mathcal{AH}_4 $, and construct a (nontrivial) $\mathbb{Z}/2\mathbb{Z} \times \mathbb{Z}/2\mathbb{Z} $-graded extension of $\mathcal{AH}_4 $. By the extension theory of \cite{MR2677836}, this implies the existence of nontrivial $\mathbb{Z}/2\mathbb{Z} \times \mathbb{Z}/2\mathbb{Z} $-graded extensions for each $\mathcal{AH}_i $, $1 \leq i  \leq 6 $.  In particular,  we can conclude that the $k$-invariant vanishes, so the Brauer--Picard $3$-group of the Asaeda--Haagerup subfactor splits as a product of Eilenberg--MacLane space $K(\mathbb{Z}/2 \times \mathbb{Z}/2, 1) \times K(\mathbb{C}^\times, 3)$.  As a consequence, extensions of by any group $G$ are classified by pairs of a homomorphism from $G$ to the Klein $4$-group together with an element of $H^3(G, \mathbb{C}^\times)$.

\subsection{Intermediate subfactors}

From a complete understanding of the Brauer--Picard groupoid of $\mathcal{C}$, one can read off all division algebras in $\mathcal{C}$ as internal endomorphisms of objects in the module categories over $\mathcal{C}$.  In subfactor language, this is the list of all irreducible overfactors of $N$ which can be built from a fixed finite collection of $N$--$N$ bimodules.   Each division algebra in $\mathcal{C}$ corresponds to the internal endomorphisms for some simple object $\eta$ in a module category in $\mathcal{C}$, and two objects give the same algebra if they are in the same orbit under the action of invertible objects in the dual category.  Thus one can first write down all the module categories (each module category gives a Morita equivalence to its dual category, and this is well-defined up to outer automorphism of the dual category) and then find the orbits of simple objects.  For the Asaeda--Haagerup subfactor our results immediately imply that we have the following number of division algebras $A$ in each of the $\mathcal{AH}_i$ such that the dual category of $A$-$A$ bimodules is $\mathcal{AH}_j$.  (For $1 \leq i, j \leq 3$ all of these subfactors were described in the appendix to \cite{GSbp}.)

$$
\begin{array}{c|c|c|c|c|c|c}
& \multicolumn{6}{c}{\text{\# of div. algs. in $\mathcal{AH}_i$ with dual $\mathcal{AH}_j$}} \\
 & \mathcal{AH}_1 & \mathcal{AH}_2 & \mathcal{AH}_3 & \mathcal{AH}_4 & \mathcal{AH}_5 & \mathcal{AH}_6\\ \hline
\mathcal{AH}_1  & 20 & 13 & 12 & 2 & 2 & 2 \\ \hline
\mathcal{AH}_2 & 18 & 14 & 13 & 2 & 2 & 2 \\ \hline
\mathcal{AH}_3 & 18 & 12  & 14 & 2 & 2 & 2 \\ \hline
\mathcal{AH}_4 & 20 & 8 & 20 & 2 & 2 & 2 \\ \hline
\mathcal{AH}_5 & 16 & 12 & 12 & 2 & 2 & 2 \\ \hline
\mathcal{AH}_6 & 20 & 8 & 12 & 2 & 2 & 2 \\
\end{array}
$$

However, one can go further and use the composition rules in the Brauer--Picard groupoid to understand all inclusions between all of these algebras, and thus all lattices of intermediate subfactors.  We did so in \cite{MR2909758} for the Haagerup fusion categories. Here's a quick sketch of where this relationship comes from.  Suppose that $1 \subset A \subset B$ is an inclusion of division algebras.   Let $\mathcal{D}$ be the fusion category of $A$--$A$ bimodules and $\mathcal{E}$ the fusion category of $B$-$B$ bimodules.  Let $\mathcal{L}$ be the $\mathcal{C}$--$\mathcal{D}$ bimodule category of $1$--$A$ bimodules, let $\mathcal{M}$ be the $\mathcal{D}$--$\mathcal{E}$ bimodule category of $A$--$B$ bimodules, and let $\mathcal{N}$ be the $\mathcal{C}$--$\mathcal{E}$ bimodule category of $1$-$B$ bimodules.  Then $B = A \otimes_A B$ shows that under the tensor product map $\mathcal{L} \boxtimes_\mathcal{D} \mathcal{M} \rightarrow \mathcal{N}$ the tensor product of two simple objects is simple.  Thus intermediate subfactors yield a factorization of an object in some bimodule category as a tensor product of two other objects in bimodule categories.  In the other direction, if $\iota = \eta \psi$ is a decomposition of a simple object $\iota$ in a bimodule category as a tensor product of objects in (possibly different) bimodule categories, then we get an inclusion of division rings $1 \subset \eta \bar{\eta} = \eta 1 \bar{\eta} \subset \eta \psi \bar{\psi} \bar{\eta} = \iota \bar{\iota}$.  Thus there is a correspondence between inclusions of algebras and factorizations of objects in bimodule categories.  In full generality, this correspondence can get a bit delicate as automorphsims of $\mathcal{D}$ (both inner and outer!) play a role, as does the group $\pi_2$ which measures the ambiguity of the tensor product.

In a later paper, we hope to give a full description of this recipe for reading off lattices of intermediate subfactors from tensor product rules between bimodule categories, and to apply this to give a classification of all intermediate subfactor lattices for all division algebras in the Asaeda--Haagerup fusion categories.  It is already known that there are several interesting quadrilaterals involve $AH$, $AH+1$ and $AH+2$, and we hope to find more.  Such a calculation will require computing some additional compositions which we have not yet worked out.

\appendix
\section{Fusion Rules}
In this appendix we list the fusion modules corresponding to the four cases in Theorem \ref{cases}.
Each case corresponds to a triple of fusion modules $(K_{AH_1},L_{AH_2},M_{AH_3} )$. We first recall the fusion rules of $AH_{1-3} $, given by the following tables from \cite{GSbp}.

 $\mathcal{AH}_1 $ has $6$ simple objects, which we will order $1, \chi, \psi,\tau, \zeta,\sigma $, and have dimensions $1,\frac{d+1}{2},\frac{d-1}{2},\frac{3d-1}{2},d,\frac{d+3}{2} $, respectively, where $d=4+\sqrt{17}$.  We use the abbreviation $\Lambda=\psi + \chi + \sigma + \zeta + \tau $. Since the multiplication is commutative, we omit the sub-diagonal entries. 
 
\begingroup \centering
\begin{tabular}{ c || c | c | c | c | c }
                 & $\psi$             & $\chi$       & $\sigma $ & $\zeta$          & $\tau$     \\ \hline \hline
$\psi$              & $1 + \psi + \zeta $             & $\chi + \tau $       & $\zeta + \tau $ & $\psi + \sigma + \zeta + \tau$          & $\chi + \sigma + \zeta + 2\tau$     \\ \hline
$\chi$         &      & $1+ \psi+\chi+ \tau$     & $\sigma + \zeta + \tau $        & $\sigma + \zeta + 2\tau$   & $ \Lambda +\zeta + \tau$  \\ \hline
$\sigma $      &       &             & $1 + \chi + \sigma + \zeta + \tau $   & $\Lambda +\tau$ & $\Lambda + \zeta + 2\tau$           \\ \hline
$\zeta$          &            &     &         & $1+\Lambda + \zeta + 2\tau$ & $\Lambda+ \chi + \sigma + 2\zeta+3\tau$ \\ \hline
$\tau $    &    &           & & & $1+2\Lambda+ \sigma+2\zeta+4\tau$ \\ 
\end{tabular}
\captionof{table}{ $AH_1$ multiplication table.}
\endgroup

$\mathcal{AH}_2 $ has $9$ simple objects, which we will order $1, \alpha, \rho, \alpha \rho, \rho \alpha, \alpha \rho \alpha, \pi, \alpha \pi, \eta $, and have dimensions $1,1,\frac{d-1}{2},\frac{d-1}{2},\frac{d-1}{2},\frac{d-1}{2},d,d,d+1 $, respectively.   The rules involving $\alpha$ are: $\alpha^2=1$, $\pi \alpha=\alpha \pi$, $\alpha \eta = \eta\alpha=\eta$, $\rho \alpha \rho = \alpha \rho \alpha + \eta $.  We use the abbreviations $\Gamma = \pi + \alpha \pi + \eta$, $\Delta= \rho + \alpha \rho + \rho \alpha + \alpha \rho \alpha $.

\begingroup \centering
\begin{tabular}{ c || c | c | c}
                 & $\rho$             & $\pi$       & $\eta $  \\ \hline \hline
$\rho$          & $1 + \rho + \pi$ & $\rho + \Gamma$       & $\alpha \rho + \alpha \rho \alpha + \Gamma$  \\ \hline
$\pi$         & $\rho + \Gamma$        & $1 + \Delta + 2 \Gamma$     & $\Delta + 2 \Gamma + \eta $        \\ \hline
$\eta$      & $\rho \alpha + \alpha \rho \alpha + \Gamma$      & $\Delta+2\Gamma+ \eta$            & $1+\alpha+\Delta+2\Gamma+\pi+\alpha\pi $  \\ 
\end{tabular}
\captionof{table}{ $AH_2$ partial multiplication table.} 
\endgroup

 $\mathcal{AH}_3 $ has $9$ simple objects, which we will order $1, \beta, \xi, \beta \xi, \xi \beta, \beta \xi \beta, \mu, \beta \mu, \nu $, and have dimensions $1,1,\frac{d+1}{2},\frac{d+1}{2},\frac{d+1}{2},\frac{d+1}{2},d,d,d-1 $, respectively.   The rules involving $\beta$ are: $\beta^2=1$, $\mu \beta=\beta \mu$, $\beta \nu = \nu\beta=\nu$, $\xi \beta \xi = \beta \xi \beta + \mu +\beta \mu$.  We use the abbreviations $\Pi = \mu + \beta \mu + \nu, \ \Psi= \xi + \beta \xi + \xi \beta + \beta \xi \beta $.

\begingroup \centering
\begin{tabular}{ c || c | c | c}
                 & $\xi$             & $\mu$       & $\nu $  \\ \hline \hline
$\xi$          & $1 + \xi + \mu+ \nu$ & $\xi + \beta \xi + \beta \xi \beta + \Pi$       & $\xi + \xi \beta + \Pi$  \\ \hline
$\mu$         & $\xi + \xi \beta + \beta \xi \beta + \Pi $        & $1 + \Pi + 2 \Psi $     & $\Pi + \Psi + \mu + \beta\mu $        \\ \hline
$\nu$      & $\xi + \beta \xi + \Pi$      & $\Pi + \Psi + \mu + \beta \mu$            & $1+\beta+ \Pi + \Psi + \nu$  \\ 
\end{tabular}
\captionof{table}{ $AH_3$ partial multiplication table.} 
\endgroup

Each fusion module will be represented as $n$ adjacent $m \times n $ integer matrices (separated by vertical lines), where $m$ is the number of basis elements of $AH_l$ and n is the number of basis elements of the module category$K$. The $ij^{th} $ entry of the $k^{th} $ matrix gives the multiplicity $(R_k X_i, R_j ) $, where the $\{R_j\} $ are basis elements of $ K$ and the $ X_i$ are basis elements of $AH_l $. For each case the fusion modules are listed in the following order: first $ K_{AH_1}^{(i)}$, then $L_{AH_2}^{(i)}$, then $M_{AH_3}^{(i)}$. With this notation, here are the four cases. Recall that cases (4)-(7) correspond to cases (c), (a), (b), and (d) respectively in \cite[Theorem 6.10]{GSbp}..

Case (4): \nopagebreak

\renewcommand{\arraystretch}{1.5} 
\resizebox{\linewidth}{!}{%
$\left(
\begin{array}{c|ccccc|ccccc|ccccc|ccccc|ccccc}
& \multicolumn{2}{l}{\theta_{41}^1 \otimes \_} & & & & \multicolumn{2}{l}{\theta_{41}^2 \otimes \_} & & & & \multicolumn{2}{l}{\theta_{41}^3 \otimes \_} & & & & \multicolumn{2}{l}{\theta_{41}^4 \otimes \_} & & & & \multicolumn{2}{l}{\theta_{41}^5 \otimes \_} & & & \\
L_{\mathcal{AH}_1}^{(4)} & \theta_{41}^1 & \theta_{41}^2 & \theta_{41}^3 & \theta_{41}^4 & \theta_{41}^5  & \theta_{41}^1 & \theta_{41}^2 & \theta_{41}^3 & \theta_{41}^4 & \theta_{41}^5  & \theta_{41}^1 & \theta_{41}^2 & \theta_{41}^3 & \theta_{41}^4 & \theta_{41}^5  & \theta_{41}^1 & \theta_{41}^2 & \theta_{41}^3 & \theta_{41}^4 & \theta_{41}^5  & \theta_{41}^1 & \theta_{41}^2 & \theta_{41}^3 & \theta_{41}^4 & \theta_{41}^5\\ \hline
 1 &1 & 0 & 0 & 0 & 0 & 0 & 1 & 0 & 0 & 0 & 0 & 0 & 1 & 0 & 0 & 0 & 0 & 0 & 1 & 0 & 0 & 0
 & 0 & 0 & 1 \\
 \chi & 0 & 1 & 1 & 1 & 1 & 1 & 0 & 1 & 1 & 1 & 1 & 1 & 0 & 1 & 1 & 1 & 1 & 1 & 0 & 1 & 1 & 1
   & 1 & 1 & 2 \\
\psi &1 & 0 & 0 & 1 & 1 & 0 & 1 & 1 & 0 & 1 & 0 & 1 & 1 & 0 & 1 & 1 & 0 & 0 & 1 & 1 & 1 & 1
 & 1 & 1 & 1 \\
\tau & 2 & 2 & 2 & 1 & 3 & 2 & 2 & 1 & 2 & 3 & 2 & 1 & 2 & 2 & 3 & 1 & 2 & 2 & 2 & 3 & 3 & 3
   & 3 & 3 & 4 \\
\zeta & 1 & 1 & 1 & 2 & 2 & 1 & 1 & 2 & 1 & 2 & 1 & 2 & 1 & 1 & 2 & 2 & 1 & 1 & 1 & 2 & 2 & 2
   & 2 & 2 & 3 \\
\sigma & 1 & 1 & 1 & 1 & 1 & 1 & 1 & 1 & 1 & 1 & 1 & 1 & 1 & 1 & 1 & 1 & 1 & 1 & 1 & 1 & 1 & 1
   & 1 & 1 & 3 \\
\end{array}
\right)$
} \\

\renewcommand{\arraystretch}{1.5} 
\resizebox{\linewidth}{!}{%
$\left(
\begin{array}{c|cccccc|cccccc|cccccc|cccccc|cccccc|cccccc}
& \multicolumn{2}{l}{\theta_{42}^1 \otimes \_} & & & & & \multicolumn{2}{l}{\theta_{42}^2 \otimes \_} & & & & & \multicolumn{2}{l}{\theta_{42}^3 \otimes \_} & & & & & \multicolumn{2}{l}{\theta_{42}^4 \otimes \_} & & & & & \multicolumn{2}{l}{\theta_{42}^5 \otimes \_ }& & & & &  \multicolumn{2}{l}{\theta_{42}^6 \otimes \_} & & & & \\
M_{\mathcal{AH}_2}^{(4)} & \theta_{42}^1 & \theta_{42}^2 & \theta_{42}^3 & \theta_{42}^4 & \theta_{42}^5 & \theta_{42}^6 & \theta_{42}^1 & \theta_{42}^2 & \theta_{42}^3 & \theta_{42}^4 & \theta_{42}^5 & \theta_{42}^6  & \theta_{42}^1 & \theta_{42}^2 & \theta_{42}^3 & \theta_{42}^4 & \theta_{42}^5 & \theta_{42}^6 & \theta_{42}^1 & \theta_{42}^2 & \theta_{42}^3 & \theta_{42}^4 & \theta_{42}^5 & \theta_{42}^6 & \theta_{42}^1 & \theta_{42}^2 & \theta_{42}^3 & \theta_{42}^4 & \theta_{42}^5 & \theta_{42}^6 & \theta_{42}^1 & \theta_{42}^2 & \theta_{42}^3 & \theta_{42}^4 & \theta_{42}^5 & \theta_{42}^6 \\ \hline
1 & 1 & 0 & 0 & 0 & 0 & 0 & 0 & 1 & 0 & 0 & 0 & 0 & 0 & 0 & 1 & 0 & 0 & 0 & 0 & 0 & 0 & 1
   & 0 & 0 & 0 & 0 & 0 & 0 & 1 & 0 & 0 & 0 & 0 & 0 & 0 & 1 \\
\alpha & 0 & 0 & 0 & 1 & 0 & 0 & 0 & 0 & 1 & 0 & 0 & 0 & 0 & 1 & 0 & 0 & 0 & 0 & 1 & 0 & 0 & 0
   & 0 & 0 & 0 & 0 & 0 & 0 & 0 & 1 & 0 & 0 & 0 & 0 & 1 & 0 \\
\rho & 0 & 0 & 0 & 0 & 0 & 1 & 0 & 0 & 0 & 0 & 1 & 0 & 0 & 0 & 0 & 0 & 1 & 0 & 0 & 0 & 0 & 0
   & 0 & 1 & 0 & 1 & 1 & 0 & 2 & 1 & 1 & 0 & 0 & 1 & 1 & 2 \\
\alpha \rho & 0 & 0 & 0 & 0 & 0 & 1 & 0 & 0 & 0 & 0 & 1 & 0 & 0 & 0 & 0 & 0 & 1 & 0 & 0 & 0 & 0 & 0
   & 0 & 1 & 1 & 0 & 0 & 1 & 1 & 2 & 0 & 1 & 1 & 0 & 2 & 1 \\
\rho \alpha & 0 & 0 & 0 & 0 & 1 & 0 & 0 & 0 & 0 & 0 & 0 & 1 & 0 & 0 & 0 & 0 & 0 & 1 & 0 & 0 & 0 & 0
   & 1 & 0 & 0 & 1 & 1 & 0 & 1 & 2 & 1 & 0 & 0 & 1 & 2 & 1 \\
\alpha \rho \alpha & 0 & 0 & 0 & 0 & 1 & 0 & 0 & 0 & 0 & 0 & 0 & 1 & 0 & 0 & 0 & 0 & 0 & 1 & 0 & 0 & 0 & 0
   & 1 & 0 & 1 & 0 & 0 & 1 & 2 & 1 & 0 & 1 & 1 & 0 & 1 & 2 \\
\pi & 0 & 0 & 0 & 1 & 1 & 1 & 0 & 0 & 1 & 0 & 1 & 1 & 0 & 1 & 0 & 0 & 1 & 1 & 1 & 0 & 0 & 0
   & 1 & 1 & 1 & 1 & 1 & 1 & 4 & 3 & 1 & 1 & 1 & 1 & 3 & 4 \\
\alpha \pi & 1 & 0 & 0 & 0 & 1 & 1 & 0 & 1 & 0 & 0 & 1 & 1 & 0 & 0 & 1 & 0 & 1 & 1 & 0 & 0 & 0 & 1
   & 1 & 1 & 1 & 1 & 1 & 1 & 3 & 4 & 1 & 1 & 1 & 1 & 4 & 3 \\
\eta & 0 & 1 & 1 & 0 & 1 & 1 & 1 & 0 & 0 & 1 & 1 & 1 & 1 & 0 & 0 & 1 & 1 & 1 & 0 & 1 & 1 & 0
   & 1 & 1 & 1 & 1 & 1 & 1 & 4 & 4 & 1 & 1 & 1 & 1 & 4 & 4 \\
\end{array}
\right)$
} \\

$\left(
\begin{array}{c|ccc|ccc|ccc}
& \multicolumn{2}{l}{\theta_{43}^1 \otimes \_} & & \multicolumn{2}{l}{\theta_{43}^2 \otimes \_}  & & \multicolumn{2}{l}{\theta_{43}^3 \otimes \_} & \\
N_{\mathcal{AH}_3}^{(4)} & \theta_{43}^1 & \theta_{43}^2 & \theta_{43}^3 & \theta_{43}^1 & \theta_{43}^2 & \theta_{43}^3 & \theta_{43}^1 & \theta_{43}^2 & \theta_{43}^3 \\ \hline
 1 & 1 & 0 & 0 & 0 & 1 & 0 & 0 & 0 & 1 \\
 \beta & 1 & 0 & 0 & 0 & 1 & 0 & 0 & 0 & 1 \\
 \xi & 0 & 1 & 1 & 1 & 0 & 1 & 1 & 1 & 4 \\
 \beta \xi & 0 & 1 & 1 & 1 & 0 & 1 & 1 & 1 & 4 \\
 \xi \beta & 0 & 1 & 1 & 1 & 0 & 1 & 1 & 1 & 4 \\
 \beta \xi \beta & 0 & 1 & 1 & 1 & 0 & 1 & 1 & 1 & 4 \\
 \mu & 1 & 0 & 2 & 0 & 1 & 2 & 2 & 2 & 7 \\
 \beta \mu & 1 & 0 & 2 & 0 & 1 & 2 & 2 & 2 & 7 \\
 \nu & 0 & 0 & 2 & 0 & 0 & 2 & 2 & 2 & 6 \\
\end{array}
\right)$\\

Case (5):\nopagebreak

\renewcommand{\arraystretch}{1.5} 
\resizebox{\linewidth}{!}{%
$\left(
\begin{array}{c|cccc|cccc|cccc|cccc}
& \multicolumn{2}{l}{\theta_{51}^1 \otimes \_} & & & \multicolumn{2}{l}{\theta_{51}^2 \otimes \_} & & & \multicolumn{2}{l}{\theta_{51}^3 \otimes \_} & & & \multicolumn{2}{l}{\theta_{51}^4 \otimes \_} & & \\
L_{\mathcal{AH}_1}^{(5)} & \theta_{51}^1 & \theta_{51}^2 & \theta_{51}^3 & \theta_{51}^4  & \theta_{51}^1 & \theta_{51}^2 & \theta_{51}^3 & \theta_{51}^4 & \theta_{51}^1 & \theta_{51}^2 & \theta_{51}^3 & \theta_{51}^4 & \theta_{51}^1 & \theta_{51}^2 & \theta_{51}^3 & \theta_{51}^4 \\ \hline
 1 & 1 & 0 & 0 & 0 & 0 & 1 & 0 & 0 & 0 & 0 & 1 & 0 & 0 & 0 & 0 & 1 \\
 \chi & 2 & 0 & 1 & 1 & 0 & 2 & 1 & 1 & 1 & 1 & 1 & 2 & 1 & 1 & 2 & 1 \\
 \psi & 1 & 0 & 1 & 1 & 0 & 1 & 1 & 1 & 1 & 1 & 2 & 0 & 1 & 1 & 0 & 2 \\
 \tau & 2 & 2 & 3 & 3 & 2 & 2 & 3 & 3 & 3 & 3 & 3 & 4 & 3 & 3 & 4 & 3 \\
 \zeta & 1 & 2 & 2 & 2 & 2 & 1 & 2 & 2 & 2 & 2 & 3 & 2 & 2 & 2 & 2 & 3 \\
 \sigma & 1 & 2 & 1 & 1 & 2 & 1 & 1 & 1 & 1 & 1 & 2 & 2 & 1 & 1 & 2 & 2 \\
\end{array}
\right)$
}\\

\renewcommand{\arraystretch}{1.5} 
\resizebox{\linewidth}{!}{%
$\left(
\begin{array}{c|cccccc|cccccc|cccccc|cccccc|cccccc|cccccc}
& \multicolumn{2}{l}{\theta_{52}^1 \otimes \_} & & & & & \multicolumn{2}{l}{\theta_{52}^2 \otimes \_} & & & & & \multicolumn{2}{l}{\theta_{52}^3 \otimes \_} & & & & & \multicolumn{2}{l}{\theta_{52}^4 \otimes \_} & & & & & \multicolumn{2}{l}{\theta_{52}^5 \otimes \_} & & & & &  \multicolumn{2}{l}{\theta_{52}^6 \otimes \_} & & & & \\
M_{\mathcal{AH}_2}^{(5)} & \theta_{52}^1 & \theta_{52}^2 & \theta_{52}^3 & \theta_{52}^4  & \theta_{52}^5 & \theta_{52}^6  & \theta_{52}^1 & \theta_{52}^2 & \theta_{52}^3 & \theta_{52}^4  & \theta_{52}^5 & \theta_{52}^6  
 & \theta_{52}^1 & \theta_{52}^2 & \theta_{52}^3 & \theta_{52}^4  & \theta_{52}^5 & \theta_{52}^6  
 & \theta_{52}^1 & \theta_{52}^2 & \theta_{52}^3 & \theta_{52}^4  & \theta_{52}^5 & \theta_{52}^6  
 & \theta_{52}^1 & \theta_{52}^2 & \theta_{52}^3 & \theta_{52}^4  & \theta_{52}^5 & \theta_{52}^6  
 & \theta_{52}^1 & \theta_{52}^2 & \theta_{52}^3 & \theta_{52}^4  & \theta_{52}^5 & \theta_{52}^6   \\ \hline
 1 & 1 & 0 & 0 & 0 & 0 & 0 & 0 & 1 & 0 & 0 & 0 & 0 & 0 & 0 & 1 & 0 & 0 & 0 & 0 & 0 & 0 & 1
    & 0 & 0 & 0 & 0 & 0 & 0 & 1 & 0 & 0 & 0 & 0 & 0 & 0 & 1 \\
 \alpha & 1 & 0 & 0 & 0 & 0 & 0 & 0 & 1 & 0 & 0 & 0 & 0 & 0 & 0 & 0 & 0 & 0 & 1 & 0 & 0 & 0 & 0
   & 1 & 0 & 0 & 0 & 0 & 1 & 0 & 0 & 0 & 0 & 1 & 0 & 0 & 0 \\
\rho & 0 & 0 & 0 & 0 & 1 & 1 & 0 & 0 & 1 & 1 & 0 & 0 & 0 & 1 & 1 & 1 & 0 & 1 & 0 & 1 & 1 & 1
   & 1 & 0 & 1 & 0 & 0 & 1 & 1 & 1 & 1 & 0 & 1 & 0 & 1 & 1 \\
 \alpha \rho & 0 & 0 & 0 & 0 & 1 & 1 & 0 & 0 & 1 & 1 & 0 & 0 & 1 & 0 & 1 & 0 & 1 & 1 & 1 & 0 & 0 & 1
   & 1 & 1 & 0 & 1 & 1 & 1 & 1 & 0 & 0 & 1 & 1 & 1 & 0 & 1 \\
 \rho \alpha & 0 & 0 & 1 & 1 & 0 & 0 & 0 & 0 & 0 & 0 & 1 & 1 & 0 & 1 & 1 & 0 & 1 & 1 & 0 & 1 & 0 & 1
   & 1 & 1 & 1 & 0 & 1 & 1 & 1 & 0 & 1 & 0 & 1 & 1 & 0 & 1 \\
\alpha \rho \alpha &  0 & 0 & 1 & 1 & 0 & 0 & 0 & 0 & 0 & 0 & 1 & 1 & 1 & 0 & 1 & 1 & 0 & 1 & 1 & 0 & 1 & 1
   & 1 & 0 & 0 & 1 & 0 & 1 & 1 & 1 & 0 & 1 & 1 & 0 & 1 & 1 \\
\pi & 1 & 0 & 1 & 1 & 1 & 1 & 0 & 1 & 1 & 1 & 1 & 1 & 1 & 1 & 2 & 2 & 2 & 1 & 1 & 1 & 2 & 2
   & 1 & 2 & 1 & 1 & 2 & 1 & 2 & 2 & 1 & 1 & 1 & 2 & 2 & 2 \\
 \alpha \pi & 1 & 0 & 1 & 1 & 1 & 1 & 0 & 1 & 1 & 1 & 1 & 1 & 1 & 1 & 1 & 2 & 2 & 2 & 1 & 1 & 2 & 1
   & 2 & 2 & 1 & 1 & 2 & 2 & 1 & 2 & 1 & 1 & 2 & 2 & 2 & 1 \\
 \eta & 0 & 2 & 1 & 1 & 1 & 1 & 2 & 0 & 1 & 1 & 1 & 1 & 1 & 1 & 2 & 2 & 2 & 2 & 1 & 1 & 2 & 2
   & 2 & 2 & 1 & 1 & 2 & 2 & 2 & 2 & 1 & 1 & 2 & 2 & 2 & 2 \\
\end{array}
\right)$
}\\

\renewcommand{\arraystretch}{1.5} 
\resizebox{\linewidth}{!}{%
$\left(
\begin{array}{c|cccccc|cccccc|cccccc|cccccc|cccccc|cccccc}
& \multicolumn{2}{l}{\theta_{53}^1 \otimes \_} & & & & & \multicolumn{2}{l}{\theta_{53}^2 \otimes \_} & & & & & \multicolumn{2}{l}{\theta_{53}^3 \otimes \_} & & & & & \multicolumn{2}{l}{\theta_{53}^4 \otimes \_} & & & & & \multicolumn{2}{l}{\theta_{53}^5 \otimes \_} & & & & &  \multicolumn{2}{l}{\theta_{53}^6 \otimes \_} & & & & \\
 N_{\mathcal{AH}_3}^{(5)} & \theta_{53}^1 & \theta_{53}^2 & \theta_{53}^3 & \theta_{53}^4  & \theta_{53}^5 & \theta_{53}^6  & \theta_{53}^1 & \theta_{53}^2 & \theta_{53}^3 & \theta_{53}^4  & \theta_{53}^5 & \theta_{53}^6  
 & \theta_{53}^1 & \theta_{53}^2 & \theta_{53}^3 & \theta_{53}^4  & \theta_{53}^5 & \theta_{53}^6  
 & \theta_{53}^1 & \theta_{53}^2 & \theta_{53}^3 & \theta_{53}^4  & \theta_{53}^5 & \theta_{53}^6  
 & \theta_{53}^1 & \theta_{53}^2 & \theta_{53}^3 & \theta_{53}^4  & \theta_{53}^5 & \theta_{53}^6  
 & \theta_{53}^1 & \theta_{53}^2 & \theta_{53}^3 & \theta_{53}^4  & \theta_{53}^5 & \theta_{53}^6   \\ \hline
 1&1 & 0 & 0 & 0 & 0 & 0 & 0 & 1 & 0 & 0 & 0 & 0 & 0 & 0 & 1 & 0 & 0 & 0 & 0 & 0 & 0 & 1
   & 0 & 0 & 0 & 0 & 0 & 0 & 1 & 0 & 0 & 0 & 0 & 0 & 0 & 1 \\
 \beta & 0 & 0 & 0 & 1 & 0 & 0 & 0 & 0 & 1 & 0 & 0 & 0 & 0 & 1 & 0 & 0 & 0 & 0 & 1 & 0 & 0 & 0
   & 0 & 0 & 0 & 0 & 0 & 0 & 1 & 0 & 0 & 0 & 0 & 0 & 0 & 1 \\
 \xi & 0 & 0 & 1 & 0 & 0 & 1 & 0 & 0 & 0 & 1 & 1 & 0 & 1 & 0 & 0 & 0 & 0 & 1 & 0 & 1 & 0 & 0
   & 1 & 0 & 0 & 1 & 0 & 1 & 2 & 2 & 1 & 0 & 1 & 0 & 2 & 2 \\
 \beta \xi & 0 & 1 & 0 & 0 & 1 & 0 & 1 & 0 & 0 & 0 & 0 & 1 & 0 & 0 & 0 & 1 & 1 & 0 & 0 & 0 & 1 & 0
   & 0 & 1 & 0 & 1 & 0 & 1 & 2 & 2 & 1 & 0 & 1 & 0 & 2 & 2 \\
 \xi \beta & 0 & 1 & 0 & 0 & 0 & 1 & 1 & 0 & 0 & 0 & 1 & 0 & 0 & 0 & 0 & 1 & 0 & 1 & 0 & 0 & 1 & 0
   & 1 & 0 & 1 & 0 & 1 & 0 & 2 & 2 & 0 & 1 & 0 & 1 & 2 & 2 \\
 \beta \xi \beta & 0 & 0 & 1 & 0 & 1 & 0 & 0 & 0 & 0 & 1 & 0 & 1 & 1 & 0 & 0 & 0 & 1 & 0 & 0 & 1 & 0 & 0
   & 0 & 1 & 1 & 0 & 1 & 0 & 2 & 2 & 0 & 1 & 0 & 1 & 2 & 2 \\
 \mu & 1 & 0 & 0 & 0 & 1 & 1 & 0 & 1 & 0 & 0 & 1 & 1 & 0 & 0 & 1 & 0 & 1 & 1 & 0 & 0 & 0 & 1
   & 1 & 1 & 1 & 1 & 1 & 1 & 3 & 4 & 1 & 1 & 1 & 1 & 4 & 3 \\
 \beta \mu & 0 & 0 & 0 & 1 & 1 & 1 & 0 & 0 & 1 & 0 & 1 & 1 & 0 & 1 & 0 & 0 & 1 & 1 & 1 & 0 & 0 & 0
   & 1 & 1 & 1 & 1 & 1 & 1 & 3 & 4 & 1 & 1 & 1 & 1 & 4 & 3 \\
 \nu & 0 & 0 & 0 & 0 & 1 & 1 & 0 & 0 & 0 & 0 & 1 & 1 & 0 & 0 & 0 & 0 & 1 & 1 & 0 & 0 & 0 & 0
   & 1 & 1 & 1 & 1 & 1 & 1 & 4 & 2 & 1 & 1 & 1 & 1 & 2 & 4 \\
\end{array}
\right)$
}\\

Case (6):\nopagebreak

\renewcommand{\arraystretch}{1.5} 
\resizebox{\linewidth}{!}{%
$\left(
\begin{array}{c|ccccc|ccccc|ccccc|ccccc|ccccc}
& \multicolumn{2}{l}{\theta_{61}^1 \otimes \_} & & & & \multicolumn{2}{l}{\theta_{61}^2 \otimes \_} & & & & \multicolumn{2}{l}{\theta_{61}^3 \otimes \_} & & & & \multicolumn{2}{l}{\theta_{61}^4 \otimes \_} & & & & \multicolumn{2}{l}{\theta_{61}^5 \otimes \_} & & & \\
L_{\mathcal{AH}_1}^{(6)} & \theta_{61}^1 & \theta_{61}^2 & \theta_{61}^3 & \theta_{61}^4 & \theta_{61}^5  & \theta_{61}^1 & \theta_{61}^2 & \theta_{61}^3 & \theta_{61}^4 & \theta_{61}^5  & \theta_{61}^1 & \theta_{61}^2 & \theta_{61}^3 & \theta_{61}^4 & \theta_{61}^5  & \theta_{61}^1 & \theta_{61}^2 & \theta_{61}^3 & \theta_{61}^4 & \theta_{61}^5  & \theta_{61}^1 & \theta_{61}^2 & \theta_{61}^3 & \theta_{61}^4 & \theta_{61}^5\\ \hline
 1 &1 & 0 & 0 & 0 & 0 & 0 & 1 & 0 & 0 & 0 & 0 & 0 & 1 & 0 & 0 & 0 & 0 & 0 & 1 & 0 & 0 & 0
   & 0 & 0 & 1 \\
 \chi & 1 & 0 & 0 & 1 & 1 & 0 & 1 & 1 & 0 & 1 & 0 & 1 & 1 & 0 & 1 & 1 & 0 & 0 & 1 & 1 & 1 & 1
   & 1 & 1 & 3 \\
 \psi & 0 & 0 & 0 & 1 & 1 & 0 & 0 & 1 & 0 & 1 & 0 & 1 & 0 & 0 & 1 & 1 & 0 & 0 & 0 & 1 & 1 & 1
   & 1 & 1 & 2 \\
 \tau & 1 & 1 & 1 & 1 & 3 & 1 & 1 & 1 & 1 & 3 & 1 & 1 & 1 & 1 & 3 & 1 & 1 & 1 & 1 & 3 & 3 & 3
   & 3 & 3 & 7 \\
 \zeta & 1 & 1 & 1 & 0 & 2 & 1 & 1 & 0 & 1 & 2 & 1 & 0 & 1 & 1 & 2 & 0 & 1 & 1 & 1 & 2 & 2 & 2
   & 2 & 2 & 5 \\
 \sigma & 0 & 1 & 1 & 1 & 1 & 1 & 0 & 1 & 1 & 1 & 1 & 1 & 0 & 1 & 1 & 1 & 1 & 1 & 0 & 1 & 1 & 1
   & 1 & 1 & 4 \\
\end{array}
\right)$
}\\

$ \left(
\begin{array}{c|ccc|ccc|ccc}
& \multicolumn{2}{l}{\theta_{62}^1 \otimes \_} & & \multicolumn{2}{l}{\theta_{62}^2 \otimes \_}  & & \multicolumn{2}{l}{\theta_{62}^3 \otimes \_} & \\
M_{\mathcal{AH}_2}^{(6)} & \theta_{62}^1 & \theta_{62}^2 & \theta_{62}^3 & \theta_{62}^1 & \theta_{62}^2 & \theta_{62}^3 & \theta_{62}^1 & \theta_{62}^2 & \theta_{62}^3 \\ \hline
1 & 1 & 0 & 0 & 0 & 1 & 0 & 0 & 0 & 1 \\
 \alpha & 1 & 0 & 0 & 0 & 1 & 0 & 0 & 0 & 1 \\
 \rho & 0 & 1 & 1 & 1 & 2 & 1 & 1 & 1 & 2 \\
 \alpha \rho & 0 & 1 & 1 & 1 & 2 & 1 & 1 & 1 & 2 \\
 \rho \alpha & 0 & 1 & 1 & 1 & 2 & 1 & 1 & 1 & 2 \\
 \alpha \rho \alpha & 0 & 1 & 1 & 1 & 2 & 1 & 1 & 1 & 2 \\
 \pi& 1 & 2 & 2 & 2 & 3 & 4 & 2 & 4 & 3 \\
 \alpha \pi & 1 & 2 & 2 & 2 & 3 & 4 & 2 & 4 & 3 \\
 \eta & 2 & 2 & 2 & 2 & 4 & 4 & 2 & 4 & 4 \\
\end{array}
\right)$\\

\renewcommand{\arraystretch}{1.5} 
\resizebox{\linewidth}{!}{%
$\left(
\begin{array}{c|cccccc|cccccc|cccccc|cccccc|cccccc|cccccc}
& \multicolumn{2}{l}{\theta_{63}^1 \otimes \_} & & & & & \multicolumn{2}{l}{\theta_{63}^2 \otimes \_} & & & & & \multicolumn{2}{l}{\theta_{63}^3 \otimes \_} & & & & & \multicolumn{2}{l}{\theta_{63}^4 \otimes \_} & & & & & \multicolumn{2}{l}{\theta_{63}^5 \otimes \_} & & & & & \multicolumn{2}{l}{\theta_{63}^6 \otimes \_} & & & & \\
N_{\mathcal{AH}_3}^{(6)} & \theta_{63}^1 & \theta_{63}^2 & \theta_{63}^3 & \theta_{63}^4  & \theta_{63}^5 & \theta_{63}^6  & \theta_{63}^1 & \theta_{63}^2 & \theta_{63}^3 & \theta_{63}^4  & \theta_{63}^5 & \theta_{63}^6  
 & \theta_{63}^1 & \theta_{63}^2 & \theta_{63}^3 & \theta_{63}^4  & \theta_{63}^5 & \theta_{63}^6  
 & \theta_{63}^1 & \theta_{63}^2 & \theta_{63}^3 & \theta_{63}^4  & \theta_{63}^5 & \theta_{63}^6  
 & \theta_{63}^1 & \theta_{63}^2 & \theta_{63}^3 & \theta_{63}^4  & \theta_{63}^5 & \theta_{63}^6  
 & \theta_{63}^1 & \theta_{63}^2 & \theta_{63}^3 & \theta_{63}^4  & \theta_{63}^5 & \theta_{63}^6   \\ \hline
1& 1 & 0 & 0 & 0 & 0 & 0 & 0 & 1 & 0 & 0 & 0 & 0 & 0 & 0 & 1 & 0 & 0 & 0 & 0 & 0 & 0 & 1
    & 0 & 0 & 0 & 0 & 0 & 0 & 1 & 0 & 0 & 0 & 0 & 0 & 0 & 1 \\
\beta & 0 & 1 & 0 & 0 & 0 & 0 & 1 & 0 & 0 & 0 & 0 & 0 & 0 & 0 & 0 & 1 & 0 & 0 & 0 & 0 & 1 & 0
   & 0 & 0 & 0 & 0 & 0 & 0 & 0 & 1 & 0 & 0 & 0 & 0 & 1 & 0 \\
 \xi & 1 & 0 & 0 & 0 & 1 & 1 & 0 & 1 & 1 & 1 & 0 & 0 & 0 & 1 & 1 & 1 & 1 & 1 & 0 & 1 & 1 & 1
   & 1 & 1 & 1 & 0 & 1 & 1 & 1 & 1 & 1 & 0 & 1 & 1 & 1 & 1 \\
 \beta \xi &0 & 1 & 1 & 1 & 0 & 0 & 1 & 0 & 0 & 0 & 1 & 1 & 0 & 1 & 1 & 1 & 1 & 1 & 0 & 1 & 1 & 1
   & 1 & 1 & 1 & 0 & 1 & 1 & 1 & 1 & 1 & 0 & 1 & 1 & 1 & 1 \\
\xi \beta & 0 & 1 & 0 & 0 & 1 & 1 & 1 & 0 & 1 & 1 & 0 & 0 & 1 & 0 & 1 & 1 & 1 & 1 & 1 & 0 & 1 & 1
   & 1 & 1 & 0 & 1 & 1 & 1 & 1 & 1 & 0 & 1 & 1 & 1 & 1 & 1 \\
\beta \xi \beta & 1 & 0 & 1 & 1 & 0 & 0 & 0 & 1 & 0 & 0 & 1 & 1 & 1 & 0 & 1 & 1 & 1 & 1 & 1 & 0 & 1 & 1
   & 1 & 1 & 0 & 1 & 1 & 1 & 1 & 1 & 0 & 1 & 1 & 1 & 1 & 1 \\
  \mu &1 & 0 & 1 & 1 & 1 & 1 & 0 & 1 & 1 & 1 & 1 & 1 & 1 & 1 & 1 & 2 & 2 & 2 & 1 & 1 & 2 & 1
   & 2 & 2 & 1 & 1 & 2 & 2 & 1 & 2 & 1 & 1 & 2 & 2 & 2 & 1 \\
\beta \mu &  0 & 1 & 1 & 1 & 1 & 1 & 1 & 0 & 1 & 1 & 1 & 1 & 1 & 1 & 2 & 1 & 2 & 2 & 1 & 1 & 1 & 2
   & 2 & 2 & 1 & 1 & 2 & 2 & 2 & 1 & 1 & 1 & 2 & 2 & 1 & 2 \\
 \nu & 0 & 0 & 1 & 1 & 1 & 1 & 0 & 0 & 1 & 1 & 1 & 1 & 1 & 1 & 2 & 2 & 1 & 1 & 1 & 1 & 2 & 2
   & 1 & 1 & 1 & 1 & 1 & 1 & 2 & 2 & 1 & 1 & 1 & 1 & 2 & 2 \\
\end{array}
\right)$
}\\

Case (7):\nopagebreak

$\left(
\begin{array}{c|cc|cc}
& \multicolumn{2}{l}{\theta_{71}^1 \otimes \_} & \multicolumn{2}{l}{\theta_{71}^2  \otimes \_} \\
L_{\mathcal{AH}_1}^{(7)} & \theta_{71}^1 & \theta_{71}^2 & \theta_{71}^1 & \theta_{71}^2  \\ \hline
1 & 1 & 0 & 0 & 1 \\
 \chi & 2 & 2 & 2 & 3 \\
 \psi &1 & 2 & 2 & 2 \\
 \tau & 4 & 6 & 6 & 7 \\
 \zeta & 3 & 4 & 4 & 5 \\
 \sigma & 3 & 2 & 2 & 4 \\
\end{array}
\right) \quad  \left(
\begin{array}{c|ccc|ccc|ccc}
& \multicolumn{2}{l}{\theta_{72}^1 \otimes \_} & & \multicolumn{2}{l}{\theta_{72}^2 \otimes \_}  & & \multicolumn{2}{l}{\theta_{72}^3 \otimes \_} & \\
M_{\mathcal{AH}_2}^{(7)} & \theta_{72}^1 & \theta_{72}^2 & \theta_{72}^3 & \theta_{72}^1 & \theta_{72}^2 & \theta_{72}^3 & \theta_{72}^1 & \theta_{72}^2 & \theta_{72}^3 \\ \hline
1& 1 & 0 & 0 & 0 & 1 & 0 & 0 & 0 & 1 \\
 \alpha &1 & 0 & 0 & 0 & 0 & 1 & 0 & 1 & 0 \\
\rho& 0 & 1 & 1 & 1 & 1 & 2 & 1 & 2 & 1 \\
 \alpha \rho & 0 & 1 & 1 & 1 & 2 & 1 & 1 & 1 & 2 \\
 \rho \alpha & 0 & 1 & 1 & 1 & 2 & 1 & 1 & 1 & 2 \\
\alpha \rho \alpha& 0 & 1 & 1 & 1 & 1 & 2 & 1 & 2 & 1 \\
\pi & 1 & 2 & 2 & 2 & 4 & 3 & 2 & 3 & 4 \\
 \alpha \pi &1 & 2 & 2 & 2 & 3 & 4 & 2 & 4 & 3 \\
 \eta & 2 & 2 & 2 & 2 & 4 & 4 & 2 & 4 & 4 \\
\end{array}
\right)$\\

$\left(
\begin{array}{c|ccc|ccc|ccc}
& \multicolumn{2}{l}{\theta_{73}^1 \otimes \_} & & \multicolumn{2}{l}{\theta_{73}^2 \otimes \_}  & & \multicolumn{2}{l}{\theta_{73}^3 \otimes \_} & \\
N_{\mathcal{AH}_3}^{(7)} & \theta_{73}^1 & \theta_{73}^2 & \theta_{73}^3 & \theta_{73}^1 & \theta_{73}^2 & \theta_{73}^3 & \theta_{73}^1 & \theta_{73}^2 & \theta_{73}^3 \\ \hline
1 & 1 & 0 & 0 & 0 & 1 & 0 & 0 & 0 & 1 \\
\beta & 0 & 1 & 0 & 1 & 0 & 0 & 0 & 0 & 1 \\
\xi & 0 & 1 & 1 & 1 & 0 & 1 & 1 & 1 & 4 \\
\beta \xi & 1 & 0 & 1 & 0 & 1 & 1 & 1 & 1 & 4 \\
 \xi \beta &  1 & 0 & 1 & 0 & 1 & 1 & 1 & 1 & 4 \\
\beta \xi \beta & 0 & 1 & 1 & 1 & 0 & 1 & 1 & 1 & 4 \\
\mu & 1 & 0 & 2 & 0 & 1 & 2 & 2 & 2 & 7 \\
 \beta \mu &0 & 1 & 2 & 1 & 0 & 2 & 2 & 2 & 7 \\
\nu & 0 & 0 & 2 & 0 & 0 & 2 & 2 & 2 & 6 \\
\end{array}
\right)$

\newcommand{\urlprefix}{}
\bibliographystyle{alpha}
%Included for winedt:
%input "bibliography/bibliography.bib"
\bibliography{bibliography}

\newcommand{\noopsort}[1]{}\def\cprime{$'$} \def\cprime{$'$} \def\cprime{$'$}
\begin{thebibliography}{ENO10}

\bibitem[AG11]{MR2812458}
Marta Asaeda and Pinhas Grossman.
\newblock A quadrilateral in the {A}saeda-{H}aagerup category.
\newblock {\em Quantum Topol.}, 2(3):269--300, 2011.
\newblock \arxiv{1006.2524}, \mathscinet{MR2812458}.

\bibitem[AH99]{MR1686551}
Marta Asaeda and Uffe Haagerup.
\newblock Exotic subfactors of finite depth with {J}ones indices
  {$(5+\sqrt{13})/2$} and {$(5+\sqrt{17})/2$}.
\newblock {\em Comm. Math. Phys.}, 202(1):1--63, 1999.
\newblock \mathscinet{MR1686551} \doi{10.1007/s002200050574}
  \arxiv{math.OA/9803044}.

\bibitem[EG11]{MR2837122}
David~E. Evans and Terry Gannon.
\newblock The exoticness and realisability of twisted {H}aagerup-{I}zumi
  modular data.
\newblock {\em Comm. Math. Phys.}, 307(2):463--512, 2011.

\bibitem[ENO05]{MR2183279}
Pavel Etingof, Dmitri Nikshych, and Viktor Ostrik.
\newblock On fusion categories.
\newblock {\em Ann. of Math. (2)}, 162(2):581--642, 2005.
\newblock \arxiv{math/0203060}, \mathscinet{MR2183279}.

\bibitem[ENO10]{MR2677836}
Pavel Etingof, Dmitri Nikshych, and Victor Ostrik.
\newblock Fusion categories and homotopy theory.
\newblock {\em Quantum Topol.}, 1(3):209--273, 2010.
\newblock \arxiv{0909.3140}, \mathscinet{MR2677836}.

\bibitem[GI08]{MR2418197}
Pinhas Grossman and Masaki Izumi.
\newblock Classification of noncommuting quadrilaterals of factors.
\newblock {\em Internat. J. Math.}, 19(5):557--643, 2008.
\newblock \arxiv{0704.1121}, \mathscinet{MR2418197}.

\bibitem[GJS13]{GJSext}
Pinhas Grossman, David Jordan, and Noah Snyder.
\newblock Cyclic extensions of fusion categories via the brauer-picard
  groupoid.
\newblock {\em Quantum Topol., to appear}, 2013.

\bibitem[GS12]{MR2909758}
Pinhas Grossman and Noah Snyder.
\newblock Quantum subgroups of the {H}aagerup fusion categories.
\newblock {\em Comm. Math. Phys.}, 311(3):617--643, 2012.
\newblock \arxiv{1102.2631}, \mathscinet{MR2909758}.

\bibitem[GS14]{GSbp}
Pinhas Grossman and Noah Snyder.
\newblock The {B}rauer-{P}icard group of the {A}saeda-{H}aagerup fusion
  categories.
\newblock {\em Trans. Amer. Math. Soc., to appear}, 2014.

\bibitem[Izu93]{MR1228532}
Masaki Izumi.
\newblock Subalgebras of infinite {$C^*$}-algebras with finite {W}atatani
  indices. {I}. {C}untz algebras.
\newblock {\em Comm. Math. Phys.}, 155(1):157--182, 1993.

\bibitem[Izu00]{MR1782145}
Masaki Izumi.
\newblock The structure of sectors associated with {L}ongo-{R}ehren inclusions.
  {I}. {G}eneral theory.
\newblock {\em Comm. Math. Phys.}, 213(1):127--179, 2000.

\bibitem[Izu01]{MR1832764}
Masaki Izumi.
\newblock The structure of sectors associated with {L}ongo-{R}ehren inclusions.
  {II}. {E}xamples.
\newblock {\em Rev. Math. Phys.}, 13(5):603--674, 2001.

\bibitem[Izu15]{IzumiNote}
Masaki Izumi.
\newblock Notes on certain categories of endomorphisms (not yet published).
\newblock 2015.

\bibitem[Jon08]{MR2501843}
Vaughan F.~R. Jones.
\newblock Two subfactors and the algebraic decomposition of bimodules over
  {$\rm II\sb 1$} factors.
\newblock {\em Acta Math. Vietnam.}, 33(3):209--218, 2008.

\bibitem[Lon94]{MR1257245}
Roberto Longo.
\newblock A duality for {H}opf algebras and for subfactors. {I}.
\newblock {\em Comm. Math. Phys.}, 159(1):133--150, 1994.
\newblock \mathscinet{MR1257245}.

\bibitem[MS10]{1007.1730}
Scott Morrison and Noah Snyder.
\newblock Subfactors of index less than 5, part 1: the principal graph
  odometer, 2010.
\newblock \arxiv{1007.1730}.

\bibitem[MW14]{1404.3955}
Scott Morrison and Kevin Walker.
\newblock The centre of the extended haagerup subfactor has 22 simple objects,
  2014.

\bibitem[Nai07]{MR2362670}
Deepak Naidu.
\newblock Categorical {M}orita equivalence for group-theoretical categories.
\newblock {\em Comm. Algebra}, 35(11):3544--3565, 2007.

\bibitem[Ost03]{MR1976459}
Victor Ostrik.
\newblock Module categories, weak {H}opf algebras and modular invariants.
\newblock {\em Transform. Groups}, 8(2):177--206, 2003.
\newblock \mathscinet{MR1976459} \arxiv{0111139}.

\bibitem[Pet09]{0902.1294}
Emily Peters.
\newblock A planar algebra construction of the {H}aagerup subfactor, 2009.
\newblock \arxiv{0902.1294}, to appear in Internat. J. Math.

\bibitem[Pop90]{MR1055708}
Sorin Popa.
\newblock Classification of subfactors: the reduction to commuting squares.
\newblock {\em Invent. Math.}, 101(1):19--43, 1990.
\newblock \mathscinet{MR1055708} \doi{10.1007/BF01231494}.

\bibitem[Yam04]{MR2075605}
Shigeru Yamagami.
\newblock Frobenius algebras in tensor categories and bimodule extensions.
\newblock In {\em Galois theory, {H}opf algebras, and semiabelian categories},
  volume~43 of {\em Fields Inst. Commun.}, pages 551--570. Amer. Math. Soc.,
  Providence, RI, 2004.
\newblock \mathscinet{MR2075605}.

\end{thebibliography}

\end{document}